\documentclass[onefignum,onetabnum]{siamart250211}
\usepackage[T1]{fontenc}

\ifpdf
\hypersetup{
  pdftitle={Mixed VEM for the generalized Oseen  problem},
  pdfauthor={Felipe Lepe and  Gonzalo Rivera }
}
\fi

\def\CT{{\mathcal T}}

\usepackage{amsmath,amssymb,amsfonts,latexsym,stmaryrd}
\usepackage{graphicx,float} 
\usepackage{mathrsfs} 
\usepackage{setspace} 
\usepackage{lipsum}
\usepackage{graphicx,float}
\usepackage{color}
\usepackage{epstopdf}
\usepackage{multirow}
\usepackage{url}
\usepackage{diagbox}

\usepackage{xparse}
\usepackage{xcolor}

\def\ds{\displaystyle}

\def\O{\Omega}

\def\E{E}

\newcommand{\norm}[1]{\lVert#1\rVert}

\usepackage{booktabs}
\usepackage{float}

\newcommand{\vertiii}[1]{{\left\vert\kern-0.25ex\left\vert\kern-0.25ex\left\vert #1 
		\right\vert\kern-0.25ex\right\vert\kern-0.25ex\right\vert}}

\newcommand\bu{\boldsymbol{u}}
\newcommand\bv{\boldsymbol{v}}
\newcommand\bw{\boldsymbol{w}}

\newcommand\bF{\boldsymbol{f}}

\def\CT{{\mathcal T}}

\newcommand\bcI{\boldsymbol{\mathcal{I}}}

\newcommand{\dd}{\texttt{d}}

\newcommand\bsig{\boldsymbol{\sigma}}
\newcommand\btau{\boldsymbol{\tau}}

\newcommand\R{\mathbb{R}}

\renewcommand\H{\mathrm{H}}

\renewcommand\L{\mathrm{L}}

\renewcommand\O{\Omega}

\def\betab{\boldsymbol{\beta}}

\newcommand\bdiv{\mathop{\mathbf{div}}\nolimits}

\renewcommand\div{\mathop{\mathrm{div}}\nolimits}
\newcommand\rot{\mathop{\mathrm{rot}}\nolimits}
\newcommand\brot{\mathop{\mathbf{rot}}\nolimits}

\newcommand\td{\mathtt{d}}
\newcommand\tr{\mathop{\mathrm{tr}}\nolimits}

\newcommand\LO{\L^2(\O)}

\newsiamremark{remark}{Remark}
\newsiamremark{assumption}{Assumption}
\newsiamremark{hypothesis}{Hypothesis}
\crefname{hypothesis}{Hypothesis}{Hypotheses}
\newsiamthm{problem}{Problem}
\newsiamthm{claim}{Claim}

\headers{Mixed VEM for the Oseen problem}{Felipe Lepe and Gonzalo Rivera }

\title{A mixed virtual element discretization for  the generalized Oseen problem \thanks{Submitted to the editors DATE.
\funding{The first author was partially supported by
	DIUBB through project 2120173 GI/C Universidad del B\'io-B\'io (Chile).
	The second author was supported by  ANID-Chile through FONDECYT project 1231619 (Chile).  
}}}

\author{Felipe Lepe\thanks{GIMNAP-Departamento de Matem\'atica, Universidad del B\'io - B\'io, Casilla 5-C, Concepci\'on, Chile. \email{flepe@ubiobio.cl}.}
\and Gonzalo Rivera\thanks{Departamento de Ciencias Exactas,
	Universidad de Los Lagos, Casilla 933, Osorno, Chile. \email{gonzalo.rivera@ulagos.cl}.}}

\usepackage{amsopn}

\begin{document}

\maketitle

\begin{abstract}
In this paper we introduce a mixed virtual element method to approximate the solution for the two dimensional generalized Oseen problem. We introduce the pseudostress as an additional unknown, which allows to eliminate the pressure from the system; the pressure can be recovered via a post-process of the pseudostress tensor. We prove existence and uniqueness of the continuous solution via a fixed point argument.  Under standard mesh assumptions, we develop a virtual element method to approximate both the tensor and the velocity field, and we show  that it is stable. Furthermore, we provide a priori error estimates for the method and validate them through a series of numerical tests using different polygonal meshes.
\end{abstract}

\begin{keywords}
Oseen equations, virtual elements, a priori error estimates, polygonal meshes, mixed formulations.
\end{keywords}

\begin{AMS}
	  35Q35,  65N15, 65N30, 76D07
\end{AMS}
\section{Introduction}
\label{sec:intro}
The Oseen equations are an important flow model which is seen as a linearization of the Navier-Stokes equations in each step of time. The Oseen problem that we study in this paper is the following: 
let $\O \subset \R^2$ be an open and bounded domain with Lipschitz boundary $\partial\O$.  According to \cite{MR3561143}, the aim is to find the velocity field $\bu$ and the scalar pressure $p$ satisfying the following set of equations:
\begin{equation}
\label{eq:Oseen_system}
\left\{
\begin{aligned}
-\nu\Delta\bu + (\betab \cdot \nabla)\bu + \kappa\bu + \nabla p &= \bF && \text{in }\O,\\
\div\bu &= 0 && \text{in }\O,\\
\bu &= \mathbf{0} && \text{on }\partial\O,\\
\int_{\O}p &= 0.
\end{aligned}
\right.
\end{equation}	
where  $\boldsymbol{f}$ represents  an external force acting on the domain, the parameter $\nu>0$ denotes the kinematic viscosity, 
	$\boldsymbol{\beta}$ is a given vector field representing the steady-state velocity, 
	and $\kappa>0$ is the permeability coefficient of the Brinkman region. A more general boundary condition  as $\bu=\bu_D$, where $\bu_D\in \H^{1/2}(\partial\O)$ and satisfies the compatibility condition $\int_{\partial\O}\bu_D\cdot\boldsymbol{n}$, where  $\boldsymbol{n}$ is the outward unit normal vector on $\partial\O$  can be considered, but for simplicity, our theoretical work is focused on null Dirichlet boundary conditions.

An important aspect of the model \eqref{eq:Oseen_system} is that each term represents a problem by itself. Let us be precise on this. If we set $\kappa= 0$ in the free-flow subdomain, one recovers the classical Oseen system, where the properties of a porous media are not involved in the system. On the other hand,  
 if we take  $\kappa \gg 0$ and $\boldsymbol{\beta} = \boldsymbol{0}$,   \eqref{eq:Oseen_system} becomes the so-called  Brinkman equations whereas if $\kappa = 0$ and $\boldsymbol{\beta} = \boldsymbol{0}$ throughout the domain, the system reduces to the Stokes problem.



The velocity-pressure formulation of \eqref{eq:Oseen_system} is the simplest approach in which this system can be studied. However, it is possible to consider additional unknowns with physical interpretation as the stress, vorticity, stream functions, etc. that can be also considered in order to derive a mixed formulation of \eqref{eq:Oseen_system}. These variables and formulations, together with their approximations based on finite elements   are well discussed in \cite{MR851383}. However, our intention is to go beyond these formulations, introducing the so called pseudostress tensor. This unknown, that relates the gradient of the velocity and the pressure of the fluid was introduced, to the best of our knowledge, in \cite{cai2010} and from this point and onwards, several methods and mixed formulations have emerged with the introduction of this unknown. Let us refer to the following non-exhaustive list of papers where mixed formulations based in the pseudostress tensor are considered \cite{MR3629152,MR2594823,MR2878511,MR4627698,MR4805930,MR4789346,MR4890763}. Of course we encourage  the reader to read the references therein each of the aforementioned articles. 

We now aim to study the numerical approximation of \eqref{eq:Oseen_system} using the Virtual Element Method (VEM). Introduced in \cite{MR2997471}, the VEM has proven over  the last years to be a suitable and flexible alternative for the numerical analysis of partial differential equations. Several applications of this method to problems in continuum mechanics and related areas can be found in \cite{MR4510898}.  In particular,  for linear and non-linear flow problems we mention \cite{MR3164557,MR4673339,MR3614887,MR3895873,MR3803834,MR4300149} and more precisely for  the generalized Oseen problem \eqref{eq:Oseen_system} and its variants, we refer the reader to  \cite{MR4671455, MR4619470, MR4369815, MR5004638, MR4419353} showing the relevance of the  VEM for fluid mechanics problems which is in constant progress.



We now specify the main contributions of this work. First, our aim is to consider a mixed formulation of \eqref{eq:Oseen_system} via the aforementioned pseudostress tensor. This unknown leads to a mixed formulation where tensorial versions of the VEM spaces to approximate the $\H(\div)$ space are needed. These types of spaces are available in the works \cite{MR3614887,MR3629152}  for the Stokes and Brinkman flow problems. Now, our goal is to  implement these spaces for \eqref{eq:Oseen_system} where the convective term plays a new role in the analysis. First of all, the presence of the convective term does not ensure the coercivity of the bilinear form associated to the  deviator terms. Hence, the strategy changes, and a fixed point argument is needed to derive existence and uniqueness of solution at the continuous level, implying the assumption of small data for the problem. This has been needed also for the analysis of the eigenvalue problem \cite{MR4983488} in order to establish that the solution operators associated to the source problems are  well defined. The difference now is that when the saddle point is presented, in the case of  \eqref{eq:Oseen_system} a regular perturbed saddle point problem as the presented in \cite{MR3097958} appears, involving additional data that must be considered in order to satisfy the contraction that the fixed point strategy requires. When the VEM is introduced,  the stability of the mixed formulation at discrete level needs suitable arguments because of the perturbed saddle point problem. Also, we prove uniqueness of solution, convergence,   and error estimates tracking carefully the constants of the problem. The error estimates that we provide are optimal for the VEM that we are studying for each of the variables involved in the matrix system, including an a priori error estimate for the pressure, that we obtain using the error estimates of the velocity and pseudostress.  

\subsection{Outline} The paper is organized as follows: in Section \ref{sec:model} we introduce the mixed formulation of \eqref{eq:Oseen_system} and  the functional framework in which our analysis will be performed. We prove the existence and uniqueness of solutions with a fixed point argument which requires a smallness condition on  the data. In Section \ref{secVEM} we introduce the mixed virtual element method, where the basic assumptions on the meshes, degrees of freedom, and local and global spaces are introduced. Also, the VEM version of the bilinear forms is introduced and stability of the discrete schemes is proved, together with the uniqueness of discrete solutions, once again, under a small data assumption. Section \ref{sec:error} is dedicated to the error analysis and finally in Section \ref{sec:numerics} we report a series of numerical tests to assess the performance of the method.

\subsection{Preliminary notations and definitions}
Throughout our work, we will use standard notations for Lebesgue and Sobolev spaces, particularly we will use the classic notations $\L^2(\O)$ and $\H^1(\O)$ where these Hilbert spaces are endowed with the natural norms that we denote by $\|\cdot\|_{0,\O}$ and $\|\cdot\|_{1,\O}$, respectively. We also denote by $|\cdot|_{1,\O}$ the seminorm on $\H^1(\O)$. If $\bv\in\mathbb{R}^2$, we define the gradient and divergence operators in the classic fashion, i.e. 
$\nabla\bv:=(\partial v_i/\partial x_j)$, for $i,j=1,2$ and $\div\bv:=\sum_{j=1}^2(\partial v_j/\partial x_j)$. We denote by $\mathbf{\H}(\div;\O)$ and $\mathbb{H}(\bdiv;\O)$ the vectorial and tensorial version of the $\H(\div)$ spaces, which are endowed with the standard norms that we denote by $\|\cdot\|_{\div,\O}$ and $\|\cdot\|_{\bdiv,\O}$, respectively. Also, for a tensor $\btau\in\mathbb{R}^{2\times 2}$ we define its deviator by $\btau^{d}:=\btau-2^{-1}\tr(\btau)\mathbb{I}$
where $\mathbb{I}\in\mathbb{R}^{2\times 2}$ is the identity tensor.

\section{The continuous problem}
\label{sec:model}
The aim now is to introduce the mixed formulation for  \eqref{eq:Oseen_system}. From now and on, we assume that the steady-state velocity in \eqref{eq:Oseen_system} is such that $\boldsymbol{\beta}\in\mathbf{W}^{1,\infty}(\O)$. Let us introduce the pseudostress tensor as follows
\begin{equation}
\label{eq:pseudostress}
\bsig:=\nu\nabla\bu-(\bu\otimes\boldsymbol{\beta})-p\mathbb{I}\quad\text{in}\,\O.
\end{equation}
From these definitions, we obtain directly that $-\bdiv\bsig+\kappa\bu=\bF$ in $\O$. On the other hand, if we apply 
the trace operator in \eqref{eq:pseudostress} we obtain the following identity for the pressure
\begin{equation}
\label{eq:continuous_pressure}
p=-\frac{1}{2}(\tr(\bsig)+\tr(\bu\otimes\boldsymbol{\beta}))\quad\text{in}\,\,\O.
\end{equation}
Moreover, due to the condition $(p,1)_{0,\O}=0$ we deduce that
\begin{equation*}
(\tr(\bsig),1)_{0,\O}=-(\tr(\bu\otimes\boldsymbol{\beta}),1)_{0,\O}\quad\text{in}\,\,\O.
\end{equation*}
Finally, it is easy to check from  the incompressibility condition the following identity  $\bsig^{\texttt{d}}=\nu\nabla\bu-(\bu\otimes\boldsymbol{\beta})^{\texttt{d}}$ in $\O$. 
Hence, system  \eqref{eq:Oseen_system} can be equivalently written as follows:
\begin{align*}
	\bsig^{\texttt{d}}-\nu\nabla\bu+(\bu\otimes\boldsymbol{\beta})^{\texttt{d}} & =  \boldsymbol{0} & \text{in $\Omega$},\\
	\bdiv\bsig-\kappa\bu &= -\bF & \text{in $\Omega$},\\
	\bu &= \boldsymbol{0}& \text{on $\partial\O$}\\
	(\tr(\bsig),1)_{0,\O}+(\tr(\bu\otimes\boldsymbol{\beta}),1)_{0,\O}&=0& \text{in}\,\Omega.
\end{align*}

To simplify the presentation of the material, we define the spaces $\mathbb{H}:=\mathbb{H}(\bdiv;\O)$ and $\mathbf{Q}:=\L^2(\O)^2$, which we endow with the respective standard norms. With these spaces at hand, a variational formulation of problem above reads as follows. Given $\bF\in\mathbf{Q}$, find $(\bsig,\bu)\in\mathbb{H}\times\mathbf{Q}$ such that
\begin{equation}
\label{eq:variational1}
	\left\{
	\begin{array}{rcll}
\displaystyle\frac{1}{\nu}\int_{\O}\bsig^{\texttt{d}}:\btau^{\texttt{d}}+\frac{1}{\nu}(\bu\otimes\boldsymbol{\beta})^{\texttt{d}}: \btau+\int_{\O}\bu\cdot\bdiv\btau&=&0&\forall \btau\in \mathbb{H},\\

\displaystyle\int_{\O}\bdiv\bsig\cdot\bv-\int_{\O}\kappa\bu\cdot\bv&=&-\displaystyle\int_{\O}\bF\cdot\bv&\forall \bv\in \mathbf{Q}.
\end{array}
	\right.
\end{equation}

Now we introduce the continuous bilinear forms $a:\mathbb{H}\times\mathbb{H}\rightarrow\mathbb{R}$, $b:\mathbb{H}\times\mathbf{Q}\rightarrow\mathbb{R}$, $c:\mathbf{Q}\times\mathbb{H}\rightarrow\mathbb{R}$, and $d:\mathbf{Q}\times\mathbf{Q}\rightarrow\mathbb{R}$, which are defined as follows
\begin{align*}
a(\boldsymbol{\rho},\btau)&:=\dfrac{1}{\nu}\int_\O\boldsymbol{\rho}^{\dd}:\btau^{\dd}\qquad &\boldsymbol{\rho},\btau\in\mathbb{H},\\
b(\btau,\bv)&:=\int_\O\bdiv(\btau)\cdot\bv\qquad &(\btau,\bv)\in \mathbb{H}\times\mathbf{Q},\\
c(\bv,\btau)&:=\dfrac{1}{\nu}\int_\O(\bv\otimes\boldsymbol{\beta})^{\dd}:\btau\qquad &(\bv,\btau)\in \mathbf{Q}\times\mathbb{H},\\
d(\bv,\bw)&:=\int_\O\kappa\bv\cdot\bw\qquad &\bv,\bw\in \mathbf{Q}.
\end{align*}

We also introduce the functional $F:\mathbf{Q}	\rightarrow\mathbb{R}$ defined by $F(\bv):=-(\bF,\bv)_{0,\O}$. Hence, with all these definitions at hand, problem \eqref{eq:variational1} can be written  as follows: Find $(\bsig,\bu)\in\mathbb{H}\times\mathbf{Q}$ such that
\begin{equation}
\label{eq:variational1_1}
	\left\{
	\begin{array}{rcll}
a(\bsig,\btau)+b(\btau,\bu)+c(\bu,\btau)&=&0&\forall \btau\in \mathbb{H},\\

b(\bsig,\bv)-d(\bv,\bu)&=&F(\bv)&\forall \bv\in \mathbf{Q}.
\end{array}
	\right.
\end{equation}

For the analysis, it is convenient to decompose $\mathbb{H}$ in a suitable way. Let us define the space $\mathbb{H}_0:=\{\btau\in\mathbb{H}\,:\,(\tr(\btau),1)_{0,\O}=0\}$. Hence, it is possible to prove that $\mathbb{H}:=\mathbb{H}_0\oplus \mathbb{R}\mathbb{I}$. More precisely, if $\btau\in\mathbb{H}$, there exist $\btau_0\in\mathbb{H}_0$ and $d\in\mathbb{R}$ such that $\btau=\btau_0+d\mathbb{I}$, where $d=(2|\O|)^{-1}(\tr(\btau),1)_{0,\O}$. Hence, problem \eqref{eq:variational1_1} now can be stated as follows: Find $(\bsig,\bu)\in\mathbb{H}_0\times\mathbf{Q}$ such that
\begin{equation}
\label{eq:variational1_H0}
	\left\{
	\begin{array}{rcll}
a(\bsig,\btau)+b(\btau,\bu)+c(\bu,\btau)&=&0&\forall \btau\in \mathbb{H}_0,\\
b(\bsig,\bv)-d(\bv,\bu)&=&F(\bv)&\forall \bv\in \mathbf{Q}.
\end{array}
	\right.
\end{equation}

Let us introduce  the space $\mathbb{K}$ as the kernel of bilinear form $b(\cdot,\cdot)$, which is defined as follows
\begin{equation*}
\mathbb{K}:=\{\btau\in\mathbb{H}_0\,:\, b(\btau,\bv)=0\quad\forall\bv\in\mathbf{Q}\}=\{\btau\in\mathbb{H}_0\,:\, \bdiv\btau=0\quad\text{in}\,\O\}.
\end{equation*}
Also, we recall the following technical result (see \cite[Chapter 9, Proposition 9.1.1]{MR3097958}).
\begin{lemma}\label{lmm:cota}
There exists a constant $C_1>0$ such that
\begin{equation*}
C_1\|\btau\|_{0,\O}^2\leq \|\btau^{\dd}\|_{0,\O}^2+\|\bdiv\btau\|_{0,\O}^2\qquad\forall \btau\in \mathbb{H}_0.
\end{equation*}
\end{lemma}
Then, Lemma \ref{lmm:cota} allows us to prove the $\mathbb{K}$-coercivity of $a(\cdot,\cdot)$. Indeed, for $\btau\in \mathbb{K}$ we have 
\begin{equation*}
a(\btau,\btau)=\frac{1}{\nu}\|\btau^{\texttt{d}}\|_{0,\O}^2
\geq \frac{C_1}{\nu}\|\btau\|_{0,\O}^2=\frac{C_1}{\nu}\|\btau\|_{\bdiv,\O}^2.
\end{equation*}

On the other hand, we have  that  $b(\cdot,\cdot)$ satisfies the following inf-sup condition
\begin{equation}\label{eq:inf-sup_cont}
\displaystyle\sup_{\btau\in\mathbb{H}_0}\frac{b(\btau,\boldsymbol{q})}{\|\btau\|_{\bdiv,\O}}\geq\gamma\|\boldsymbol{q}\|_{0,\O}\quad\forall \boldsymbol{q}\in\mathbf{Q}.
\end{equation}

Finally, for all $\bv\in\mathbf{Q}$,  bilinear form $d(\cdot,\cdot)$ is such that
\begin{equation}\label{eq:c_coercive}
d(\bv,\bv)=\int_{\O}\kappa\bv\cdot\bv= \kappa\|\bv\|_{0,\O}^2\quad\forall\bv\in\mathbf{Q}.
\end{equation}

For the analysis of \eqref{eq:variational1_H0} it is necessary to perform a fixed point argument. With this in mind, let us consider the following    problem: Find $(\bsig,\bu)\in\mathbb{H}_0\times\mathbf{Q}$ such that 
\begin{equation}\label{def:oseen_system_source_00}
	\left\{
	\begin{array}{rcll}
a(\bsig,\btau)+ b(\btau,\bu)&=&\mathcal{G}_{\boldsymbol{w},\boldsymbol{\beta}}(\btau)&\forall \btau\in \mathbb{H}_0,\\
 b(\bsig, \bv)-d(\bv,\bw)&=&F(\bv)&\forall \bv\in \mathbf{Q}.
\end{array}
	\right.
\end{equation}
where $\mathcal{G}_{\boldsymbol{w},\boldsymbol{\beta}}:\mathbb{H}_0\rightarrow\mathbb{R}$ is the  functional defined by 
$$
\mathcal{G}_{\boldsymbol{w},\boldsymbol{\beta}}(\btau):=-\displaystyle\frac{1}{\nu}\int_{\O}(\boldsymbol{w}\otimes\boldsymbol{\beta})^{\texttt{d}}:\btau \,\,\,\forall\btau\in\mathbb{H}_0.
$$ 
It is easy to check that these functionals $\mathcal{G}_{\boldsymbol{w},\boldsymbol{\beta}}$ and $F$ are bounded, i.e.
$$
\frac{|\mathcal{G}_{\boldsymbol{w},\boldsymbol{\beta}}(\btau)|}{\|\btau\|_{\bdiv,\O}}\leq\frac{2+\sqrt{2}}{2\nu}\|\boldsymbol{w}\|_{0,\O}\|\boldsymbol{\beta}\|_{\infty,\O}\,\,\,\,\text{and}\,\,\,\,\,\frac{|F(\bv)|}{\|\bv\|_{0,\O}}\leq \|\boldsymbol{f}\|_{0,\O}.
$$
Since  $a(\cdot,\cdot)$, $b(\cdot,\cdot)$, and $d(\cdot,\cdot)$ are continuous,  the following estimates hold
$$
|a(\bsig,\btau)|\leq\dfrac{(2+\sqrt{2})^2}{4\nu}\|\bsig\|_{\bdiv,\O}\|\btau\|_{\bdiv,\O}\,\,\,\,\,\,\,\,\,|b(\bsig,\bv)|\leq\|\bsig\|_{\bdiv,\O}\|\bv\|_{0,\O},
$$
and $|d(\bv,\bw)|\leq \kappa\|\bv\|_{0,\O}\|\bw\|_{0,\O}$. As a consequence of the previous properties and  according to \cite[Theorem 4.3.2]{MR3097958}, the solution of \eqref{def:oseen_system_source_00} satisfies the following estimates of data dependence
\begin{equation}
\label{eq:bound_sigma}
\|\bsig\|_{\bdiv,\O}\leq \frac{\nu\gamma^2+\kappa(2+\sqrt{2})^2}{C_1\gamma^2}\|\mathcal{G}_{\boldsymbol{w},\boldsymbol{\beta}}\|_{\mathbb{H}(\bdiv,\O)'}+\frac{2+\sqrt{2}}{\sqrt{C_1}\gamma}\|F\|_{0,\O},
\end{equation}
and
\begin{equation}
\label{eq:bound_u}
\|\bu\|_{0,\O}\leq \frac{2+\sqrt{2}}{\sqrt{C_1}\gamma}\|\mathcal{G}_{\boldsymbol{w},\boldsymbol{\beta}}\|_{\mathbb{H}(\bdiv,\O)'}+\frac{(2+\sqrt{2})^2}{\kappa\dfrac{(2+\sqrt{2})^2}{4}+2\nu\gamma^2}\|F\|_{0,\O}.
\end{equation}

The next step is to conclude the existence and uniqueness of solutions for the problem. To do this task,  we use a fixed point strategy. Let us introduce the following ball
\begin{equation}
\label{eq:ball}
M_{\textrm{R}_0}:=\{\bv\in\mathbf{Q}\,:\, \|\bv\|_{0,\O}\leq \textrm{R}_0\},
\end{equation}
where $\textrm{R}_0>0$. Also, we define the following operator
\begin{align*}
\mathcal{K}&:\mathbf{Q}\rightarrow\mathbf{Q}\times\mathbb{H}_0,\\
&\bw\mapsto \mathcal{K}(\bw):=(\mathcal{K}_1(\bw),\mathcal{K}_2(\bw))=(\bu,\bsig).
\end{align*}

Let $\bw\in M_{\textrm{R}_0}$. From the definition of $\mathcal{K}_1$, \eqref{eq:bound_u}, and  \eqref{eq:ball}, we have 
\begin{multline*}
\|\mathcal{K}_1(\bw)\|_{0,\O}=\|\bu\|_{0,\O}\leq \frac{2+\sqrt{2}}{\sqrt{C_1}\gamma}\|\mathcal{G}_{\boldsymbol{w},\boldsymbol{\beta}}\|_{\mathbb{H}(\bdiv,\O)'}+\frac{(2+\sqrt{2})^2}{\kappa\dfrac{(2+\sqrt{2})^2}{4}+2\nu\gamma^2}\|F\|_{0,\O}\\ 
\leq  \frac{(2+\sqrt{2})^2}{4\nu\sqrt{C_1}\gamma}(\|\bw\|_{0,\O}^2+\|\boldsymbol{\beta}\|_{\infty,\O}^2)+\frac{(2+\sqrt{2})^2}{\kappa\dfrac{(2+\sqrt{2})^2}{4}+2\nu\gamma^2}\|\bF\|_{0,\O}\\
\leq  \frac{(2+\sqrt{2})^2}{4\nu\sqrt{C_1}\gamma}(\textrm{R}_0^2+\|\boldsymbol{\beta}\|_{\infty,\O}^2)+\frac{(2+\sqrt{2})^2}{\kappa\dfrac{(2+\sqrt{2})^2}{4}+2\nu\gamma^2}\|\bF\|_{0,\O}.
\end{multline*}
Now, defining $\widehat{C}:=\displaystyle\frac{(2+\sqrt{2})^2}{4\nu\sqrt{C_1}\gamma}$ and $H:=\widehat{C}\|\boldsymbol{\beta}\|_{\infty,\O}^2+\displaystyle\frac{(2+\sqrt{2})^2}{\kappa\dfrac{(2+\sqrt{2})^2}{4}+2\nu\gamma^2}\|\bF\|_{0,\O}$, we have 
\begin{equation*}
\widehat{C}\textrm{R}_0^2+H\leq\textrm{R}_0\Longrightarrow\widehat{C}\textrm{R}_0^2-\textrm{R}_0+H\leq 0,
\end{equation*}
which implies that the radius $\textrm{R}_0$ must satisfy
\begin{equation*}
\displaystyle\frac{1-\sqrt{1-4\widehat{C}H}}{2\widehat{C}}<\textrm{R}_0<\frac{1-\sqrt{1+4\widehat{C}H}}{2\widehat{C}}.
\end{equation*}
Additionally, for $\textrm{R}_0$ to be well defined, the condition $1-4\widehat{C}H> 0$ must hold. Implying
that 
\begin{equation*}
\frac{(2+\sqrt{2})^2}{4\nu\sqrt{C_1}\gamma}\left(\frac{(2+\sqrt{2})^2}{4\nu\sqrt{C_1}\gamma}\|\boldsymbol{\beta}\|_{\infty,\O}+\frac{(2+\sqrt{2})^2}{\kappa\dfrac{(2+\sqrt{2})^2}{4}+2\nu\gamma^2}\|\bF\|_{0,\O} \right)<\frac{1}{4}.
\end{equation*}
Hence, this leads to the fact that $\mathcal{K}_1(M_{\textrm{R}_0})\subset M_{\textrm{R}_0}$.

Now we prove that  $\mathcal{K}_1$ is a contraction.
\begin{lemma}\label{lmm:cotaj}
There exists a positive constant $L$, depending only on data, such that 
$$\|\mathcal{K}_1(\bw_1)-\mathcal{K}_1(\bw_2)\|_{0,\O}\leq L\|\bw_1-\bw_2\|_{0,\O},\qquad \forall\bw_1,\bw_2\in M_{R_0}.$$ 
\end{lemma}
\begin{proof}
Let $\bw_1,\bw_2\in M_{R_0}$, such that $\mathcal{K}(\bw_1):=(\bu_1,\bsig_1)$ and  $\mathcal{K}(\bw_2):=(\bu_2,\bsig_2)$ are solutions to the following respective problems:
$$
	\left\{
	\begin{array}{rcll}
a(\bsig_1,\btau)+ b(\btau,\bu_1)&=&-\displaystyle\frac{1}{\nu}\int_{\O}(\boldsymbol{w}_1\otimes\boldsymbol{\beta})^{\texttt{d}}:\btau&\forall \btau\in \mathbb{H}_0,\\
 b(\bsig_1, \bv)-d(\bv,\bu_1)&=&-(\boldsymbol{f},\bv)_{0,\O}&\forall \bv\in \mathbf{Q},
\end{array}
	\right.
$$
$$
	\left\{
	\begin{array}{rcll}
a(\bsig_2,\btau)+ b(\btau,\bu_2)&=&-\displaystyle\frac{1}{\nu}\int_{\O}(\boldsymbol{w}_2\otimes\boldsymbol{\beta})^{\texttt{d}}:\btau&\forall \btau\in \mathbb{H}_0,\\
 b(\bsig_2, \bv)-d(\bv,\bu_2)&=&-(\boldsymbol{f},\bv)_{0,\O}&\forall \bv\in \mathbf{Q}.
\end{array}
	\right.
$$
Subtracting these problems we obtain
\begin{equation}
\label{eq:contraction1}
\left\{\begin{aligned}
a(\bsig_1-\bsig_2,\btau)+ b(\btau,\bu_1-\bu_2)
  &= -\frac{1}{\nu}\int_{\O}(\boldsymbol{w}_1\otimes\boldsymbol{\beta})^{\texttt{d}}:\btau
     +\frac{1}{\nu}\int_{\O}(\boldsymbol{w}_2\otimes\boldsymbol{\beta})^{\texttt{d}}:\btau,
\\
&\qquad\hspace{5 cm} \forall \btau\in \mathbb{H}_0,
\\
b(\bsig_1-\bsig_2, \bv)-d(\bv,\bu_1-\bu_2)
  &= 0 \qquad\forall \bv\in \mathbf{Q}.
\end{aligned}\right.
\end{equation}
Now, testing \eqref{eq:contraction1} with $\btau=\bsig_1-\bsig_2$ and  $\bv=-\kappa\bdiv(\bsig_1-\bsig_2)$ and adding both equations we have 
\begin{multline*}
a(\bsig_1-\bsig_2,\bsig_1-\bsig_2)+b(\bsig_1-\bsig_2,\kappa\bdiv(\bsig_1-\bsig_2))\\
=-\displaystyle\frac{1}{\nu}\int_{\O}(\boldsymbol{w}_1\otimes\boldsymbol{\beta})^{\texttt{d}}:\btau+\displaystyle\frac{1}{\nu}\int_{\O}(\boldsymbol{w}_2\otimes\boldsymbol{\beta})^{\texttt{d}}:\btau,
\end{multline*}
and hence
\begin{equation*}
\frac{1}{\nu}\|(\bsig_1-\bsig_2)^{\texttt{d}}\|_{0,\O}^2+\kappa\|\bdiv(\bsig_1-\bsig_2)\|_{0,\O}^2\leq \frac{2+\sqrt{2}}{2\nu}\|\bw_1-\bw_2\|_{0,\O}\|\bsig_1-\bsig_2\|_{\bdiv,\O}\|\boldsymbol{\beta}\|_{\infty,\O}.
\end{equation*}
Now, invoking Lemma \ref{lmm:cota}, together with elementary algebraic manipulations,  yield 
\begin{equation*}
\frac{1}{2}\min\left\{\frac{1}{\nu},\kappa\right\}\min\{C_1,1\}\|\bsig_1-\bsig_2\|_{\bdiv,\O}\leq\frac{(2+\sqrt{2})\|\boldsymbol{\beta}\|_{\infty,\O}}{2\nu}\|\bw_1-\bw_2\|_{0,\O},
\end{equation*}
which implies 
\begin{equation}
\label{eq:contraction2}
\|\bsig_1-\bsig_2\|_{\bdiv,\O}\leq\frac{(2+\sqrt{2})\|\boldsymbol{\beta}\|_{\infty,\O}}{\nu}\left(\min\left\{\nu^{-1},\kappa\right\}\min\{C_1,1\}\right)^{-1} \|\bw_1-\bw_2\|_{0,\O}.
\end{equation}

On the other hand, invoking the  inf-sup condition \eqref{eq:inf-sup_cont} and \eqref{eq:contraction2} we have
\begin{align*}
\gamma\|\mathcal{K}_1(\bw_1)-\mathcal{K}_1(\bw_2)\|_{0,\O}&=\gamma\|\bu_1-\bu_2\|_{0,\O}\leq \sup_{\btau\in\mathbb{H}_0}\frac{b(\btau,\bu_1-\bu_2)}{\|\btau_h\|_{\bdiv,\O}}\\
&\leq \sup_{\btau\in\mathbb{H}_0}\frac{\mathcal{G}_{\boldsymbol{\bw_1},\boldsymbol{\beta}}(\btau)-\mathcal{G}_{\boldsymbol{\bw_2},\boldsymbol{\beta}}(\btau)-a(\bsig_1-\bsig_2,\btau)}{\|\btau_h\|_{\bdiv,\O}}\\
	&\leq \sup_{\btau\in\mathbb{H}_0}\frac{\mathcal{G}_{\boldsymbol{\bw_1},\boldsymbol{\beta}}(\btau)-\mathcal{G}_{\boldsymbol{\bw_2},\boldsymbol{\beta}}(\btau)}{\|\btau_h\|_{\bdiv,\O}}+\sup_{\btau\in\mathbb{H}_0}\frac{a(\bsig_1-\bsig_2,\btau)}{\|\btau_h\|_{\bdiv,\O}}\\ 
	&\leq \dfrac{2+\sqrt{2}}{2\nu}\|\bw_1-\bw_2\|_{0,\O}\|\boldsymbol{\beta}\|_{\infty,\O}\\
	&+\frac{1}{\nu}\left(\frac{(2+\sqrt{2})^2}{4}\right)\|\bsig_1-\bsig_2\|_{\bdiv,\O}\leq L_1\|\bw_1-\bw_2\|_{0,\O},
\end{align*}
where 
$$L_1:=\frac{\|\boldsymbol{\beta}\|_{\infty,\O}(2+\sqrt{2})}{2\nu}\left(1+\frac{(2+\sqrt{2})^2}{2\nu}(\min\{\nu^{-1},\kappa\}\min\{C_1,1\})^{-1} \right).$$
Therefore
$$\|\mathcal{K}_1(\bw_1)-\mathcal{K}_1(\bw_2)\|_{0,\O}\leq \frac{L_1}{\gamma}\|\bw_1-\bw_2\|_{0,\O}.$$
Finally, defining 
$L:=\gamma^{-1}L_1$, we conclude the proof.
\end{proof}

Now we are  in position to establish that $\mathcal{J}$ is a contraction.
\begin{theorem}
Given $\boldsymbol{f}\in\mathbf{Q}$, assume  that the data is sufficiently small. Also, assume that 
\begin{equation}
\label{eq_contraction}
\frac{\|\boldsymbol{\beta}\|_{\infty,\O}(2+\sqrt{2})}{2\nu}\left(1+\frac{(2+\sqrt{2})^2}{2\nu}(\min\{\nu^{-1},\kappa\}\min\{C_1,1\})^{-1} \right)< \gamma.
\end{equation}
Then, there exists a unique solution of problem \eqref{eq:variational1_H0}.
\end{theorem}
\begin{proof}
The result is a direct consequence of the well definition of $\mathcal{J}$  together with  $\mathcal{J}_1(M_{R_0})\subset M_{R_0}$, Lemma \ref{lmm:cotaj}, and the fact that \eqref{eq_contraction} gives that $\mathcal{J}$ is a contraction mapping.
\end{proof}
\section{The Virtual Element Approximation}
\label{secVEM}

The aim of this section is to propose and analyze a mixed virtual element method to approximate the solution of problem \eqref{eq:variational1_H0}. To this end, we first introduce the assumptions and definitions required 
to work within the virtual element framework~~\cite{MR2997471,MR4586821,MR3264352}.

We begin by describing the geometrical properties to operate under the VEM approach.
Let $\{\mathcal{T}_h(\O)\}_{h>0}$ be a sequence of decompositions of $\O$ into non-overlapping 
polygonal elements $\E$. We suppose that for each $h>0$, the tessellation $\mathcal{T}_h(\O)$ is built according with the procedure described below.

The polygonal domain $\Omega$ is partitioned
into a polygonal mesh $\mathcal{T}_h$ that is \emph{regular},
in the sense that there exist positive constants 
$ c, \eta$ such that
\begin{enumerate}
\item each edge $e \in \partial \E$ has a length $h_e \ge c \: h_E$,
where $h_E$ denotes the diameter of $E$; 
\item each polygon $E$ in the mesh is star-shaped with respect to a ball of radius $\eta h_E$.
\end{enumerate}

\subsection{Virtual element spaces and their degrees of freedom}

In what follows we will introduce the  virtual element spaces to approximate the pseudostress $\bsig$ and the velocity  $\bu$ of problem~\eqref{eq:variational1_H0}.  Let us summarize the  construction provided in  \cite[Subsection 3.2]{MR3629152} of the VEM spaces that we require. For each integer $k\geq 0$ and for each $E\in\mathcal{T}_h$, we introduce the following local virtual element space of order $k$:
\begin{multline}
\nonumber \mathbf{W}_h^E:=\{ \btau:=(\tau_1,\tau_2)^{\texttt{t}}\in\H(\div;E)\cap\H(\rot;E):\,\btau\cdot\boldsymbol{n}|_e\in \textsc{P}_k(e)\quad\forall e\subset\partial E,\\
 \quad\div\btau\in \textsc{P}_{k}(E), \quad\text{rot}\,\btau\in \textsc{P}_{k-1}(E) \},
\end{multline}
where $\text{rot}\,\btau:=\dfrac{\partial\tau_2}{\partial x_1}-\dfrac{\partial\tau_1}{\partial x_2}$ and $\textsc{P}_{-1}(E)=\{0\}$. 

In addition, given a  geometric object $\mathcal{O}$ of dimension $d\in\{1, 2\}$, as an edge or an element, with barycenter $x_{\mathcal{O}}$ and diameter  $h_{\mathcal{O}}$, we consider the following set of normalized monomials on $\mathcal{O}$ (of dimension $k+1$ for $d=1$ and $(k+1)(k+2)/2$ for $d=2$)
\begin{equation*}
\label{eq:monomio}
\mathcal{M}_{k}(\mathcal{O}):=\left\{ q \;\ \Big| \;\ q:=\left(\dfrac{\mathbf{x}-\mathbf{x}_{\mathcal{O}}}{h_{\mathcal{O}}}\right)^{\boldsymbol{\alpha}} \text{ for } \alpha\in \mathbb{N}^{d} \text{ with } |\boldsymbol{\alpha}|\leq k\right\},
\end{equation*}
where, for a multi-index $\boldsymbol{\alpha}:=(\alpha_{1},...,\alpha_{d})$, we denote, as usual, $\boldsymbol{\alpha}:=\alpha_{1}+...+\alpha_{d}$ and $\mathbf{x}^{\boldsymbol{\alpha}}:=x_{1}^{\alpha_{1}}...x_{d}^{\alpha_{d}}$. It is easy to
check that $\mathcal{M}_{k}(\mathcal{O})$ is a  basis of $\textsc{P}_k(\mathcal{O})$ (see \cite{MR2997471} for more details). 

Now, given $\btau\in \mathbf{W}_h^E$ we define the following degrees of freedom
\begin{align}
\ds\int_e\btau\cdot\boldsymbol{n} q\qquad&\forall q\in\mathcal{M}_k(e)\quad\forall\,\text{edge}\,\,e\in\mathcal{T}_h,\label{eq:dof_normal0}\\
\int_E\btau\cdot\nabla q\qquad&\forall q\in\mathcal{M}_{k}(E)\backslash\{1\}\quad\forall E\in\mathcal{T}_h,\label{eq:dof_grad0}\\
\int_E\btau\cdot\boldsymbol{q}\qquad&\forall\boldsymbol{q}\in\mathcal{H}_{k}^{\bot}(E)\quad\forall E\in\mathcal{T}_h,\label{eq:dof_rot0} 
\end{align}
where $\mathcal{H}_k^{\bot}$ is a basis for $(\nabla\textsc{P}_{k+1}(E))^{\bot|_{\textbf{\textsc{P}}_k(E)}}\cap\textbf{\textsc{P}}_k(E)$, which is the $\textbf{\textsc{L}}^2$-orthogonal of $\nabla\textsc{P}_{k+1}(E)$ in $\textbf{\textsc{P}}_k(E)$. A complete description of the details and properties of these spaces can be found in \cite[Subsection 3.2]{MR3629152}. 


We  now  introduce   for each $E\in\mathcal{T}_h$ the tensorial local virtual element space
\begin{equation}
\label{eq:global_space}
\mathbb{H}_{h}^{E}:=\{\btau\in\mathbb{H}(\bdiv;E)\cap\mathbb{H}(\brot;E) : (\tau_{i1},\tau_{i2})^{t}\in \mathbf{W}_{h}^{E}\quad \forall i\in\{1,2\}\},
\end{equation}
which is unisolvent with respect to the following degrees of freedom:
\begin{align}
\ds\int_e\btau\boldsymbol{n}\cdot\boldsymbol{q}\qquad&\forall\boldsymbol{q}\in\boldsymbol{\mathcal{M}}_k(e)\quad\forall\,\text{edge}\,\,e\in\mathcal{T}_h,\label{eq:dof_normal}\\
\int_E\btau:\nabla\boldsymbol{q}\qquad&\forall\boldsymbol{q}\in\boldsymbol{\mathcal{M}}_{k}(E)\backslash\{(1,0)^{t},(0,1)^{t}\}\quad\forall E\in\mathcal{T}_h,\label{eq:dof_grad}\\
\int_E\btau :\boldsymbol{\rho}\qquad&\forall\boldsymbol{\rho}\in\boldsymbol{\mathcal{H}}_{k}^{\bot}(E)\quad\forall E\in\mathcal{T}_h,\label{eq:dof_rot}
\end{align}
where 
\begin{align*}
\boldsymbol{\mathcal{M}}_k(\mathcal{O}):=\left\{(q,0)^{t}: q\in\mathcal{M}_{k}(\mathcal{O})\right\}\cup \left\{(0,q)^{t}:q\in\mathcal{M}_{k}(\mathcal{O}) \right\}, \\
\end{align*}
and 
\begin{equation*}
\ds \boldsymbol{\mathcal{H}}_k^{\bot}:=\left\{\begin{pmatrix}\mathbf{q}\\\boldsymbol{0}\end{pmatrix} \,:\mathbf{q}\in\mathcal{H}_k^{\bot}(E)\right\}\cup\left\{\begin{pmatrix}\boldsymbol{0}\\\mathbf{q}\end{pmatrix} \,:\mathbf{q}\in\mathcal{H}_k^{\bot}(E)\right\}.
\end{equation*}
Finally, for every decomposition $ \mathcal{T}_h$ of $\O$ into simple polygons  $E$, we define  the global virtual element space
$$\mathbb{H}_h:=\{\btau_h\in \mathbb{H}: \btau_h|_E\in \mathbb{H}_h^E\,\, \text{for all}\,\, E\in \mathcal{T}_h\}.$$

Finally, we now define the following global virtual element space that allows us to discretize $\mathbb{H}_0$
$$\mathbb{H}_{0,h}:=\left\{\btau_h\in \mathbb{H}_h: \int_\O\tr(\btau_h)=0\right\}.$$
 
\subsection{Polynomial projections and discrete forms}
This section is devoted to the construction of several polynomial projections that play a key role in defining the discrete bilinear forms used to approximate the continuous one.

Let us begin by recalling two classical $\L^2$ projections. First, we consider the projector
$\mathcal{P}_{k}^h:\textbf{L}^2(\O)\rightarrow\textbf{\textsc{P}}_{k}(\mathcal{T}_h)$,
 which is defined for all $\btau\in \textbf{L}^2(\O)$ by
\begin{equation*}
	\ds\int_E\mathcal{P}_{k}^h(\btau)\cdot\boldsymbol{q}=\int_E\btau\cdot\boldsymbol{q}\quad\forall E\in\mathcal{T}_h, \quad\forall\boldsymbol{q}\in\textbf{\textsc{P}}_{k}(E).
\end{equation*}

Next, with the aim of introducing another important projection, we decompose the bilinear form $a(\cdot,\cdot)$ as follows:
\begin{align*}
	\displaystyle a(\bsig,\btau)=\sum_{E\in\mathcal{T}_h}a^{\E}(\bsig,\btau)&:=\sum_{\E\in\mathcal{T}_h}\int_{E}\bsig^{\td}:\btau^{\td} \quad \forall \bsig,\btau \in \mathbb{H}(\bdiv;\O).
\end{align*}

Thus, inspired by the research~\cite[Subection 3.3]{zbMATH06969484}, for each $E\in\mathcal{T}_h$ we consider $\Pi_h^E:=\mathbb{L}^2(E)\rightarrow\mathbb{P}_k(E)$ be the $\mathbb{L}^2(E)$-orthogonal projector, which satisfies the following properties
\begin{itemize}
\item[(P1)] $ a^{\E}\big(\Pi_h^E\btau,\widehat{\Pi}_h^{\E}\boldsymbol{\rho} \big)=a^{\E}\big(\Pi_h^E\btau,\boldsymbol{\rho} \big)$, for all $\btau,\boldsymbol{\rho}\in\mathbb{L}^2(E)$
\item[(P2)] for any element $E\in\CT_h$, there holds
	\begin{equation*}
		\|\Pi_h^E(\btau)\|_{0,E}\leq\|\btau\|_{0,E}\quad\forall\btau\in\mathbb{L}^2(E),
	\end{equation*}
\item[(P3)] given an integer $0\leq m\leq k+1$, there exists a positive constant $C$, independent of $E$, such that 
	\begin{equation}\label{eq_aproxpro}
		\| \btau-\Pi_h^E\btau\|_{0,E}\leq C h_E^m|\btau|_{m,E}  \quad \forall \btau\in\mathbb{H}^m(E).
	\end{equation}
\end{itemize}

It is well known that the projection operators $\mathcal{P}_{k}^h$ and $\Pi_h^E \boldsymbol{\tau}$ are fully computable by means of  the degrees of freedom \eqref{eq:dof_normal}-\eqref{eq:dof_rot}.

On the other hand, let $S^{\E}(\cdot,\cdot)$ be any symmetric positive definite bilinear form satisfying 
\begin{equation}
\label{eq:s_stable}
	\dfrac{c_0}{\nu}\int_E\btau_h:\btau_h\leq S^{\E}(\btau_h,\btau_h)\leq \dfrac{c_1}{\nu}\int_E\btau_h:\btau_h\quad\forall\btau_h\in\mathbb{H}_h^E,
\end{equation}
where $c_0$ and $c_1$ are positive constants depending on the mesh assumptions.

Thus, for each element we define the bilinear form  approximating $a^{\E}(\cdot,\cdot)$
\begin{equation}
 a_h^{\E}(\bsig_h,\btau_h):=a^{\E}(\Pi_h^E\bsig_h,\Pi_h^E\btau_h)+ S^E\left(\bsig_h-\Pi_h^E
\bsig_h,\btau_h-\Pi_h^E \btau_h\right).
\end{equation}

The construction of $a_h^E(\cdot,\cdot,)$ guarantees the usual consistency property of the VEM. i.e, for all $E\in \CT_h$ and $\btau_h\in\mathbb{L}^2(E)$ 
\begin{itemize}
\item consistency: 
$$a_h^E(\btau_h,\boldsymbol{q}_k)=a^E(\btau_h,\boldsymbol{q}_k),\qquad \forall\boldsymbol{q_k}\in\mathbb{P}_k(E), $$
which follows from the projector property and (P1).
\item Additionally, we have the following estimate:
$$a_h^E(\btau_h,\btau_h)\leq \dfrac{1}{\nu}\max\left\{\left(\dfrac{\sqrt{2}+2}{2}\right)^2,c_1\right\}(\|\btau_h\|_{0,E}^2+\|\btau_h-\Pi_h^E\btau_h\|_{0,E}^2),$$ 
which follows from \eqref{eq:s_stable} and  (P2).
\end{itemize}
As a consequence of the previous property, it is easy to check that this bilinear form is bounded. Let $\bsig_h,\btau_h\in\mathbb{H}_h$ we have

\begin{align*}
|a_h^E(\bsig_h,\btau_h)|&\leq \dfrac{1}{\nu}\max\left\{\left(\dfrac{\sqrt{2}+2}{2}\right)^2,c_1\right\}(\|\bsig_h\|_{0,\O}^2+\|\bsig_h-\Pi_h^E\bsig_h\|_{0,E}^2)^{1/2}\\
&\hspace{1.5 cm}\times(\|\btau_h\|_{0,E}^2+\|\btau_h-\Pi_h^E\btau_h\|_{0,E}^2)^{1/2}\\
&\leq \dfrac{C}{\nu}\max\left\{\left(\dfrac{\sqrt{2}+2}{2}\right)^2,c_1\right\}\|\bsig_h\|_{0,E}\|\btau_h\|_{0,E}\\
&\leq \dfrac{C}{\nu}\max\left\{\left(\dfrac{\sqrt{2}+2}{2}\right)^2,c_1\right\}\|\bsig_h\|_{\bdiv,E}\|\btau_h\|_{\bdiv,E}.
\end{align*}

We also write locally the forms $b(\cdot,\cdot)$,  $c(\cdot,\cdot)$, and $d(\cdot,\cdot)$ as follows
\begin{align*}
b^{\E}(\btau_h,\bv_h)&:=\int_\E\bdiv\btau_h\cdot\bv_h, \\
c^{\E}_h(\bv_h,\btau_h)&:=\int_{\E}(\bv_h\otimes\boldsymbol{\beta})^{\dd}: \Pi_h^E \btau_h,\\
d^E(\bv_h,\bw_h)&:=\int_{\E}\kappa \bv_h\cdot \bw_h\qquad \bv_h,\bw_h\in \mathbf{P}_{k-1}(\mathcal{T}_h).
\end{align*}
For the bilinear forms $b^E(\cdot,\cdot)$ and $d^E(\cdot,\cdot)$ the continuity is straightforward. For $c_h^E(\cdot,\cdot)$ we observe that
\begin{align*}
|c_h^E(\bv_h,\btau_h)|&\leq \frac{1}{\nu}\|(\bv_h\otimes\boldsymbol{\beta})^{\texttt{d}}\|_{0,E}\|\Pi_h^E\btau_h\|_{0,E}\\
&\leq \frac{\|\boldsymbol{\beta}\|_{\infty,E}}{\nu}\left(\frac{2+\sqrt{2}}{2}\right)\|\bv_h\|_{0,E}\|\btau_h\|_{\bdiv,E}.
\end{align*}

For the functional $F(\cdot)$,we have $F^{\E}(\bv_h) := ( \bF,  \bv_h )_{0,\E} \quad\bv_h \in \mathbf{Q}_h^\E$.

As usual, the construction of the global bilinear forms is obtained by summing the corresponding elemental contributions over all elements of $\mathcal{T}_h$. For instance,
\begin{equation}
a_h(\bsig_h,\btau_h):= \sum_{E\in\mathcal{T}_h} a_h^E(\bsig_h,\btau_h),\qquad
\forall \bsig_h,\btau_h \in \mathbb{H}_h.
\end{equation}
Analogously, we define the remaining global bilinear forms.

\subsection{The mixed virtual element formulation}
With the spaces described previously at hand,  we are in position to introduce the VEM discretization of problem \eqref{eq:variational1_H0}  which reads as follows:  Find  $(\bsig_h,\bu_h)\in\mathbb{H}_{0,h}\times\mathbf{Q}_h$, such that 
\begin{equation}
	\label{eq:fv:disc}
	\left\{
\begin{array}{rcll}
a_h(\bsig_h,\btau_h)+b(\btau_h,\bu_h)+c_h(\bu_h,\btau_h)&=&0&\forall \btau_h\in \mathbb{H}_{0,h},\\
b(\bsig_h,\bv_h)-d(\bu_h,\bv_h)&=&F(\bv_h)&\forall \bv_h\in \mathbf{Q}_h.
	\end{array}
	\right.
\end{equation}

To recover the discrete pressure, we consider the following equation
\begin{equation}
\label{eq:discrete_pressure}
p_h=-\frac{1}{2}(\tr(\Pi_h\bsig)+\tr(\bu_h\otimes\boldsymbol{\beta}))\quad\text{in}\,\,\O,
\end{equation}
where $\Pi_h$ is given by: $(\Pi_h\btau)|_E:=\Pi_h^E(\btau|_E)$ for all $\btau\in\mathbb{H}$, with $\Pi_h^E$.
Our task now is to establish that \eqref{eq:fv:disc} is well posed. Let us begin by recalling the following inf-sup condition that bilinear form $b(\cdot,\cdot)$ satisfies for the VEM spaces or our paper. According to \cite[Lemma 5.3]{MR3614887}, there exists a constant $\widehat{\gamma}>0$, independent of $h$, such that
\begin{equation}
\label{eq:inf-sup_disc}
\displaystyle\sup_{\btau_h\in\mathbb{H}_{0,h}}\frac{b(\btau_h,\boldsymbol{q}_h)}{\|\btau_h\|_{\bdiv,\O}}\geq\widehat{\gamma}\|\boldsymbol{q}_h\|_{0,\O}\quad\forall \boldsymbol{q}_h\in\mathbf{Q}_h.
\end{equation}
For the forthcoming analysis, and for all $(\bsig_h,\bu_h),(\btau_h,\bv_h)\in\mathbb{H}_{0,h}\times\mathbf{Q}_h$,   it is convenient to define the following bilinear form 
\begin{equation}
\label{eq:A_h}
A_h((\bsig_h,\bu_h),(\btau_h,\bv_h)):=a_h(\bsig_h,\btau_h)+b(\btau_h,\bu_h)+b(\bsig_h,\bv_h)-d(\bu_h,\bv_h).
\end{equation}
Observe that this bilinear form does not depend on the convective term. Also, with the aim of simplify the presentation of the material, we define the following norm
\begin{equation*}
\vertiii{(\btau,\bv)}^2:=\|\btau\|_{\bdiv,\O}^2+\|\bv\|_{0,\O}^2\quad\forall\btau\in\mathbb{H},\,\,\forall\bv\in\mathbf{Q}.
\end{equation*}
The following result is instrumental.
\begin{lemma}
\label{inf-sup_Ah}
For each $(\bsig_h,\bu_h)\in \mathbb{H}_{0,h}\times \mathbf{Q}_h$, there exists $(\btau_h,\bv_h)\in \mathbb{H}_{0,h}\times \mathbf{Q}_h$ with 
$\vertiii{(\btau_h,\bv_h)} \leq C\vertiii{(\bsig_h,\bu_h)},$ and a constant $C_J>0$, depending on $\nu$, $\kappa$, the constant $\widehat{\gamma}$ of \eqref{eq:inf-sup_disc} and  the domain $\O$,  such that
\begin{equation}
\vertiii{(\bsig_h,\bu_h)}^2\leq C_J A_h((\bsig_h,\bu_h),(\btau_h,\bv_h)).
\end{equation}
\end{lemma}
\begin{proof}
Let $\alpha>0$ and set $(\btau_h,\bv_h):=(\bsig_h,-\bu_h+\alpha\bdiv\bsig_h)$. Consider the following well known  Young's inequality $ab\leq (2\epsilon_1)^{-1}a^2+2^{-1}\epsilon_1 b^2$.  Replacing this choice of functions in \eqref{eq:A_h}, applying the aforementioned Young's inequality,  and invoking property \eqref{eq:s_stable}, we have 
\begin{multline*}
A_h((\bsig_h,\bv_h),(\bsig_h,-\bu_h+\alpha\bdiv\bsig_h))=a_h(\bsig_h,\bsig_h)+b(\bsig_h,\bu_h)\\
+b(\bsig_h,-\bu_h+\alpha\bdiv\bsig_h)-d(\bu_h,-\bu_h+\alpha\bdiv\bsig_h)\\
=a_h(\bsig_h,\bsig_h)+\alpha b(\bsig_h,\bdiv\bsig_h)+d(\bu_h,\bu_h)-\alpha d(\bu_h,\bdiv\bsig_h)\\
\geq \sum_{E\in\CT_h}\left[\left\|\left(\Pi_h^E\bsig_h\right)^{\texttt{d}}\right\|_{0,E}^2+S^E(\bsig_h-\Pi_h^E\bsig_h,\bsig_h-\Pi_h^E\bsig_h)\right]\\
+\alpha\|\bdiv\bsig_h\|_{0,\O}^2+\kappa\|\bu_h\|_{0,\O}^2
+\alpha\kappa\left(-\frac{\epsilon_1\|\bdiv\bsig_h\|_{0,\O}^2}{2}-\frac{\|\bu_h\|_{0,\O}^2}{2\epsilon_1}\right)\\
\geq\sum_{E\in\CT_h}\left[\frac{1}{\nu}\left\|\left(\Pi_h^E\bsig_h\right)^{\texttt{d}}\right\|_{0,E}^2+c_0\|\bsig_h-\Pi_h^E\bsig_h\|_{0,E}^2\right]+\alpha\left(1-\frac{\kappa\epsilon_1}{2}\right)\|\bdiv\bsig_h\|_{0,\O}^2\\
+\kappa\left(1-\frac{\alpha}{2\epsilon_1}\right)\|\bu_h\|_{0,\O}^2\\
\geq \alpha\left(1-\frac{\kappa\epsilon_1}{2}\right)\|\bdiv\bsig_h\|_{0,\O}^2+\kappa\left(1-\frac{\alpha}{2\epsilon_1}\right)\|\bu_h\|_{0,\O}^2\\
\geq \alpha\left(\frac{\kappa\epsilon_1}{2}-1\right)(-\|\bsig_h\|_{\bdiv,\O}^2)+\kappa\left(1-\frac{\alpha}{2\epsilon_1}\right)\|\bu_h\|_{0,\O}^2\\
=\alpha\left(1-\frac{\kappa\epsilon_1}{2}\right)\|\bsig_h\|_{\bdiv,\O}^2+\kappa\left(1-\frac{\alpha}{2\epsilon_1}\right)\|\bu_h\|_{0,\O}^2.
\end{multline*}

Now, as a consequence of the discrete inf-sup condition \eqref{eq:inf-sup_disc}, we have that the operator $\bdiv:\mathbb{H}_h\rightarrow\mathbf{Q}_h$ is surjective, implying that for $\bu_h\in\mathbf{Q}_h$ there exists $\widetilde{\bsig}_h\in\mathbb{H}_h$ such that $\bdiv\widetilde{\bsig}_h=-\bu_h$. Moreover, from \eqref{eq:inf-sup_disc}  we have that $\|\widetilde{\bsig}_h\|_{\bdiv,\O}\leq \widehat{\gamma}^{-1}\|\bu_h\|_{0,\O}$. Hence, if we set $(\btau_h,\bv_h)=(-\widetilde{\bsig}_h,\boldsymbol{0})$ in \eqref{eq:A_h}  and for $\epsilon_2>0$, we have 
\begin{multline*}
A_h((\bsig_h,\bu_h),(-\widetilde{\bsig}_h,\boldsymbol{0}))=a_h(\bsig_h,-\widetilde{\bsig}_h)-b(\widetilde{\bsig}_h,\bu_h)\\
=\sum_{E\in\CT_h}-a_h^E(\bsig_h,\widetilde{\bsig}_h)+\|\bu_h\|^2_{0,\O}
\geq \sum_{E\in\CT_h}-\overline{C}\|\bsig_h\|_{0,E}\|\widetilde{\bsig}_h\|_{0,\E}+\|\bu_h\|_{0,\O}^2\\
\geq \overline{C}\left(-\frac{\|\bsig_h\|_{0,\O}^2}{2\epsilon_2}-\frac{\epsilon_2\|\widetilde{\bsig}_h\|_{0,\O}^2}{2}\right)+\|\bu_h\|_{0,\O}^2\\
\geq \overline{C}\left(-\frac{\|\bsig_h\|_{\bdiv,\O}^2}{2\epsilon_2}-\frac{\epsilon_2\widehat{\gamma}^{-2}\|\bu_h\|_{0,\O}^2}{2}\right)+\|\bu_h\|_{0,\O}^2\\
=-\frac{\overline{C}}{2\epsilon_2}\|\bsig_h\|_{\bdiv,\O}^2+\left(1-\frac{\overline{C}\epsilon_2\widehat{\gamma}^{-2}}{2} \right)\|\bu_h\|_{0,\O}^2
\end{multline*}

Let $\delta>0$ and set $(\btau_h,\bv_h)=(\bsig_h-\delta\widetilde{\bsig}_h,-\bu_h+\alpha\bdiv\bsig_h)$. Operating as before,  we have
\begin{multline*}
A_h((\bsig_h,\bu_h),(\bsig_h-\delta\widetilde{\bsig}_h,-\bu_h+\alpha\bdiv\bsig_h))=a_h(\bsig_h,\bsig_h)+\delta[-a_h(\bsig_h,\widetilde{\bsig}_h)-b(\widetilde{\bsig}_h,\bu_h)]\\
+\alpha[b(\bsig_h,\bdiv\bsig_h)-d(\bu_h,\bdiv\bsig_h)]+\kappa\|\bu_h\|_{0,\O}^2\\
=\sum_{E\in\CT_h}a_h^E(\bsig_h,\bsig_h)+\delta\left[\sum_{E\in\CT_h}-a_h^E(\bsig_h,\widetilde{\bsig}_h)-b(\widetilde{\bsig}_h,\bu_h)\right]\\
+\alpha[b(\bsig_h,\bdiv\bsig_h)-d(\bu_h,\bdiv\bsig_h)]+\kappa\|\bu_h\|_{0,\O}^2\\
\geq \delta\left[-\frac{\overline{C}}{2\epsilon_2}\|\bsig_h\|_{\bdiv,\O}^2+\left(1-\frac{\overline{C}\epsilon_2\widehat{\gamma}^{-2}}{2}\|\bu_h\|_{0,\O}^2 \right) \right]\\
+\alpha\left(1-\frac{\kappa\epsilon_1}{2}\right)\|\bsig_h\|_{\bdiv,\O}^2+\kappa\left( 1-\frac{\alpha}{2\epsilon_1}\right)\|\bu_h\|_{0,\O}^2\\
=\left[\alpha\left(1-\frac{\kappa\epsilon_1}{2}\right)-\frac{\delta\overline{C}}{2\epsilon_2} \right]\|\bsig_h\|_{\bdiv,\O}^2+\left[\kappa\left(1-\frac{\alpha}{2\epsilon_1}\right)+\delta\left(1-\frac{\overline{C}\epsilon_2\widehat{\gamma}^{-2}}{2} \right) \ \right]\|\bu_h\|_{0,\O}^2.
\end{multline*}
Now, choosing
\begin{equation*}
\epsilon_1:=\frac{1}{\kappa},\quad\epsilon_2:=\frac{\widehat{\gamma}^2}{\overline{C}},\quad\alpha:=\frac{1}{\kappa},\quad\delta:=\frac{\widehat{\gamma}^2}{2\kappa\overline{C}^2},
\end{equation*}
we have
\begin{equation*}
A_h((\bsig_h,\bu_h),(\bsig_h-\delta\widetilde{\bsig}_h,-\bu_h+\alpha\bdiv\bsig_h))\geq C_J\vertiii{(\bsig_h,\bu_h)}^2,
\end{equation*}
where $C_J:=\displaystyle\min\left\{\frac{1}{4\kappa}, \frac{\kappa}{2}+\frac{\widehat{\gamma}^2}{4\kappa\overline{C}^2}\right\}$.
\end{proof}

The next step is to analyze the uniqueness of the discrete solution. This is contained in the following result.
\begin{lemma}
\label{lmm:uniqueness}
Let $(\bsig_h,\bu_h)\in\mathbb{H}_h\times\mathbf{Q}_h$ be a solution of \eqref{eq:fv:disc}. If  the following condition holds $\displaystyle\frac{C_J}{\nu}\left(\frac{2+\sqrt{2}}{2}\right)\|\boldsymbol{\beta}\|_{\infty,\O} <1$, where  $C_J>0$ is the constant given by Lemma \ref{inf-sup_Ah}, then $(\bsig_h,\bu_h)\in\mathbb{H}_h\times\mathbf{Q}_h$ is unique.
\end{lemma}
\begin{proof}
Let us suppose that $(\bsig_h^1,\bu_h^1), (\bsig_h^2,\bu_h^2)\in\mathbb{H}_h\times\mathbf{Q}_h$  are two solutions of  \eqref{eq:fv:disc}. We will prove that are the same. Invoking Lemma \ref{inf-sup_Ah} we have for $(\btau_h,\bv_h)\in\mathbb{H}_{h}\times\mathbf{Q}$ that there holds $\vertiii{(\btau_h,\bv_h)}\leq \vertiii{(\bsig_h^1-\bsig_h^2,\bu_h^1-\bu_h^2)}$ and  there exists a constant $C_J>0$ such that 
\begin{align*}
\vertiii{(\bsig_h^1-\bsig_h^2,\bu_h^1-\bu_h^2)}^2&\leq C_JA_h((\bsig_h^1-\bsig_h^2,\bu_h^1-\bu_h^2),(\btau_h,\bv_h))\\
&=C_J[A_h((\bsig_h^1,\bu_h^1),(\btau_h,\bv_h))-A_h((\bsig_h^2,\bu_h^2),(\btau_h,\bv_h))]\\
&=C_J\sum_{E\in\CT_h}c_h^E(\bu_h^2-\bu_h^1,\btau_h)\\
&\leq \frac{C_J}{\nu}\left(\frac{2+\sqrt{2}}{2}\right)\sum_{E\in\CT_h}\|\boldsymbol{\beta}\|_{\infty,E}\|\bu_h^2-\bu_h^1\|_{0,E}\|\btau_h\|_{\bdiv,E}\\
&=\frac{C_J}{\nu}\left(\frac{2+\sqrt{2}}{2}\right)\|\boldsymbol{\beta}\|_{\infty,\O}\|\bu_h^2-\bu_h^1\|_{0,\O}\|\btau_h\|_{\bdiv,\O}.
\end{align*}
The above  is summarize in the following estimate
\begin{equation*}
\|\bsig_h^2-\bsig_h^1\|^2_{\bdiv,\O}+\|\bu_h^2-\bu_h^1\|^2_{0,\O}-\frac{C_J}{\nu}\left(\frac{2+\sqrt{2}}{2}\right)\|\boldsymbol{\beta}\|_{\infty,\O}\|\bu_h^2-\bu_h^1\|_{0,\O}\|\btau_h\|_{\bdiv,\O}\leq 0,
\end{equation*}
implying that 
\begin{equation*}
\left(1-\frac{C_J}{\nu}\left(\frac{2+\sqrt{2}}{2}\right)\|\boldsymbol{\beta}\|_{\infty,\O} \right)(\|\bsig_h^2-\bsig_h^1\|^2_{\bdiv,\O}+\|\bu_h^2-\bu_h^1\|^2_{0,\O})\leq 0.
\end{equation*}
Hence, the following condition on the data
$$\frac{C_J}{\nu}\left(\frac{2+\sqrt{2}}{2}\right)\|\boldsymbol{\beta}\|_{\infty,\O} <1,$$
concludes the proof.
\end{proof}
\section{Error analysis}
\label{sec:error}
The aim of this section is to prove error estimates for the pseudostress and velocity fields. Let us begin by recalling some approximation properties of the spaces.	
	Let  $\bcI_k^h: \mathbb{H}^t(\O) \rightarrow \mathbb{H}_h$ be the tensorial version of the VEM-interpolation operator, 
	which satisfies the following classical error estimate, see \cite[Lemma 6]{MR3660301},
	\begin{equation}\label{asymp0}
		\norm{\btau - \bcI_k^h \btau}_{0,\O} \leq C h^{\min\{t, k+1\}} \norm{\btau}_{t,\O} \qquad \forall \btau \in \mathbb{H}^t(\O), \quad t>1/2.
	\end{equation}
	Moreover, the following commuting diagram property holds true, see \cite[Lemma 5]{MR3660301}:
	\begin{equation}\label{asympDiv}
		\norm{\bdiv (\btau - \bcI_k^h\btau) }_{0,\O} = \norm{\bdiv \btau - \mathcal{P}_{k}^h \bdiv \btau }_{0,\O} 
		\leq C h^{\min\{t, k\}} \norm{\bdiv\btau}_{t,\O},
	\end{equation}
	for  $\bdiv \btau \in \H^t(\O)^{2}$ and  $\mathcal{P}_{k}^h$ being  the $\LO^2$-orthogonal projection onto $\textsc{P}_{k}$.
	Also  we define the local restriction of the interpolant operator as $\btau_I:=\bcI_k^h(\btau)|_E\in \mathbb{H}_h^E$.

For  the forthcoming analysis, we make the following assumption on the solution.
\begin{assumption}
\label{as:dependence}
Let $\bF\in\mathbf{Q}$. There exists $s\in[1,k+1]$ such that
\begin{equation*}
\|\bsig\|_{\mathbb{H}^s(\O)}+\|\bdiv\bsig\|_{1+s,\O}+\|\bu\|_{1+s,\O}\leq C_{\O}\|\bF\|_{0,\O},
\end{equation*}
where the constant $C_{\O}>0$ depends on the domain $\O$ and the physical parameters.
\end{assumption}

Now we are in position to derive an a priori error estimate for the method.
\begin{lemma}
\label{lmm:conv1}
Let $\boldsymbol{f}\in \mathbf{Q}$. Let  $(\bsig,\bu)\in\mathbb{H}_0\times\mathbf{Q}$ and $(\bsig_h,\bu_h)\in\mathbb{H}_{0,h}\times\mathbf{Q}_h$ be the  solutions of problems \eqref{eq:variational1_H0} and \eqref{eq:fv:disc}, respectively.   Suppose that $\bu\otimes\boldsymbol{\beta}\in \mathbb{H}^s(\O)$ and   that Assumption \ref{as:dependence} holds. Then, if  $C_J\|\boldsymbol{\beta}\|_{\infty,\O}\nu^{-1}<\dfrac{1}{2}$, there exists a positive constant $\widehat{C}_{\nu,\beta,\kappa}$, independent of $h$, such that 
$$
\vertiii{(\bsig-\bsig_h,\bu-\bu_h)}\leq  \widehat{C}_{\nu,\beta,\kappa}h^s\|\boldsymbol{f}\|_{0,\O}.
$$

\end{lemma}
\begin{proof}
Let $(\bsig,\bu)$ and $(\bsig_h,\bu_h)$ be the solutions 
of problems \eqref{eq:variational1_H0} and \eqref{eq:fv:disc}, respectively,  and 
let $(\btau_h,\bv_h)\in\mathbb{H}_{0,h}\times\mathbf{Q}_h$ be such that
$$\vertiii{(\btau_h,\bv_h)}\leq C_J\vertiii{(\bcI_k^h\bsig-\bsig,\mathcal{P}_k^h\bu-\widehat{\bu})},$$ where 
$\bcI_k^h\bsig$ and $\mathcal{P}_k^h\bu$ are the best approximations of $\bsig$ and $\bu$, respectively.
Then, from triangle inequality we have
$$\vertiii{(\bsig-\bsig_h,\bu-\bu_h)}
\leq \vertiii{(\bsig-\bcI_k^h\bsig,\bu-\mathcal{P}_k^h\bu)}+
\vertiii{(\bcI_k^h\bsig-\bsig_h,\mathcal{P}_k^h\bu-\bu_h)}.$$

Let us focus on the second term on the right-hand side of the inequality above. Invoking Lemma \ref{inf-sup_Ah}  we have 
\begin{multline*}
C_J^{-1}\vertiii{(\bcI_k^h\bsig-\bsig_h,\mathcal{P}_k^h\bu-\bu_h)}^2\leq A_h((\bcI_k^h\bsig-\bsig_h,\mathcal{P}_k^h\bu-\bu_h),(\btau_h,\bv_h))\\
= a_h(\bcI_k^h\bsig-\bsig_h,\btau_h)+b(\btau_h,\mathcal{P}_k^h\bu-\bu_h)+b(\bcI_k^h\bsig-\bsig_h,\bv_h)-d(\mathcal{P}_k^h\bu-\bu_h,\bv_h)\\
=\sum_{E\in\CT_h} a_h^E(\bcI_k^h\bsig-\Pi_h^E\bsig,\btau_h)+a_h^E(\Pi_h^E\bsig,\btau_h)-a_h^E(\bsig_h,\btau_h)
+b^E(\btau_h,\mathcal{P}_k^h\bu)\\
-b^E(\btau_h,\bu_h)
+b^E(\bcI_k^h\bsig,\bv_h)
-b^E(\bsig_h,\bv_h)-d^E(\mathcal{P}_k^h\bu,\bv_h)+d^E(\bu_h,\bv_h)\\
=\underbrace{\sum_{E\in\CT_h}\left[a_h^E(\bcI_k^h\bsig-\Pi_h^E\bsig,\btau_h)+a^E(\Pi_h^E\bsig-\bsig,\btau_h)\right]}_{\mathbf{M}_1}+\underbrace{b(\btau_h,\mathcal{P}_k^h\bu-\bu)}_{\mathbf{M}_2}\\
+\underbrace{b(\bcI_k^h\bsig-\bsig,\bv_h)}_{\mathbf{M}_3}-\underbrace{d(\mathcal{P}_k^h\bu-\bu,\bv_h)}_{\mathbf{M}_4}+\underbrace{c_h(\bu_h,\btau_h)-c(\bu,\btau_h)}_{\mathbf{M}_5}.
\end{multline*}
We now proceed to derive bounds for each term $\mathbf{M}_i$, with $i=1,...,5.$. The term $\mathbf{M}_1$ is bounded by using the continuity of the bilinear forms $a_h^E(\cdot,\cdot)$ and $a(\cdot,\cdot)$
\begin{multline*}
\mathbf{M}_1\leq \sum_{E\in\CT_h}\dfrac{1}{\nu}\max\left\{\left(\dfrac{\sqrt{2}+2}{2}\right)^2,c_1\right\}(\|\bcI_k^h\bsig-\Pi_h^E\bsig\|_{0,E}^2+\|\bcI_k^h\bsig-\Pi_h^E\bcI_k^h\bsig\|_{0,E}^2)^{1/2}\\
\times(\|\btau_h\|_{0,E}^2+\|\btau_h-\Pi_h^E\btau_h\|_{0,E}^2)^{1/2}+\sum_{E\in\CT_h}\left(\dfrac{\sqrt{2}+2}{2\nu}\right)^2\|\Pi_h^E\bsig-\bsig\|_{0,E}\|\btau_h\|_{0,E}\\
\leq \sum_{E\in\CT_h}\dfrac{C}{\nu}\max\left\{\left(\dfrac{\sqrt{2}+2}{2}\right)^2,c_1\right\}\left(\|\bcI_k^h\bsig-\Pi_h^E\bsig\|_{0,E}+\|\bcI_k^h\bsig-\Pi_h^E\bcI_k^h\bsig\|_{0,E}\right.\\
\left.+\|\Pi_h^E\bsig-\bsig\|_{0,E}\right)\|\btau_h\|_{0,E}\\
\leq \sum_{E\in\CT_h}\dfrac{3C}{\nu}\max\left\{\left(\dfrac{\sqrt{2}+2}{2}\right)^2,c_1\right\}\left(\|\bsig-\bcI_k^h\bsig\|_{0,E}+\|\bsig-\Pi_h^E\bsig\|_{0,E}\right)\|\btau_h\|_{\bdiv,E}\\
\leq \dfrac{3CC_{\O}}{\nu}\max\left\{\left(\dfrac{\sqrt{2}+2}{2}\right)^2,c_1\right\}h^s\|\boldsymbol{f}\|_{0,\O}\|\btau_h\|_{\bdiv,\O},
\end{multline*}
where $C$ is the constant associated with the approximation properties of the interpolant and the projector (see \eqref{asymp0} and \eqref{eq_aproxpro}). The terms $\mathbf{M}_2$, $\mathbf{M}_3$, and $\mathbf{M}_4$ can be bounded directly using the Cauchy-Schwarz inequality and the approximation properties, in conjunction with the regularity assumptions as follows
\begin{multline*}
\mathbf{M}_2+\mathbf{M}_3+\mathbf{M}_4\leq \|\btau_h\|_{\bdiv,\O}\|\mathcal{P}_k^h\bu-\bu\|_{0,\O}
+(\|\bcI_k^h\bsig-\bsig\|_{\bdiv,\O}+\kappa\|\mathcal{P}_k^h\bu-\bu\|_{0,\O})\|\bv_h\|_{0,\O}\\
\leq CC_{\O}\max\{1,\kappa\}h^s\|\boldsymbol{f}\|_{0,\O}(\|\btau_h\|_{\bdiv,\O}+\|\bv_h\|_{0,\O}).
\end{multline*}
Now, let us focus on the term $\mathbf{M}_5$:
\begin{multline*}
\mathbf{M}_5=c_h(\bu_h-\mathcal{P}_k^h\bu,\btau_h)+c_h(\mathcal{P}_k^h\bu,\btau_h)-\sum_{E\in\CT_h}\left(c_h^E(\bu,\btau_h-\Pi_h^E\btau_h)+c^E(\bu,\Pi_h^E\btau_h)\right)\\
=c_h(\bu_h-\mathcal{P}_k^h\bu,\btau_h)+\sum_{E\in\CT_h}\dfrac{1}{\nu}\int_E\left((\mathcal{P}_k^h\bu-\bu)\otimes\boldsymbol{\beta}\right)^d:\Pi_h^E\btau_h\\
-\sum_{E\in\CT_h}\dfrac{1}{\nu}\int_E\left(\left(\bu\otimes\boldsymbol{\beta}\right)^d-\Pi_h^E\left(\bu\otimes\boldsymbol{\beta}\right)^d\right):(\btau_h-\Pi_h^E\btau_h)\\
\leq \sum_{E\in\CT_h}\frac{\|\boldsymbol{\beta}\|_{\infty,E}}{\nu}\left(2+\sqrt{2}\right)\|\bu_h-\mathcal{P}_k^h\bu\|_{0,E}\|\Pi_h^E\btau_h\|_{0,E}\\
+\sum_{E\in\CT_h}\dfrac{2+\sqrt{2}}{2\nu}\|\bu\otimes\boldsymbol{\beta}-\Pi_h^E\left(\bu\otimes\boldsymbol{\beta}\right)\|_{0,\E}\|\btau_h-\Pi_h^E\btau_h\|_{0,\E}\\
\leq C\frac{\|\boldsymbol{\beta}\|_{\infty,\O}}{\nu}\left(2+\sqrt{2}\right)\|\bu_h-\mathcal{P}_k^h\bu\|_{0,\O}\|\btau_h\|_{\bdiv,\O}\\
+Ch^s\left(\dfrac{2+\sqrt{2}}{2\nu}\right)\sum_{E\in\CT_h}|\bu\otimes\boldsymbol{\beta}|_{s,E}\|\btau_h\|_{\bdiv,\O}\\
\leq C\frac{\|\boldsymbol{\beta}\|_{\infty,\O}}{\nu}\left(2+\sqrt{2}\right)\|\bu_h-\mathcal{P}_k^h\bu\|_{0,\O}\|\btau_h\|_{\bdiv,\O}\\
+Ch^s\left(\dfrac{2+\sqrt{2}}{2\nu}\right)\|\boldsymbol{\beta}\|_{\mathbf{W}^{s,\infty}(\O)}\|\bu\|_{s,\O}\|\btau_h\|_{\bdiv,\O}\\
\leq C\frac{\|\boldsymbol{\beta}\|_{\infty,\O}}{\nu}\left(2+\sqrt{2}\right)\|\bu_h-\mathcal{P}_k^h\bu\|_{0,\O}\|\btau_h\|_{\bdiv,\O}\\
+CC_{\O}h^s\left(\dfrac{2+\sqrt{2}}{2\nu}\right)\|\boldsymbol{\beta}\|_{\mathbf{W}^{s,\infty}(\O)}\|\boldsymbol{f}\|_{0,\O}\|\btau_h\|_{\bdiv,\O}.
\end{multline*}

Finally, by combining all the estimates obtained above, the following bound is obtained
\begin{multline*}
\left(\dfrac{1}{C_J}-\frac{\|\boldsymbol{\beta}\|_{\infty,\O}}{\nu}\left(2+\sqrt{2}\right)\right)\vertiii{(\bcI_k^h\bsig-\bsig_h,\mathcal{P}_k^h\bu-\bu_h)}^2\\
\leq CC_{\O}\max\left\{C_\nu,\max\{1,\kappa\},\left(\dfrac{2+\sqrt{2}}{2\nu}\right)\|\boldsymbol{\beta}\|_{\mathbf{W}^{s,\infty}(\O)}\right\}\|\boldsymbol{f}\|_{0,\O}\\
\times\vertiii{(\bcI_k^h\bsig-\bsig_h,\mathcal{P}_k^h\bu-\bu_h)},
\end{multline*}
where $C_\nu:= \dfrac{3}{\nu}\max\left\{\left(\dfrac{\sqrt{2}+2}{2}\right)^2,c_1\right\}.$
Hence, the proof is completed by taking $\widehat{C}_{\nu,\beta,\kappa}:=CC_{\O}\max\left\{C_\nu,\max\{1,\kappa\},\left(\dfrac{2+\sqrt{2}}{2\nu}\right)\|\boldsymbol{\beta}\|_{\mathbf{W}^{s,\infty}(\O)}\right\}$.
\end{proof}

With this result at hand, we can also provide an error estimate for the pressure.
\begin{corollary}
\label{cor:error_prressure}
Under Assumption \ref{as:dependence}, if $p\in\L_0^2(\O)$ and $p_h\in\mathbf{Q}_h\cap\L_0^2(\O)$ are given by \eqref{eq:continuous_pressure} and \eqref{eq:discrete_pressure}, respectively., then, the following estimate holds
\begin{equation*}
\|p-p_h\|_{0,\O}\leq \mathbf{C}h^s\|\bF\|_{0,\O},
\end{equation*}
where $\mathbf{C}:=\max\left\{\displaystyle\frac{\sqrt{2}}{2},\|\boldsymbol{\beta}\|_{\infty,\O}\right\}\widehat{C}_{\nu,\beta,\kappa}$.
\end{corollary}
\begin{proof}
From identities \eqref{eq:continuous_pressure} and \eqref{eq:discrete_pressure} we have
\begin{align*}
\|p-p_h\|_{0,\O}&=\Big\|-\frac{1}{2}(\tr(\bsig)+\tr(\bu\otimes\boldsymbol{\beta}))-\left(-\frac{1}{2}(\tr(\Pi_h\bsig_h)+\tr(\bu_h\otimes\boldsymbol{\beta}))\right) \Big\|_{0,\O} \\
&\leq\frac{1}{2}\|\tr(\bsig-\widehat{\Pi}_h\bsig_h)\|_{0,\O}+\frac{1}{2}\|\tr((\bu_h-\bu)\otimes\boldsymbol{\beta})\|_{0,\O}\\
&\leq \max\left\{\frac{\sqrt{2}}{2},\|\boldsymbol{\beta}\|_{\infty,\O}\right\}(\|\bsig-\Pi_h\bsig_h\|_{0,\O}+\|\bu-\bu_h\|_{0,\O})\\
&\leq \max\left\{\frac{\sqrt{2}}{2},\|\boldsymbol{\beta}\|_{\infty,\O}\right\}(\|\bsig-\Pi_h\bsig\|_{0,\O}+\|\bsig-\bsig_h\|_{0,\O}+\|\bu-\bu_h\|_{0,\O})\\
&\leq \max\left\{\frac{\sqrt{2}}{2},\|\boldsymbol{\beta}\|_{\infty,\O}\right\}\widehat{C}_{\nu,\beta,\kappa}h^s\|\boldsymbol{f}\|_{0,\O},
\end{align*}
where for the last estimate, we have used the stability of the projector, together with the approximation properties \eqref{eq_aproxpro} and the previous lemma.
\end{proof}
\section{Numerical experiments}
\label{sec:numerics}
Now we report a series of numerical tests in order to assess the performance of the proposed mixed method. All the results have been obtained with a Matlab code. The matrix system is solved with the standard instruction "$\backslash$". 

Regarding the VEM scheme, as in \cite{MR2997471}, a choice for $S^{E}(\cdot,\cdot)$ is given by
\begin{equation*}
S^{E}(\boldsymbol{\rho}_{h},\btau_{h}):=\sum_{k=1}^{N_{E}^{k}}m_{i,E}(\boldsymbol{\rho}_{h})m_{i,E}(\btau_{h}),
\end{equation*}
where the set $\left\{m_{i,E}(\boldsymbol{\rho}_{h})\right\}_{i=1}^{N_{E}^{k}}$ corresponds to all the $E$-moments of $\boldsymbol{\rho}_{h}$ associated with the corresponding degrees of freedom. On the other hand, we define the computational errors as follows:
\begin{equation*}
\texttt{e}(\bsig):=\dfrac{\|\bsig-\Pi\bsig_h\|_{0,\O}}{\|\bsig\|_{0,\O}},\quad\texttt{e}(\bu):=\dfrac{\|\bu-\bu_h\|_{0,\O}}{\|\bu\|_{0,\O}},\quad\texttt{e}(p):=\dfrac{\|p-p_h\|_{0,\O}}{\|p\|_{0,\O}}.
\end{equation*}
We remark that computing the $\L^2$-error of the pressure requires a post-processing step to recover $p_h$ from $\bsig_h$ and $\bu_h$ using identity \eqref{eq:discrete_pressure}. On the other hand, the experimental rates of convergence are given by
\begin{equation*}
r(\cdot):=\displaystyle\frac{\log (\texttt{e}(\cdot)/\texttt{e}'(\cdot))}{\log(h/h')},
\end{equation*}
where $\texttt{e}$ and $\texttt{e}'$ denote the corresponding errors for two consecutive meshes with sizes $h$ and $h'$, respectively.
\subsection{Square with null Dirichlet boundary condition}
We begin this section on the unit square $\O:=(-1,1)^2$. Here, we impose the Dirichlet boundary condition $\bu=\boldsymbol{0}$, with $\nu=\kappa=1$ and $\boldsymbol{\beta}:=(1,0)^t$ as physical parameters. The vector field $\boldsymbol{f}$ is chosen so that the exact solution to problem \eqref{eq:Oseen_system} is: 
$$\bu(x,y):=\left(\dfrac{\partial \phi}{\partial y},-\dfrac{\partial \phi}{\partial x}\right)\qquad\text{and } p(x,y):=x,$$
where  $\phi(x,y):=(1-x^2)^2(1-y^2)^2$.  
The convexity of the domain, together with the selected boundary condition will yield to optimal order of convergence for the method. A sample of the polygonal meshes for this test are depicted in Figure \ref{FIG:Meshes}.
\begin{figure}[H]
	\begin{center}
		\begin{minipage}{13cm}
			\centering\includegraphics[height=4.2cm, width=4.2cm]{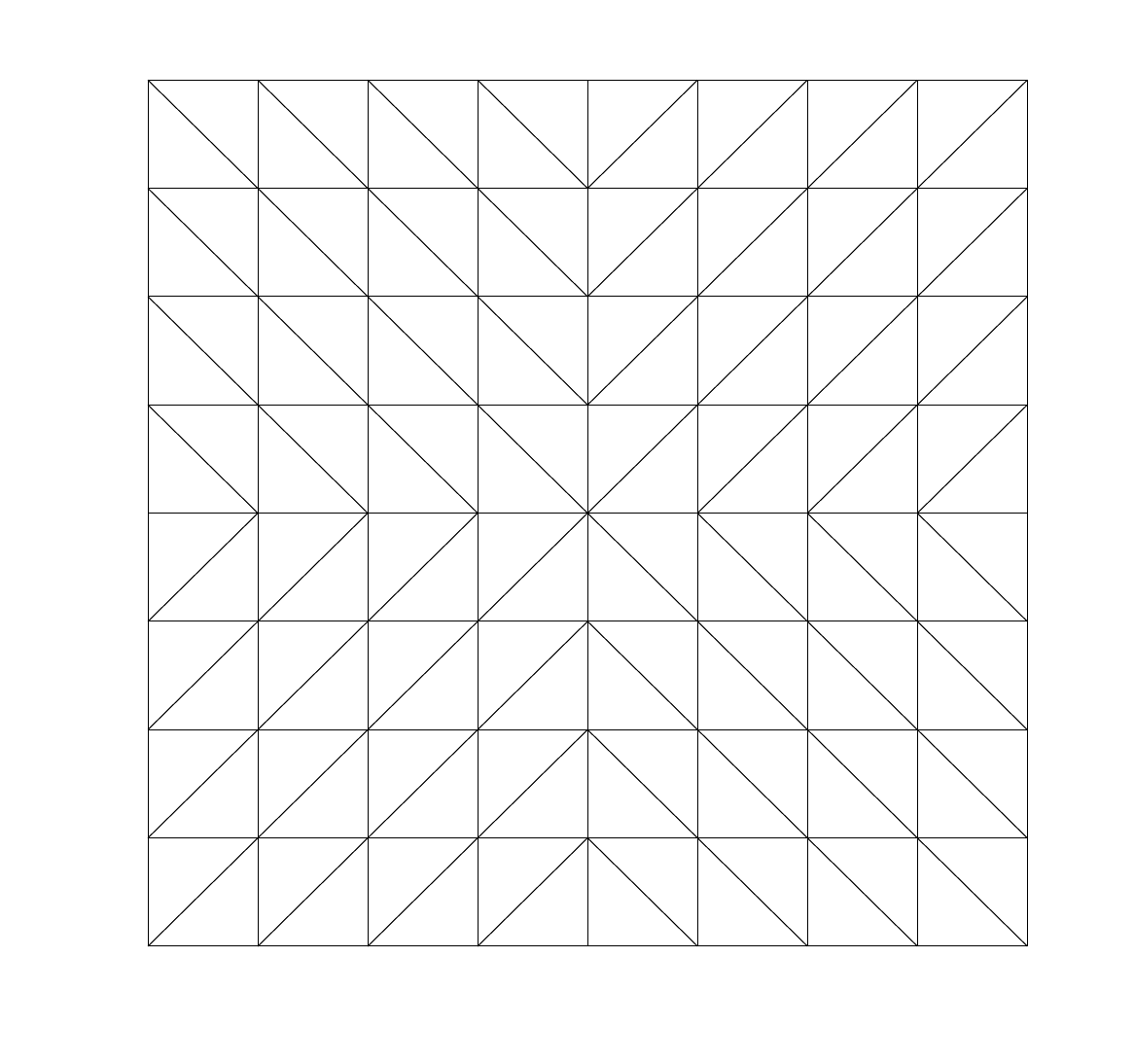}
			\centering\includegraphics[height=4.2cm, width=4.2cm]{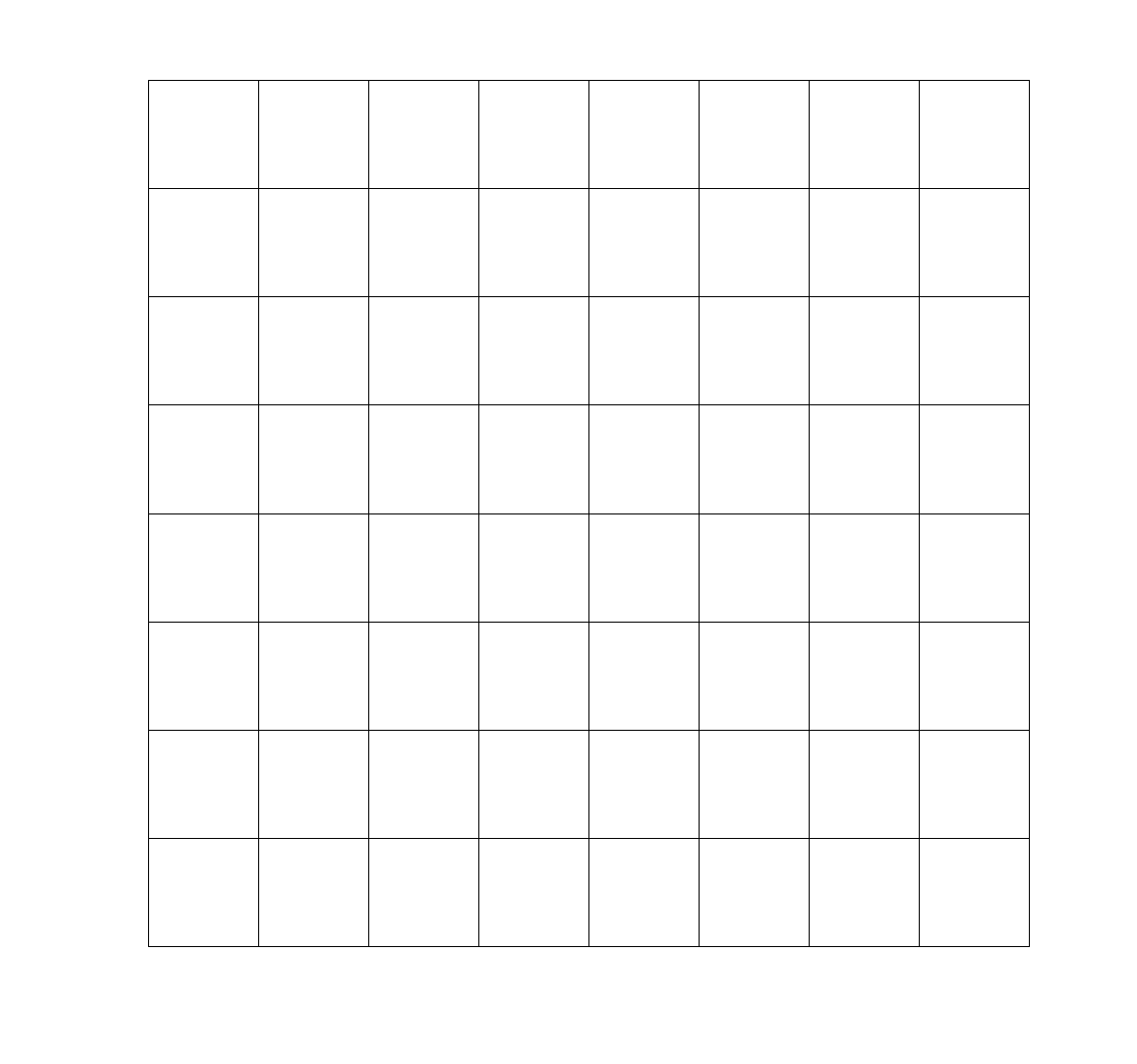}\\
			\centering\includegraphics[height=4.2cm, width=4.2cm]{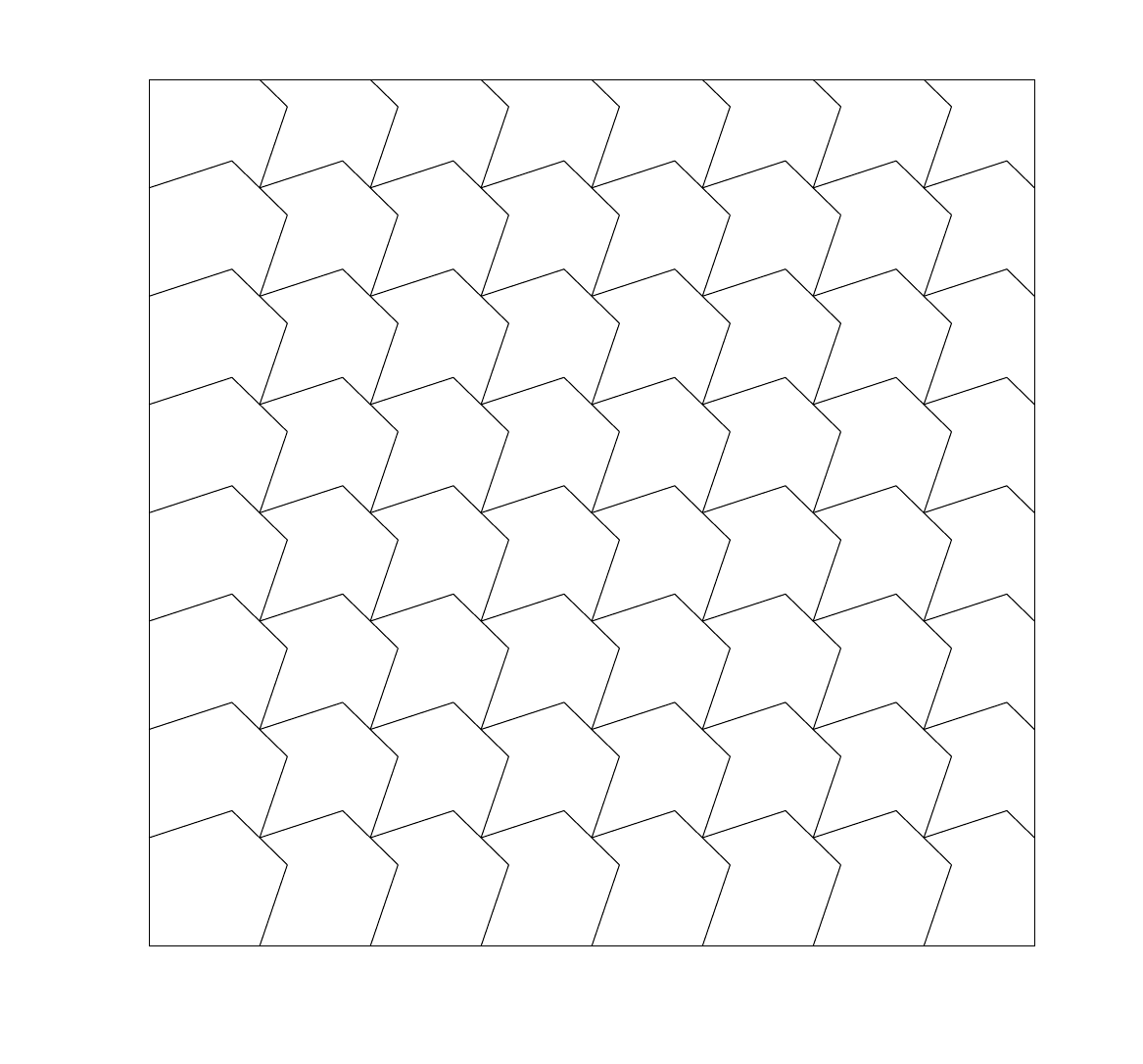}
			\centering\includegraphics[height=4.2cm, width=4.2cm]{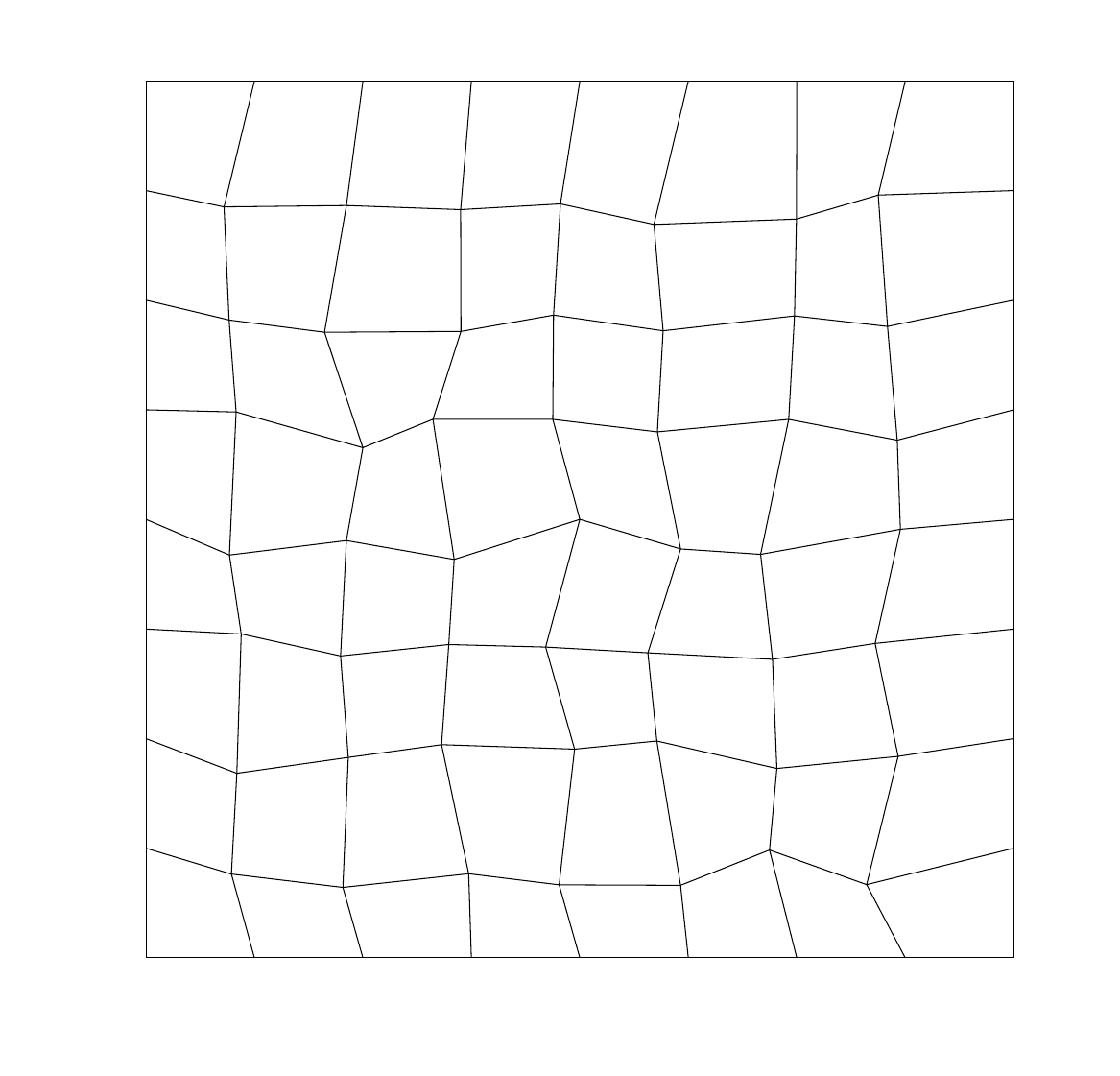}
         \end{minipage}

			\caption{ Sample meshes: $\CT_{h}^{1}$ top left), $\CT_{h}^{2}$ (top middle), $\CT_{h}^{3}$ (top right),$\CT_{h}^{4}$ (bottom left), $\CT_{h}^{5}$ (bottom middle) and  $\CT_{h}^{6}$ (bottom right).}
		\label{FIG:Meshes}
	\end{center}
\end{figure}

With these meshes, we compute the convergence rates for each of the unknowns of our problem. These results are reported in Table \ref{TABLA:1}.
\begin{table}[H]
\begin{center}
\begin{footnotesize}
\caption{Test 1: Convergence history of the velocity, pressure, and pseudostress for different polygonal meshes on $\O=(0,1)^2$.}
\begin{tabular}{|c|c|c|c|c|c|c|c|c|}
\hline
\multicolumn{7}{ |c| }{$\CT_{h}^1$}\\
\hline
 $h$ & $\texttt{e}(\bu)$ & $r(\bu)$ & $\texttt{e}(\bsig)$ &$r(\bsig)$  &$\texttt{e}(p)$ & $r(p)$ \\
\hline
    0.0354 &    0.0433&    --         &  0.0615 &    --         &   0.1270& -- \\
    0.0283 &    0.0347& 1.0006 &  0.0492 &1.0021  &  0.1011 &  1.0203\\
    0.0236 &    0.0289& 1.0004 &  0.0410 &1.0015  &  0.0841 &  1.0146\\
    0.0202 &    0.0247& 1.0003 &  0.0351 &1.0012  &  0.0719 &  1.0110\\
    0.0177 &    0.0217& 1.0003 &  0.0307 &1.0009  &  0.0629 &  1.0086\\
    0.0157 &   0.0192 &1.0002  & 0.0273 &1.0007   & 0.0558  & 1.0069\\
    0.0141 &   0.0173 &1.0002  & 0.0246 &1.0006   & 0.0502  & 1.0057\\
     \hline
\multicolumn{7}{ |c| }{$\CT_{h}^2$}\\
\hline      
    0.0667  &   0.0706 &    --        &   0.0807 &     --           & 0.0363 & --\\
    0.0500  &   0.0530 &  0.9968&   0.0606 &   0.9988 &   0.0263 &  1.1203\\
    0.0400  &   0.0424 &  0.9981&   0.0485 &   0.9993 &   0.0207 &  1.0773\\
    0.0333  &   0.0353 &  0.9987&   0.0404 &   0.9995 &   0.0171 &  1.0535\\
    0.0286  &   0.0303 &  0.9991&   0.0346 &   0.9997 &   0.0145 &  1.0391\\
    0.0250  &   0.0265 &  0.9993&   0.0303 &   0.9998 &   0.0127 &  1.0297\\
    0.0222  &   0.0236 &  0.9995&   0.0269 &   0.9998 &   0.0112 &  1.0234\\
   \hline
\multicolumn{7}{ |c| }{$\CT_{h}^3$}\\
\hline
    0.0745 &    0.0729  &       --        & 0.0832 &     --         &    0.0393& --\\
    0.0559 &    0.0547  &  0.9982  & 0.0624 &  0.9989 &    0.0283&  1.1382\\
    0.0447 &    0.0437  &  0.9992  & 0.0499 &  0.9991 &   0.0222 &  1.0990\\
    0.0373 &    0.0365  &  0.9997  & 0.0416 &  0.9993 &   0.0182 &  1.0761\\
    0.0319 &    0.0312  &  0.9999  & 0.0357 &  0.9994 &   0.0155 &  1.0615\\
    0.0280 &    0.0273  &  1.0000  & 0.0312 &  0.9994 &   0.0134 &  1.0515\\
    0.0248 &    0.0243  &  1.0001  & 0.0278 &  0.9995 &   0.0119 & 1.0442\\
 \hline
\multicolumn{7}{ |c| }{$\CT_{h}^4$}\\
\hline
    0.0815  &  0.0727 &     --       & 0.0837 &      --    & 0.0514 &--\\
    0.0629  &  0.0547 & 1.1008  & 0.0628  &1.1106 & 0.0370 & 1.2740\\
    0.0499  &  0.0438 & 0.9563  & 0.0503  &0.9617 & 0.0290 & 1.0423\\
    0.0418  &  0.0364 & 1.0517  & 0.0418  &1.0361 & 0.0242 & 1.0430\\
    0.0363  &  0.0312 & 1.0863  & 0.0359  &1.0889 & 0.0209 & 1.0113\\
    0.0313  &  0.0273 & 0.8908  & 0.0314  &0.8933 & 0.0181 & 0.9760\\
    0.0283  &  0.0243 & 1.1719  & 0.0279  &1.1571 & 0.0161 & 1.1488\\
 \hline
\end{tabular}
\label{TABLA:1}
\end{footnotesize}
\end{center}
\end{table}

For the lowest order of approximation that we are considering, we observe from table \ref{TABLA:1} that each of the unknowns is approximated with convergence rate $\mathcal{O}(h)$, independent of the polygonal mesh under consideration, confirming the robustness of the virtual element scheme with respect to the geometry involved on the meshes. Since the manufactured solutions are smooth on $\O$, these optimal orders are expected. The error curves on Figure \ref{fig:curve_plots_square} gives us a visual appreciation of the convergence history for the velocity, pressure, and pseudostress, computed with the meshes on Table \ref{TABLA:1}.

\begin{figure}[H]
	\begin{center}
		\begin{minipage}{13cm}
			\centering\includegraphics[height=5.2cm, width=5.2cm]{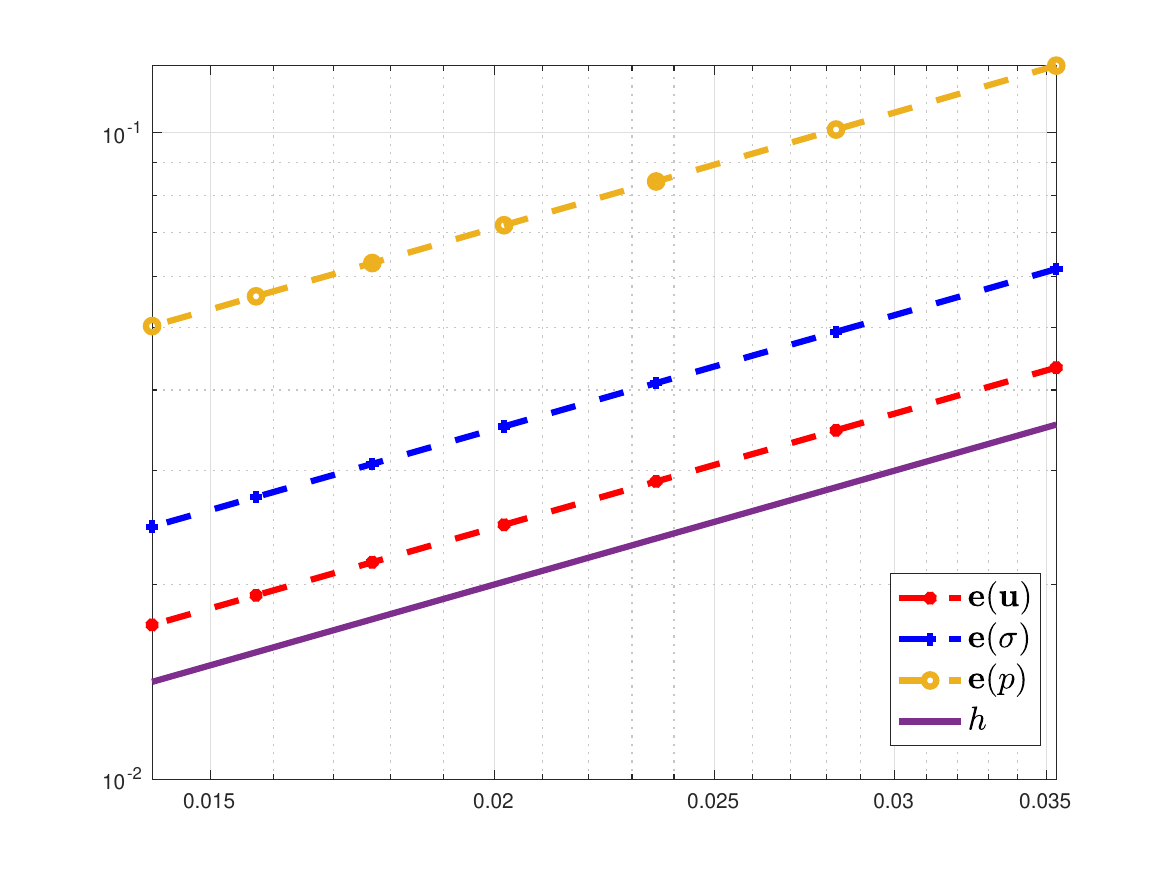}\hspace{1.4cm}
			\centering\includegraphics[height=5.2cm, width=5.2cm]{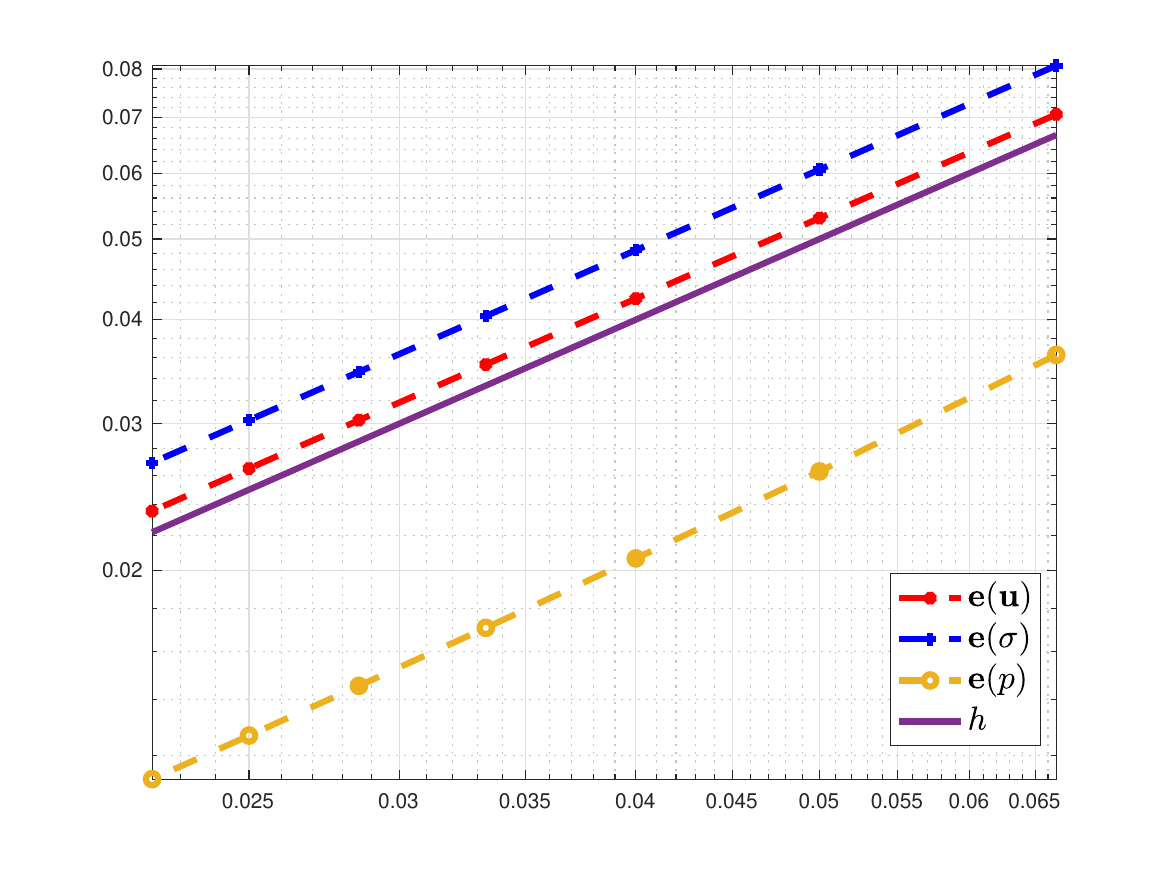}\\
			\centering\includegraphics[height=5.2cm, width=5.2cm]{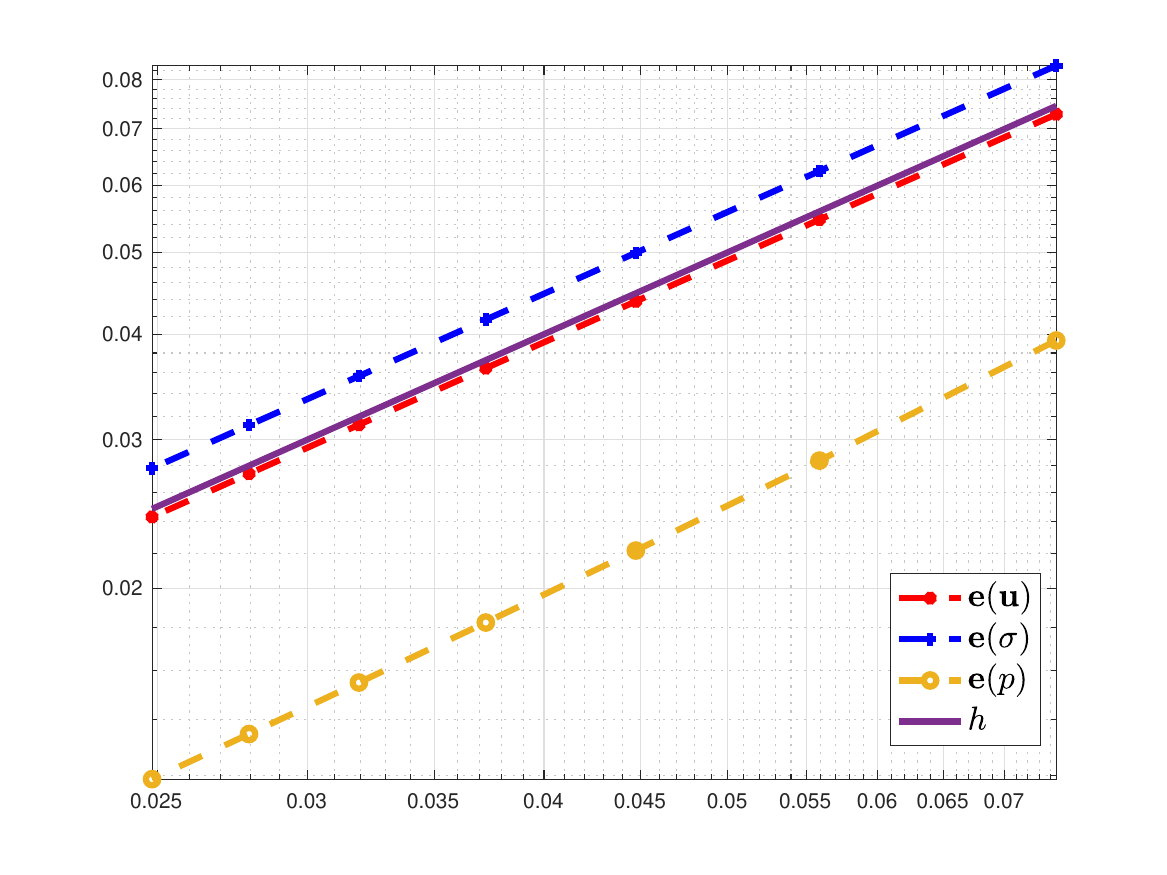}\hspace{1.4cm}
       	\centering\includegraphics[height=5.2cm, width=5.2cm]{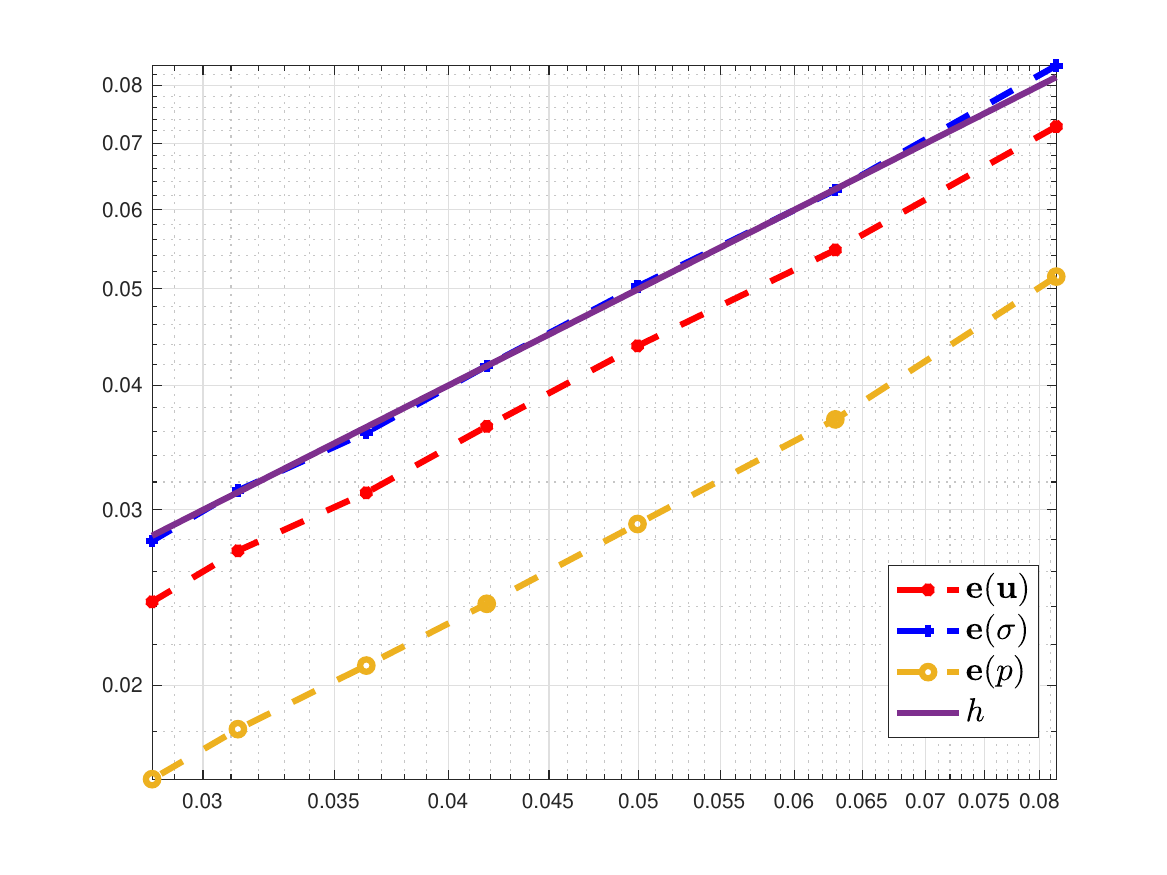}
         \end{minipage}
     				\caption{Test 1. Error curves for the velocity, pseudostress, and pressure, computed with meshes $\CT_{h}^1$, $\CT_{h}^2$, $\CT_{h}^3$, and $\CT_h^4$, respectively. We observe the optimal order of convergence on each case: in red we present the error curve of the velocity,  in blue for the pseudostress, and in yellow the error of the pressure.}
		\label{fig:errortriangulos}
	\end{center}
\end{figure}
We also present plots for the magnitude of the velocity, the pressure fluctuation, and the components of the pseudostress tensor, all computed with mesh $\CT_h^3$ of non-convex polygons.
\begin{figure}[H]
	\begin{center}
		\begin{minipage}{13cm}
			\centering\includegraphics[height=5.2cm, width=5.2cm]{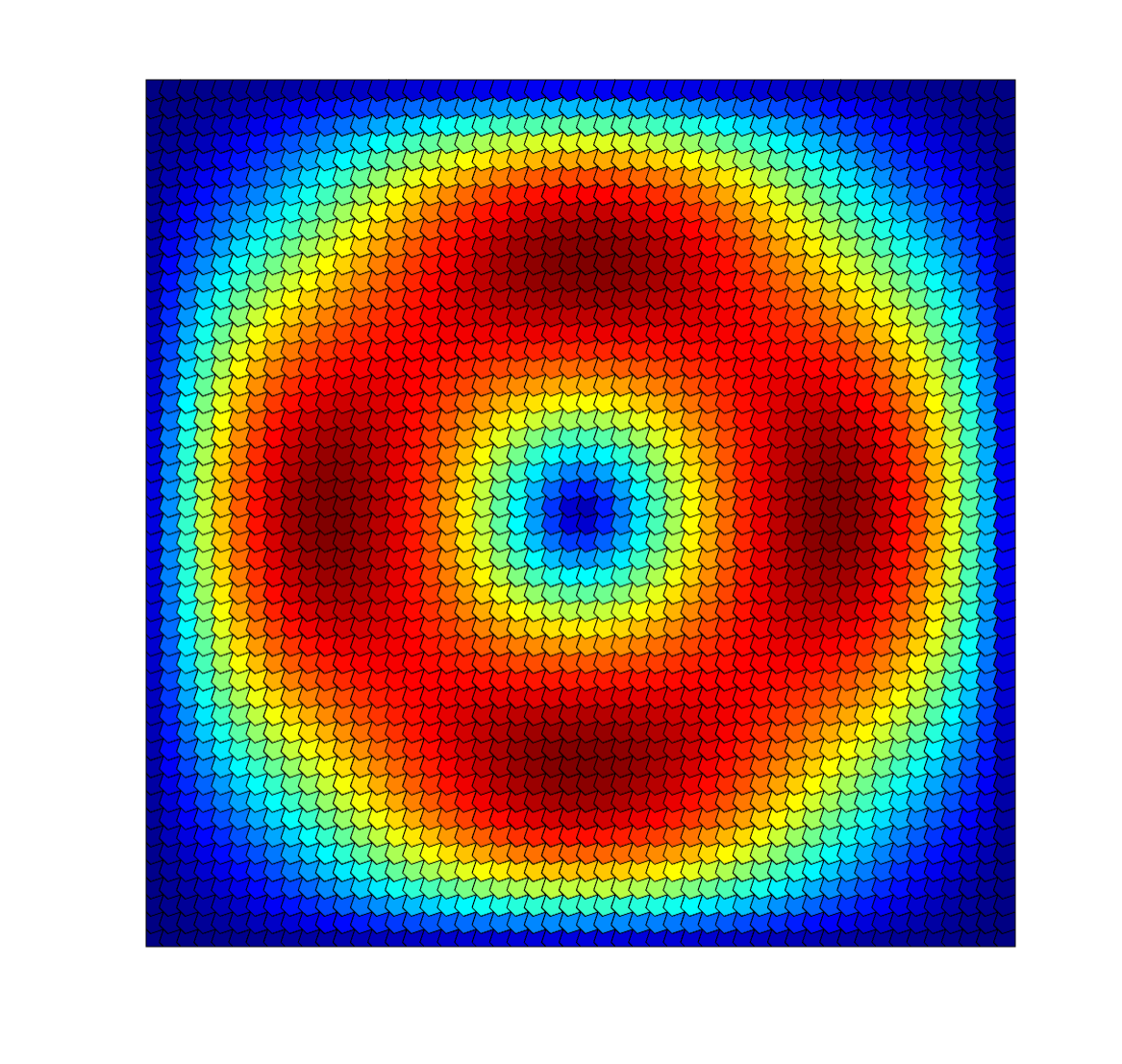}\hspace{1.4cm}
			\centering\includegraphics[height=5.2cm, width=5.2cm]{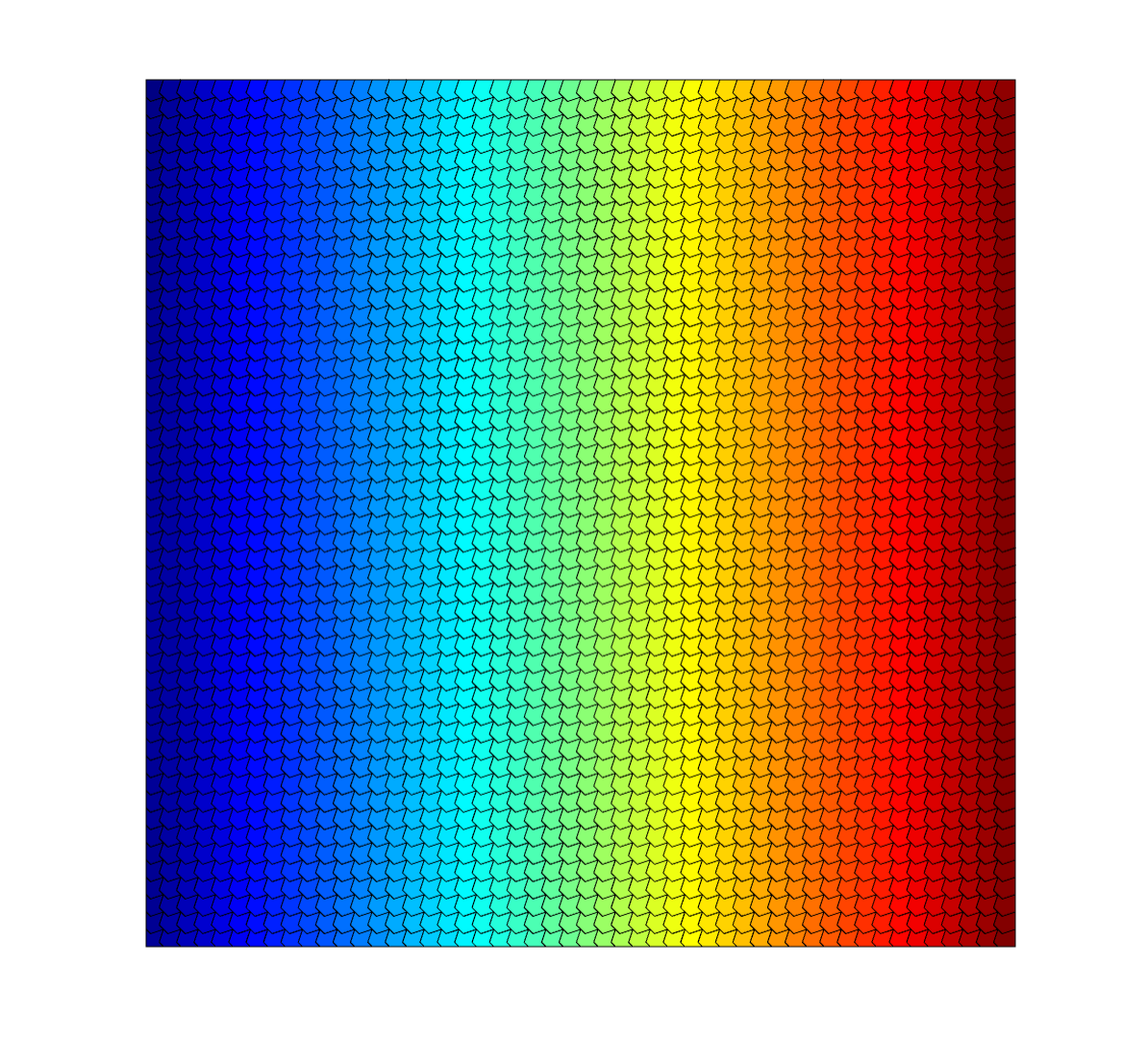}\\
			         \end{minipage}
     				\caption{Test 1: Left: plot of the computed velocity magnitude $\bu_h$ with mesh $\CT_h^3$. Right: plot of the  computed pressure fluctuation  $p_h$ with mesh $\CT_h^3$.}
		\label{fig:curve_plots_square}
	\end{center}
\end{figure}
\begin{figure}[H]
	\begin{center}
		\begin{minipage}{13cm}
			\centering\includegraphics[height=5.2cm, width=5.2cm]{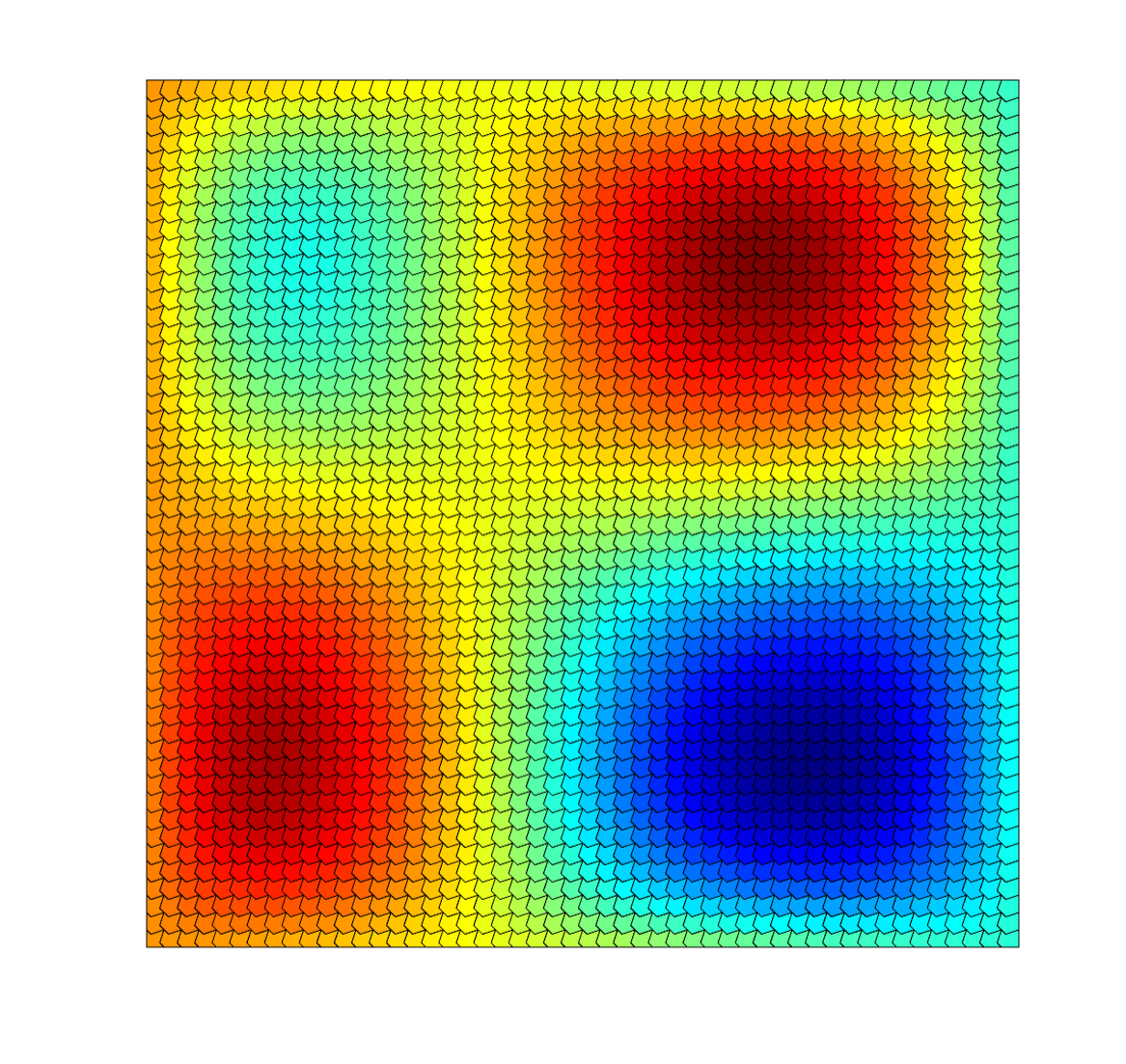}\hspace{1.4cm}
       	\centering\includegraphics[height=5.2cm, width=5.2cm]{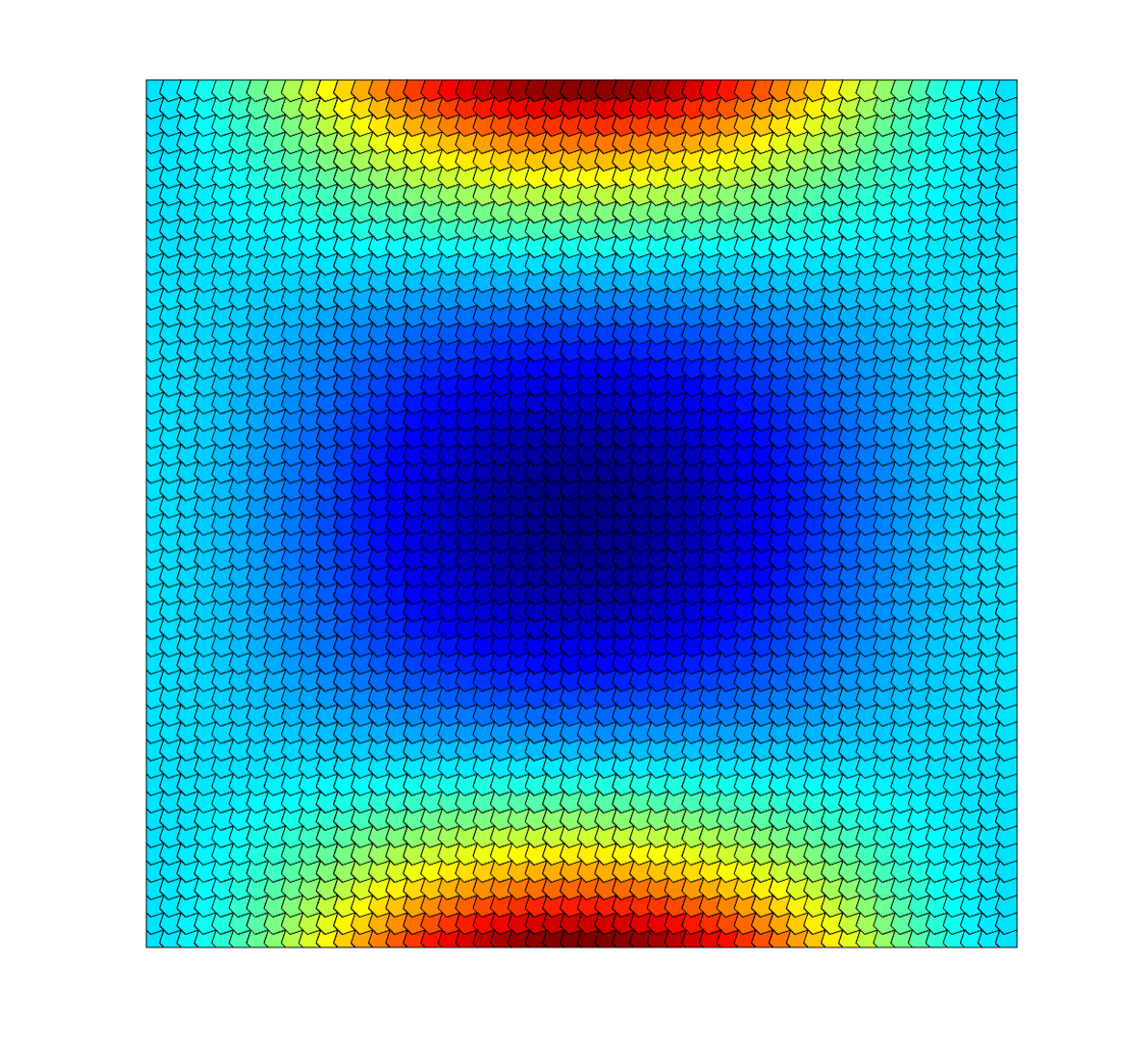}
	\centering\includegraphics[height=5.2cm, width=5.2cm]{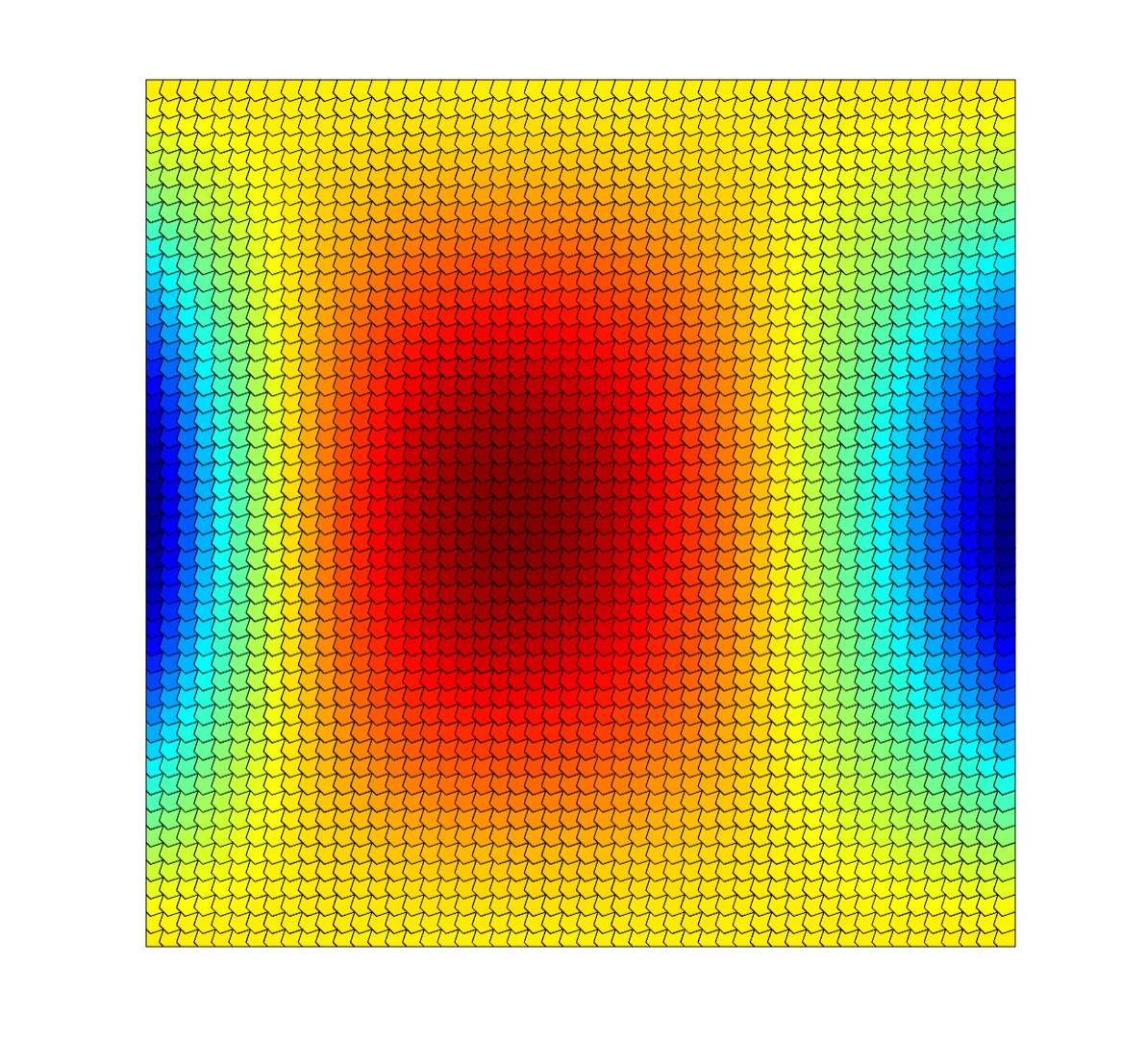}\hspace{1.4cm}
       	\centering\includegraphics[height=5.2cm, width=5.2cm]{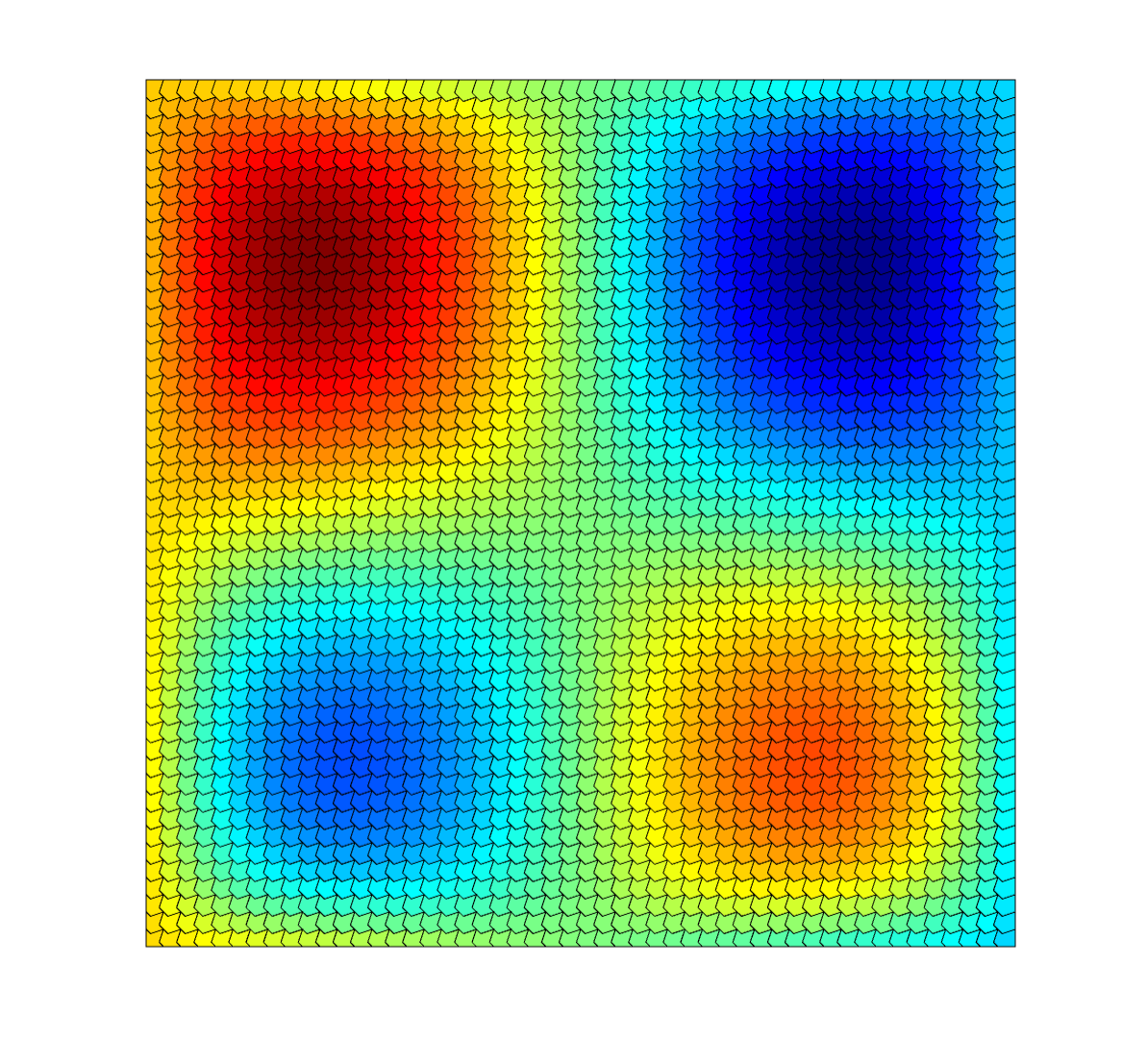}

         \end{minipage}
     				\caption{Test 1: Computed pseudostress tensor components with $\CT_h^4$. Top row,  from left to right: components $\boldsymbol{\sigma}_{h,11}$ and  $\boldsymbol{\sigma}_{h,12}$. Bottom row,  from left to right:  $\boldsymbol{\sigma}_{h,21}$ and $\boldsymbol{\sigma}_{h,22}$.}
		\label{fig:curve_plots_square}
	\end{center}
\end{figure}
\subsection{Convection-dominated regime} 
\label{test:Test 2}
In the following test, we analyze the behavior of our method in the regime where the physical parameter $\nu$ tends to small values. To this end, we consider the following problem configuration:
$\O:=(0,1)^2$ with the boundary condition $\bu=0$ on $\partial\O$. We consider the following physical constants:  $\kappa=1$, $\boldsymbol{\beta}:=(1,1)^t$.  The forcing term $\boldsymbol{f}$ is chosen in such a way that the exact velocity and pressure solve  problem  \eqref{eq:Oseen_system}. These are given by
$$\bu(x,y)=\left(2x^2y(2y-1)(x-1)^2(y-1),-2xy^2(2x-1)(x-1)(y-1)^2\right)\qquad\text{and }$$ $$p(x,y)=2\cos(x)\sin(y)-2\sin(1)(1-\cos(1).$$
We study the performance of the method for the following values of viscosity $\nu=\left\{1,10^{-1},10^{-2},10^{-3},10^{-4}\right\}$  on different mesh families. 
 For this example, we define the following relative error:
 $$\vertiii{\texttt{e}}_{\star}:=\dfrac{\vertiii{(\bu-\bu_h,\bsig-\bsig_h,p-p_h)}_{\star} }{\vertiii{(\bu,\bsig,p)}_{\star}},$$
 where $\vertiii{(\bu,\bsig , p)}_{\star}^2:=\|\bu\|_{0,\O}^2+\|\bsig\|_{0,\O}^2+\|p\|_{0,\O}^2$.
 
 \begin{figure}[H]
			\centering\includegraphics[height=4.2cm, width=4.2cm]{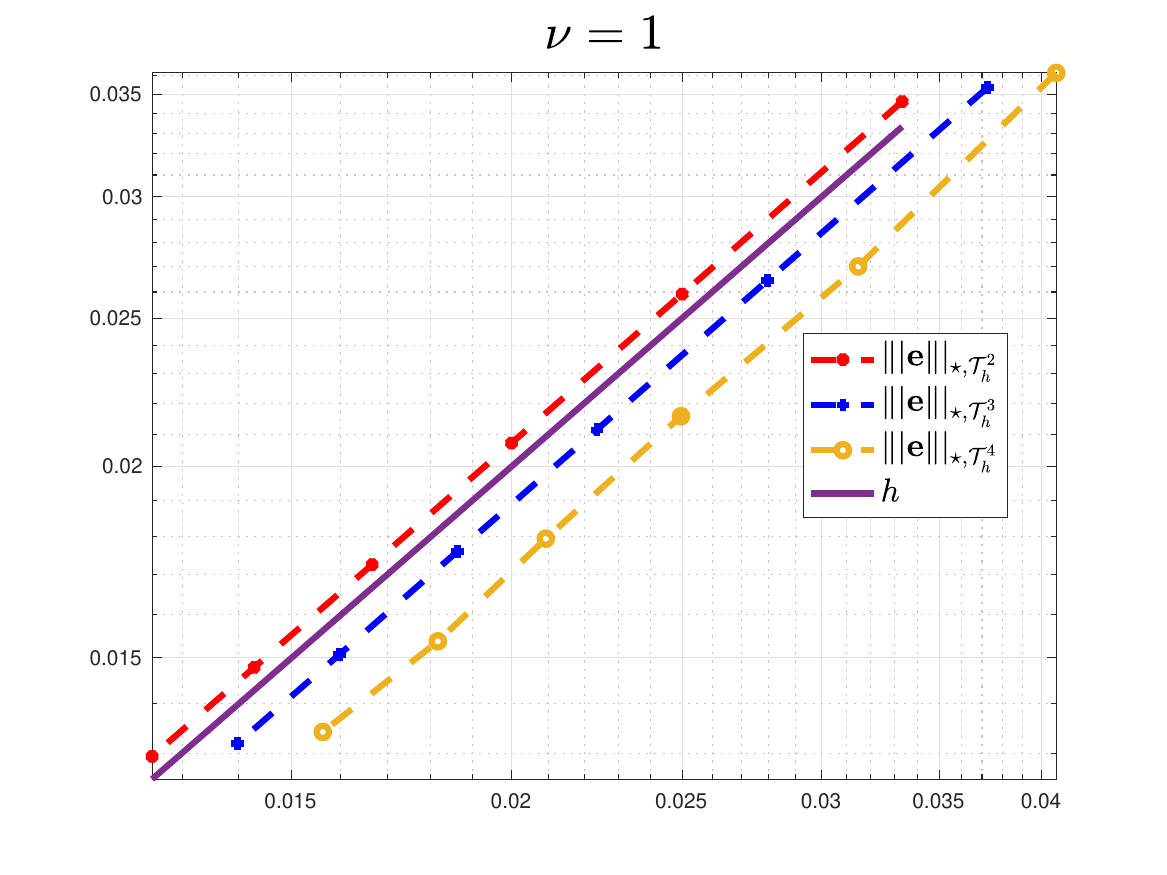}
       	\centering\includegraphics[height=4.2cm, width=4.2cm]{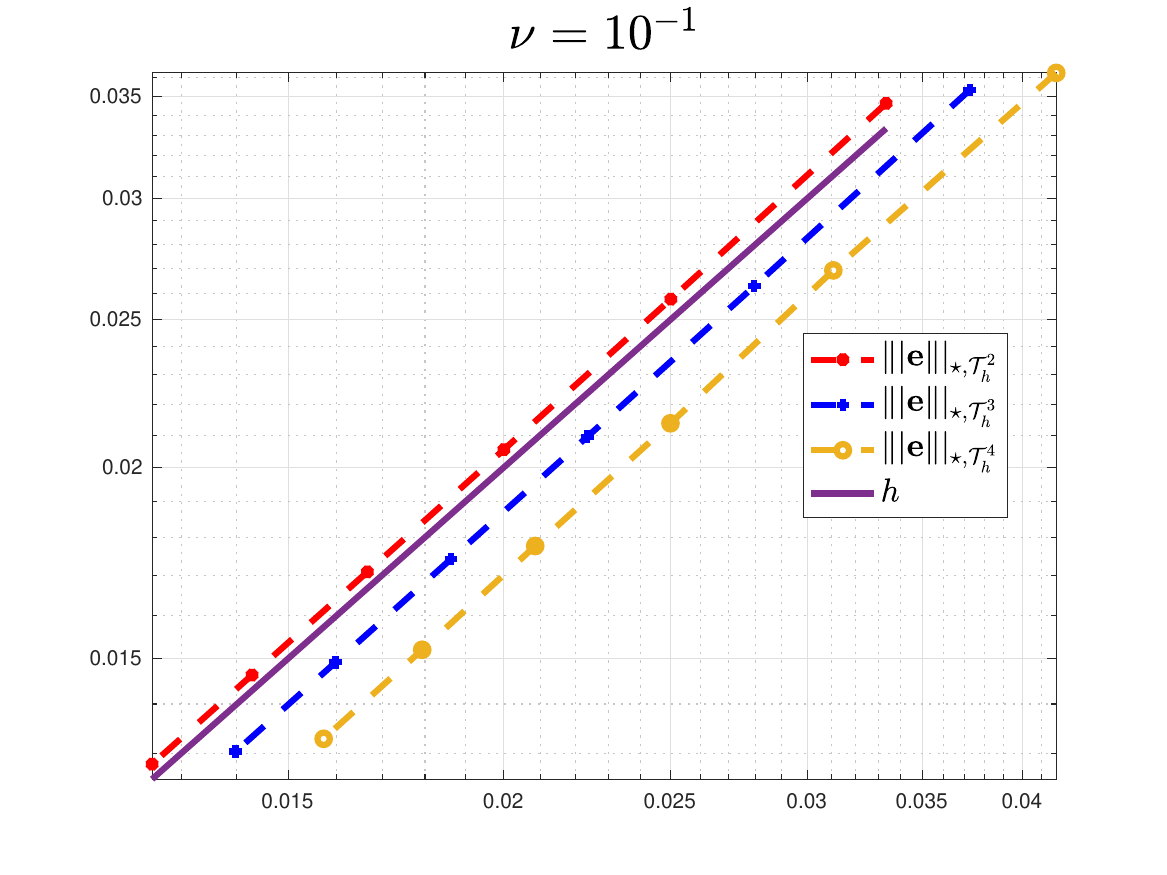}
	\centering\includegraphics[height=4.2cm, width=4.2cm]{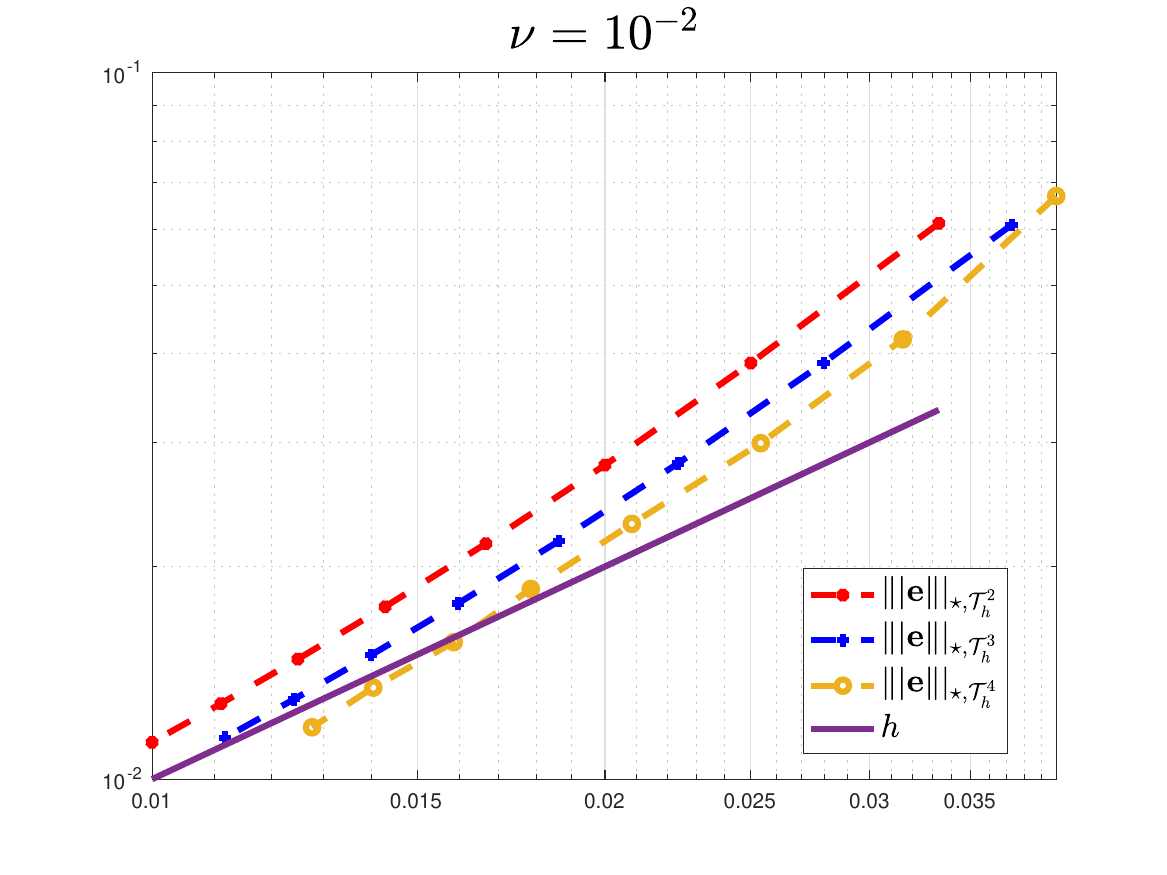}
	       	\centering\includegraphics[height=4.2cm, width=4.2cm]{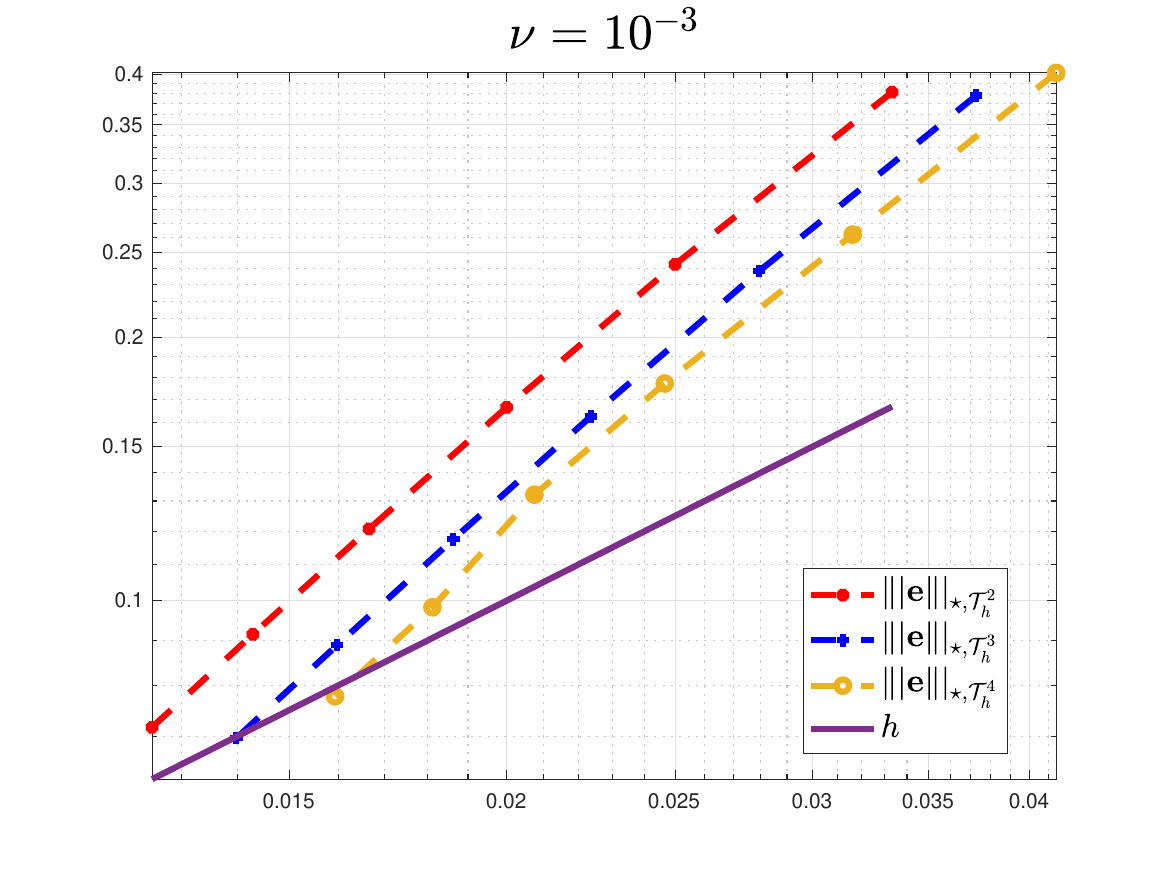}
	\centering\includegraphics[height=4.2cm, width=4.2cm]{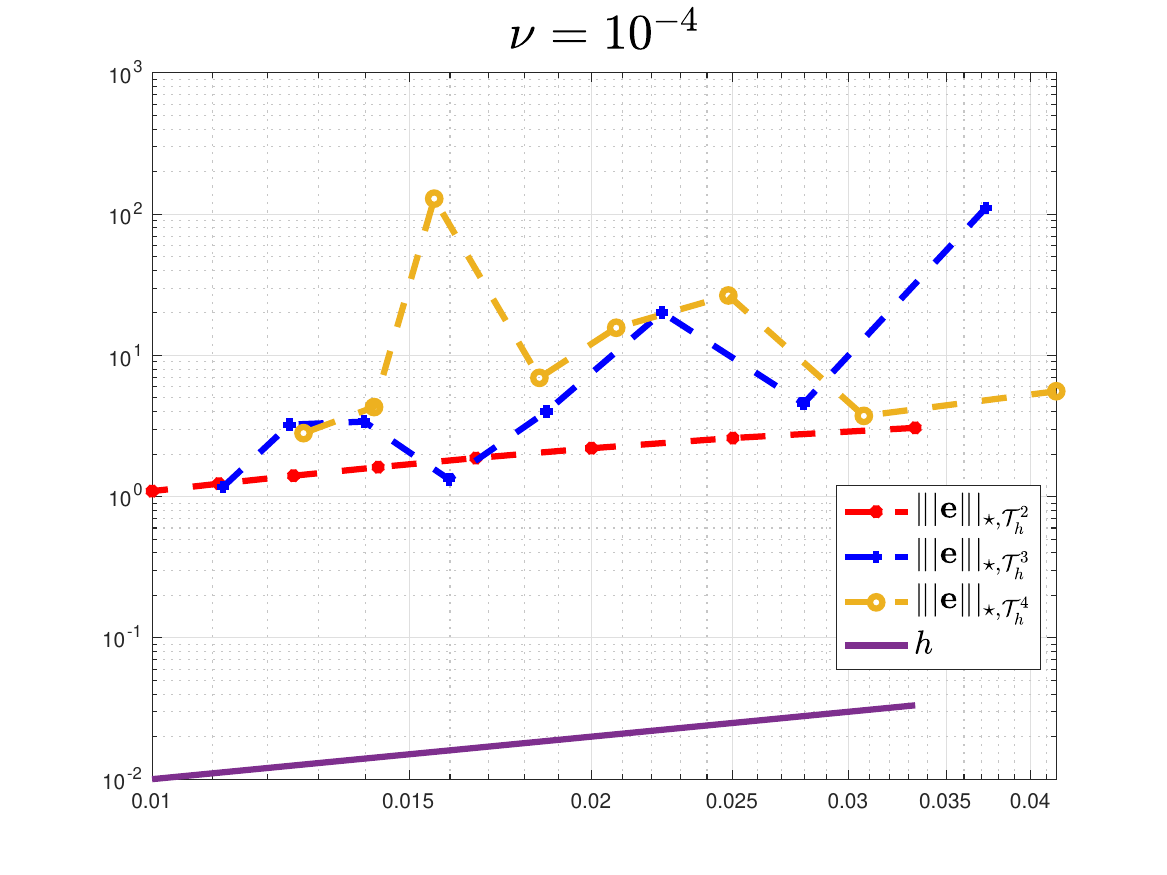}

     				\caption{Test 2: Error behavior for different values of $\nu$ and different mesh families. In red, the error using $\CT_h^2$, in blue the error computed with $\CT_h^3$, and in yellow, the error computed with $\CT_h^4$. }
		\label{fig:curve_test2_1}
\end{figure}
As shown in Figure \ref{fig:curve_test2_1}, when the parameter $\nu$ becomes small compared to $\kappa$ and $\|\boldsymbol{\beta}\|_{\infty,\O}$, the method exhibits unstable behavior. Moreover, it is observed to be more robust on square meshes. This motivates a further investigation of its performance on this type of mesh. In Figure \ref{fig:curve_test2},  we report the error curves  where a degradation of the convergence order is observed as $\nu$ decreases for the error $\vertiii{\texttt{e}}_{\star}$.
 \begin{figure}[H]
	\begin{center}
		\begin{minipage}{10cm}
			\centering\includegraphics[height=5.0cm, width=8.2cm]{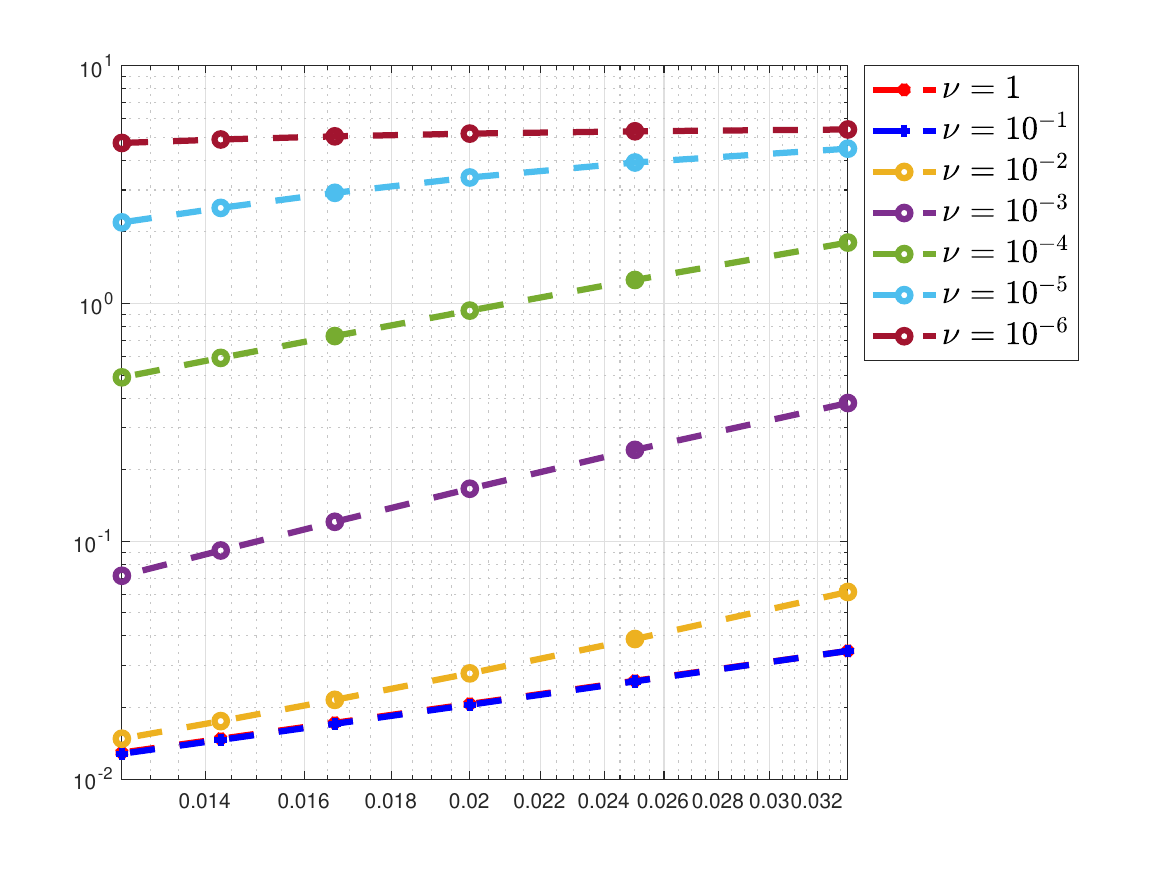}\hspace{1.4cm}

         \end{minipage}
     				\caption{Test 2: Error behavior for different values of $\nu$ with $\CT_h^2$.}
		\label{fig:curve_test2}
	\end{center}
\end{figure}
\subsection{Test 3: Permeability study}
In this test,  we now fix $\nu=1$  and $\boldsymbol{\beta}=(1,1)^{t}$, whereas  the parameter  $\kappa$ changes, considering the values 
$\kappa=\{1,10^2,10^4,10^6\}$, in order to assess the performance of our method with respect to this physical parameter. We use the same solution as in Test 2.
 \begin{figure}[H]
			\centering\includegraphics[height=4.2cm, width=4.2cm]{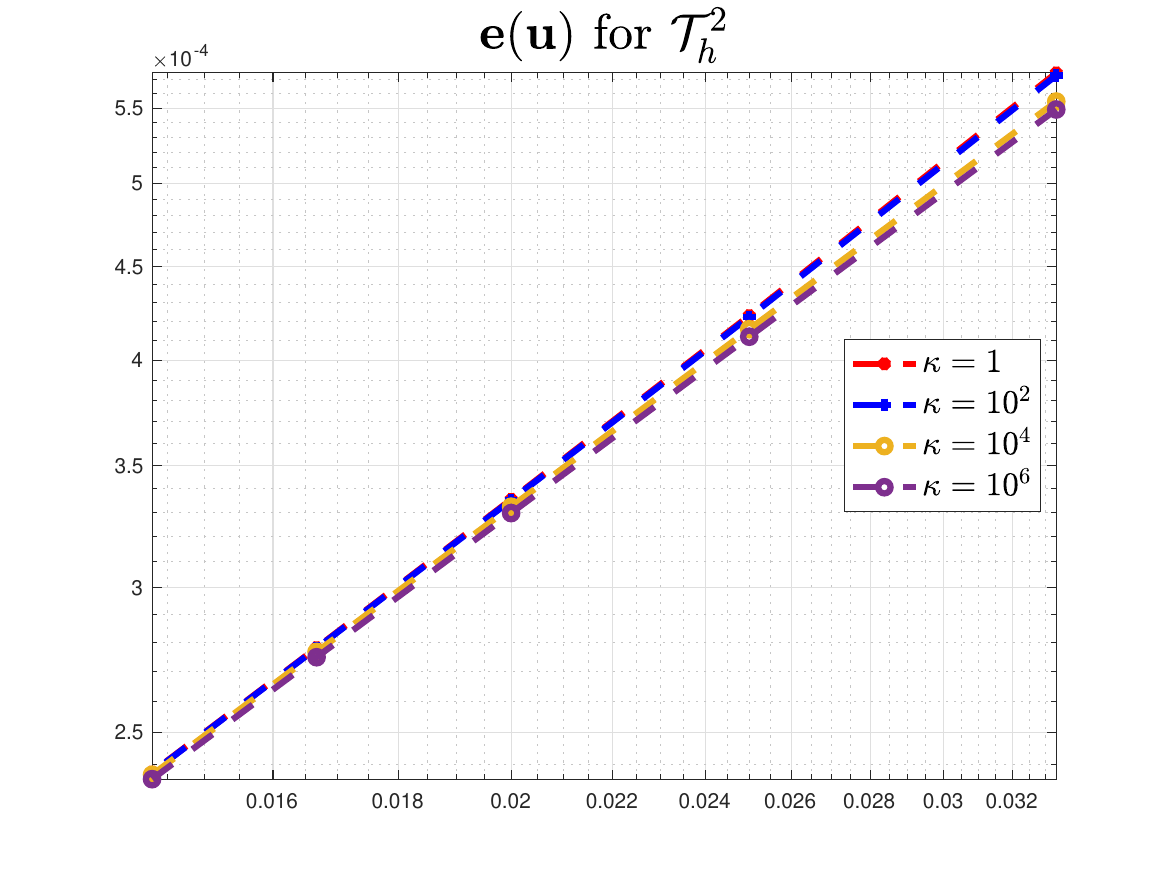}
       	\centering\includegraphics[height=4.2cm, width=4.2cm]{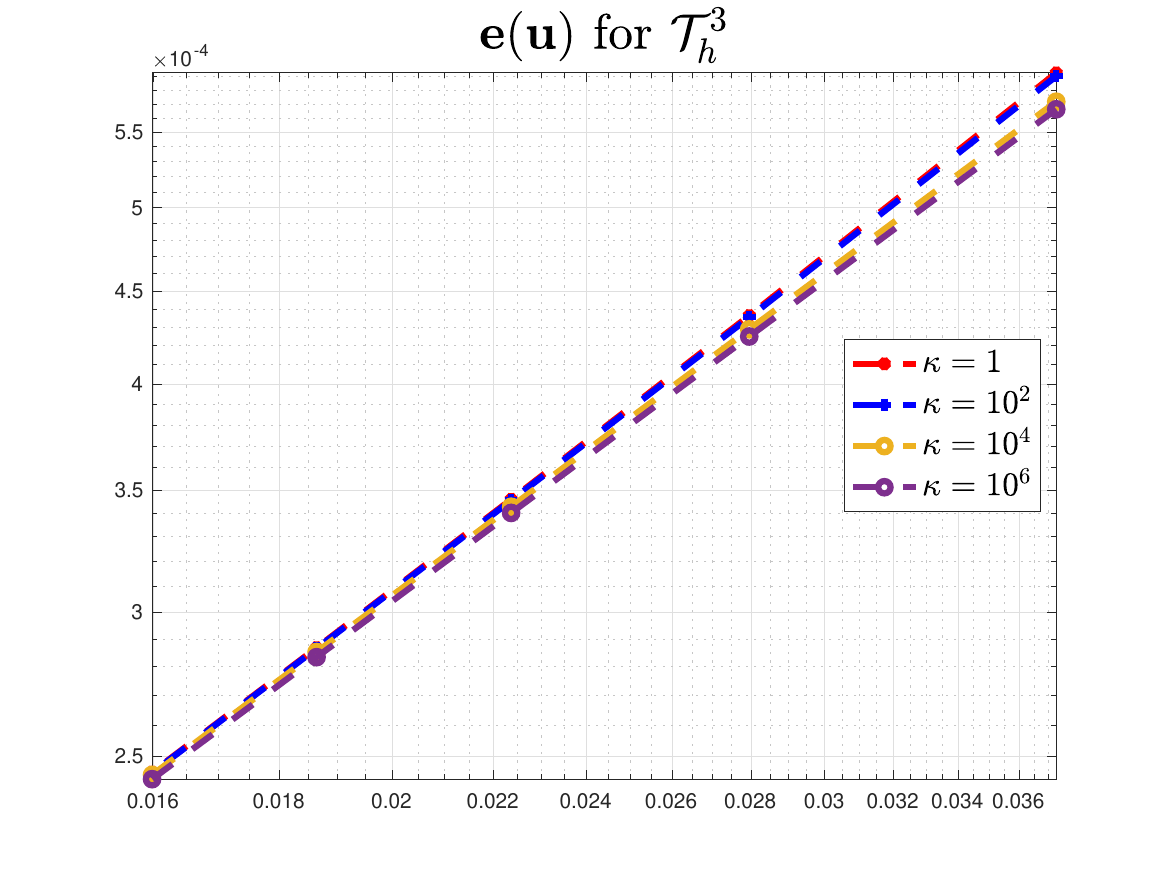}
	\centering\includegraphics[height=4.2cm, width=4.2cm]{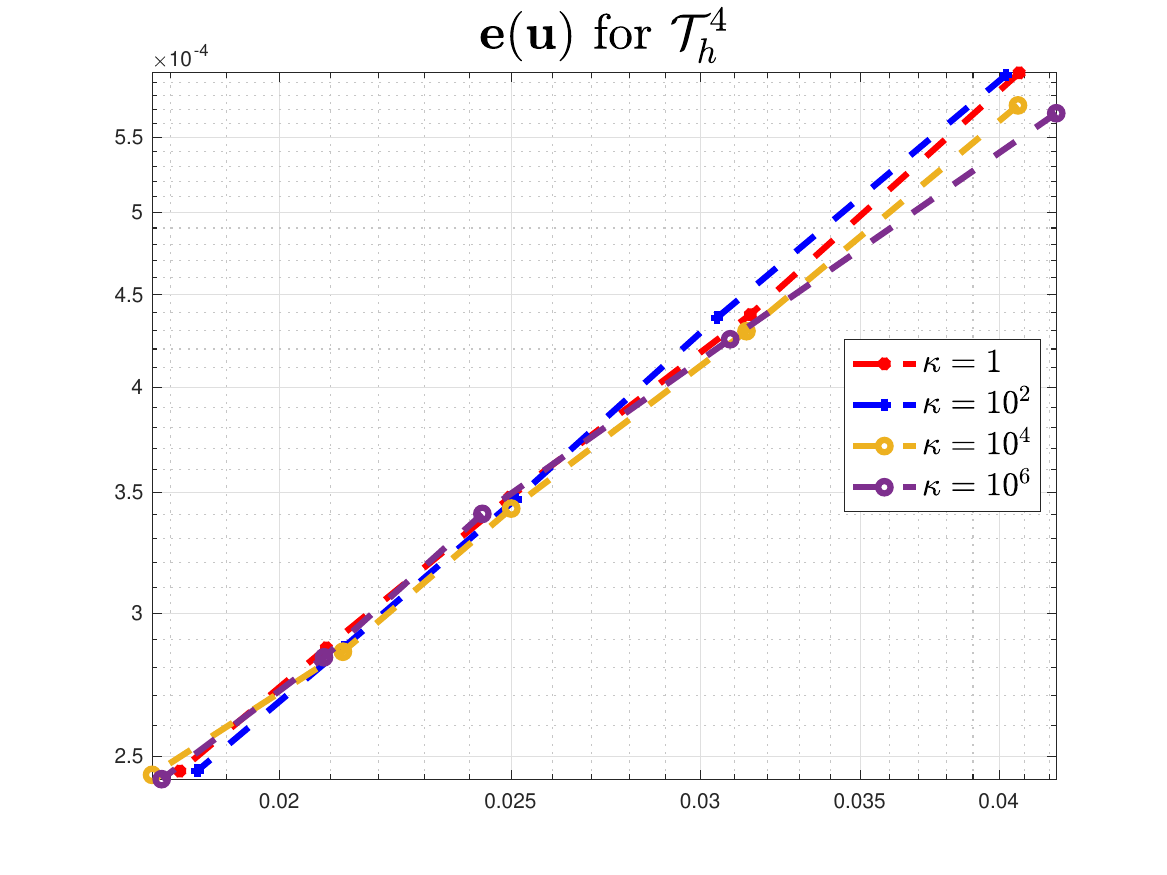}\\
	       	\centering\includegraphics[height=4.2cm, width=4.2cm]{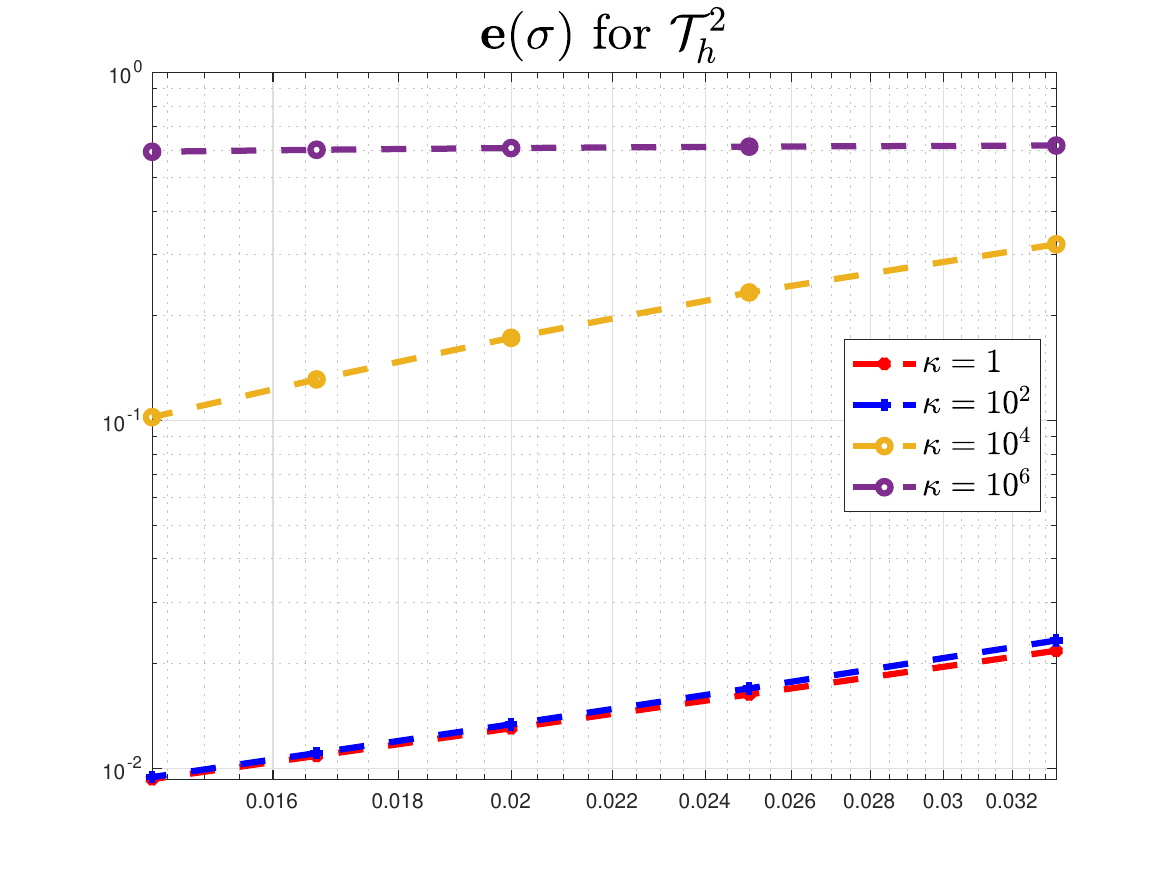}
	\centering\includegraphics[height=4.2cm, width=4.2cm]{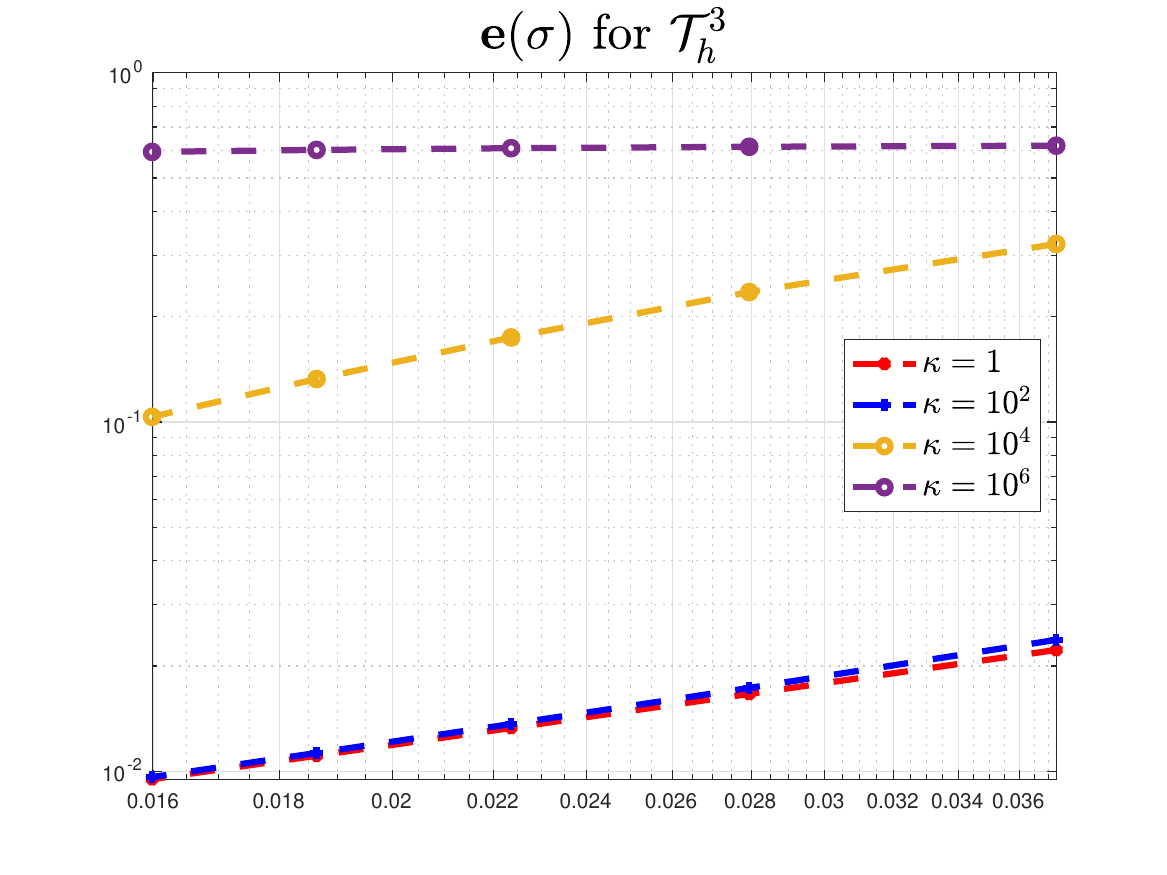}
	\centering\includegraphics[height=4.2cm, width=4.2cm]{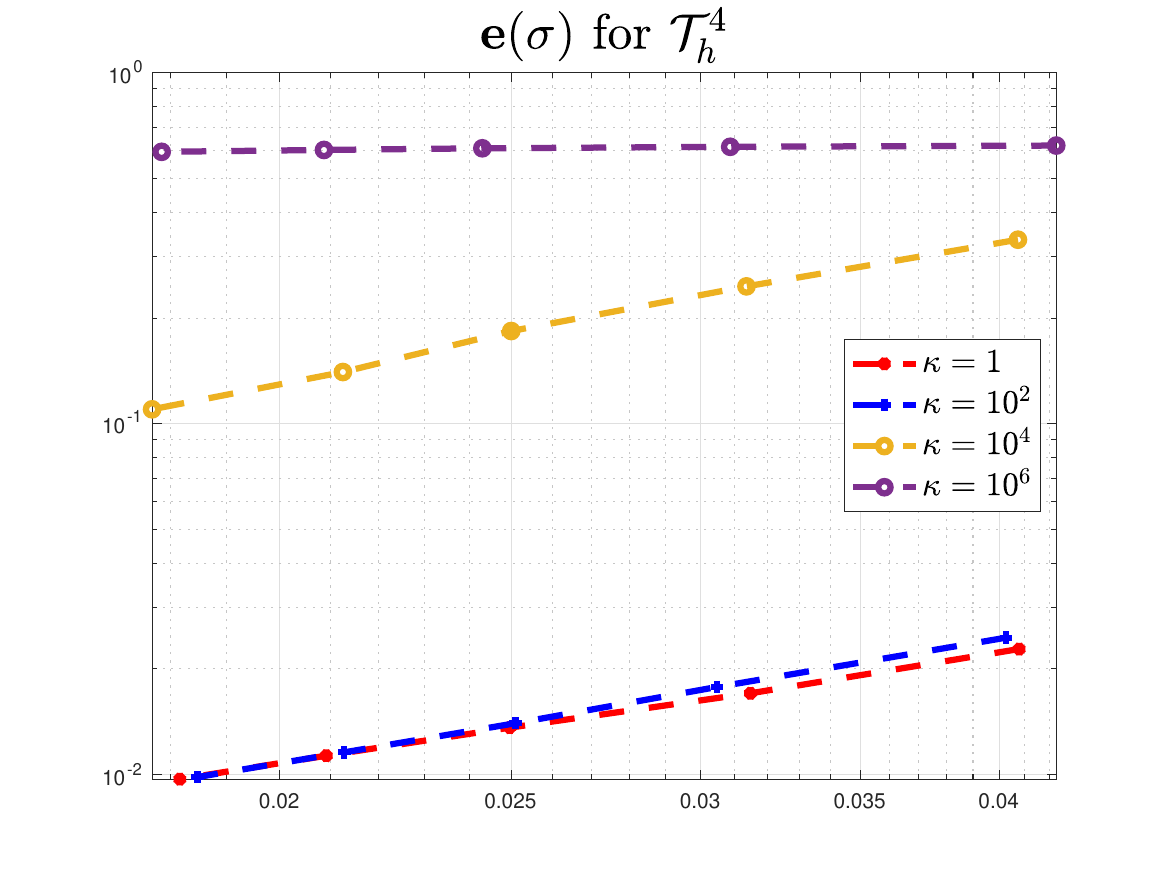}
	\centering\includegraphics[height=4.2cm, width=4.2cm]{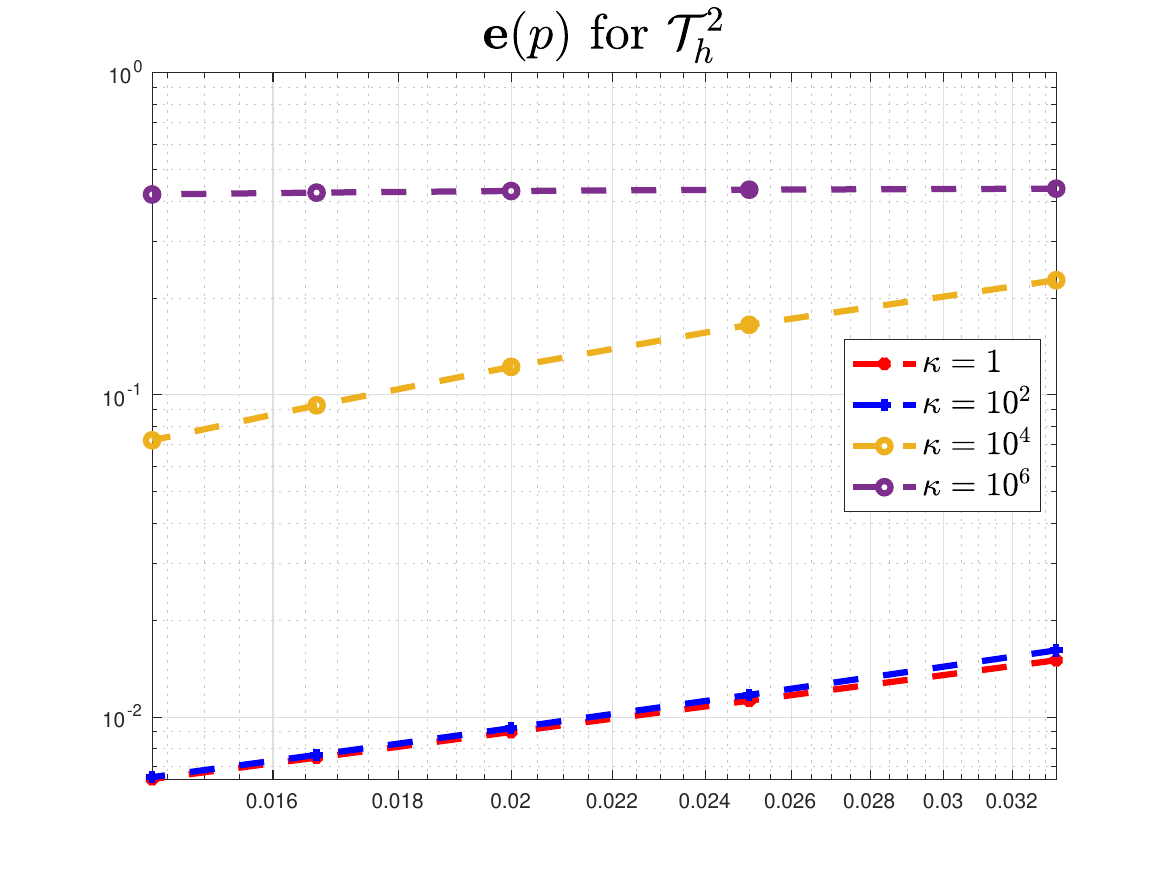}
	\centering\includegraphics[height=4.2cm, width=4.2cm]{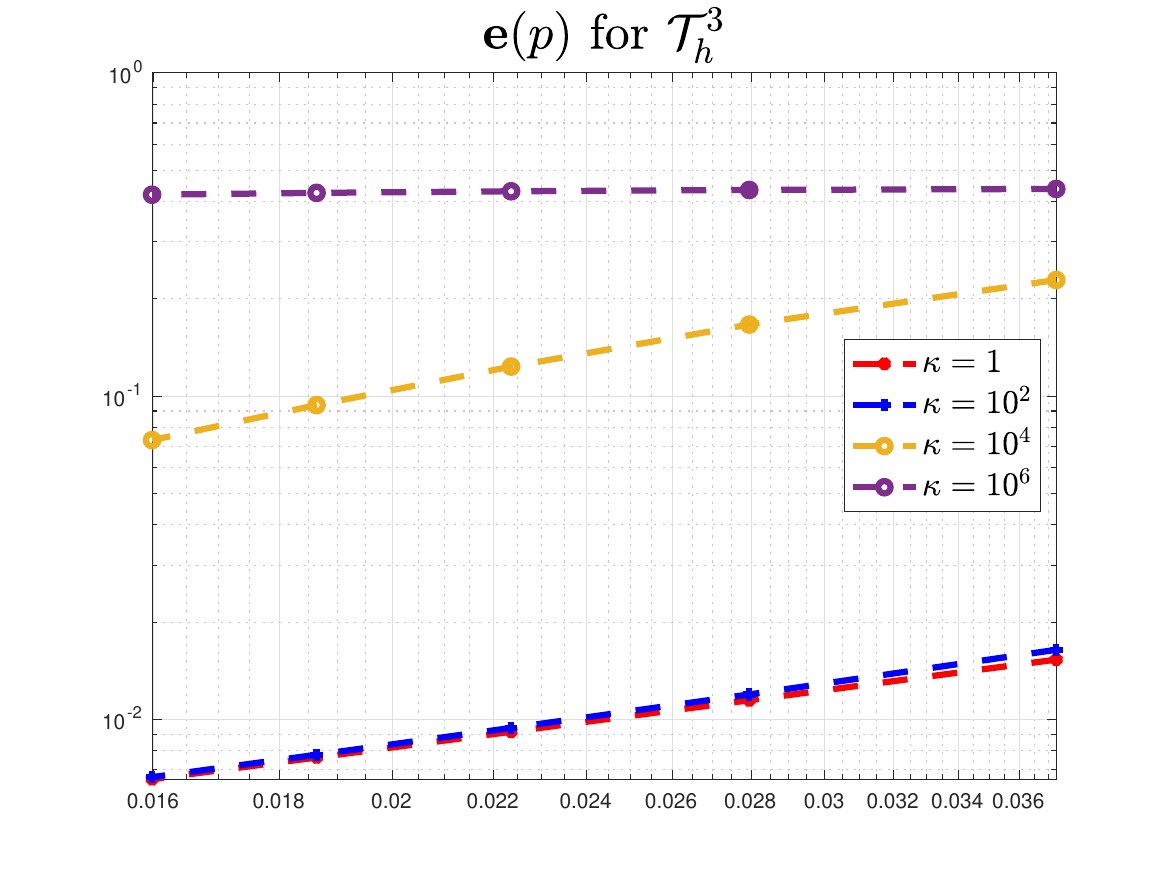}
	\centering\includegraphics[height=4.2cm, width=4.2cm]{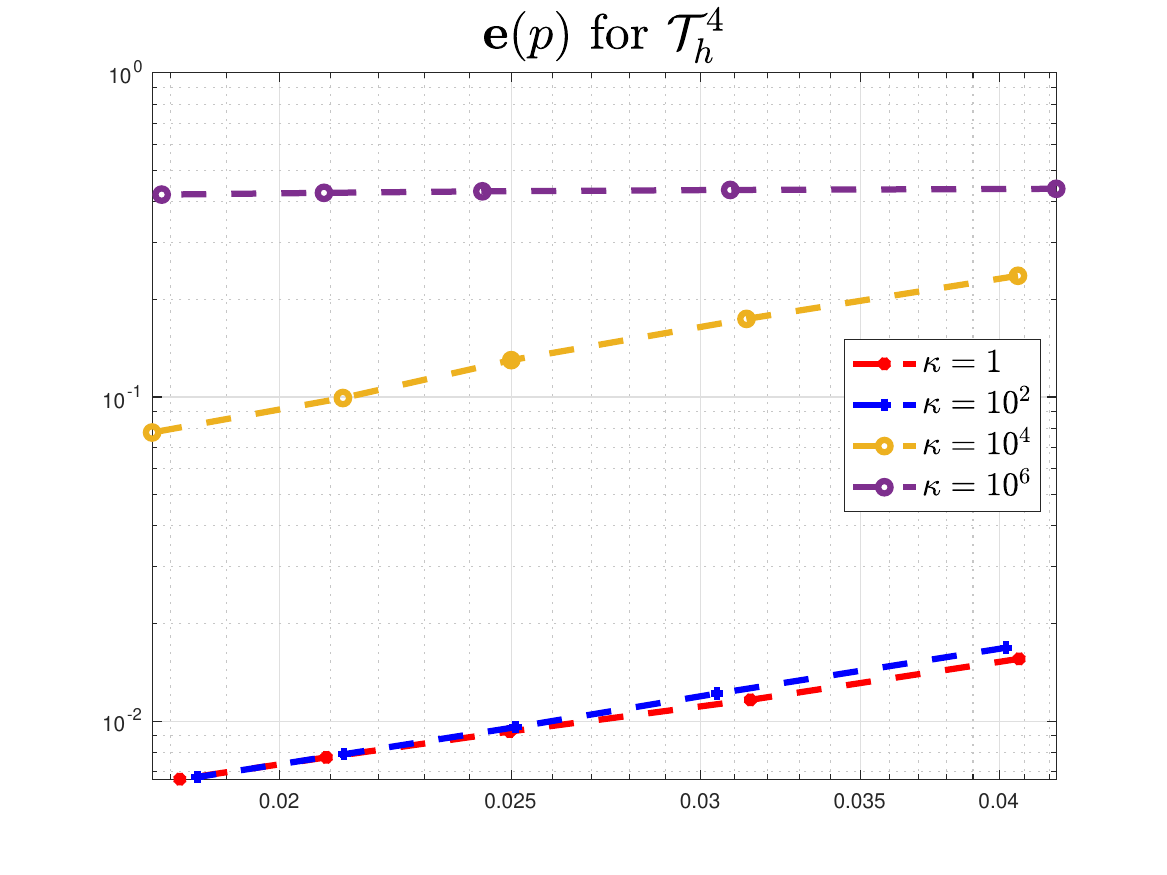}

     				\caption{Test 3: Errors in the velocity, pressure, and pseudo-stress tensor for different values of $\kappa$ and different mesh families.}
		\label{fig:curve_test3_1}
\end{figure}

Figure \ref{fig:curve_test3_1} reveals that the permeability $\kappa$ has influence on the error of the unknowns. In fact, we observe that the pressure and pseudostress tensor errors are affected when the permeability increases, where for $\kappa=1$ the error for these unknowns is precisely $\mathcal{O}(h)$. This is in some sense expectable, since from \eqref{eq:bound_sigma} and \eqref{eq:bound_u} we observe a dependency of $\kappa$ on the stability bound of  each unknown which is reflected in the numerical results. Contrary,  the velocity error shows to be independent  of the permeability value and the mesh,  remaining  stable  on each selected configuration of the test.

\subsection{Test 4: Non-homogeneous boundary condition}
For this test, let us consider the domain $\O:=(0,1)^2$ and the physical parameters $\nu=1$, $\kappa=1$ and $\boldsymbol{\beta}:=(1,1)^t$.  The forcing terms $\boldsymbol{f}$ acting on $\O$ and $\boldsymbol{g}$ on the boundary $\partial\O$ are chosen in such a way that the exact velocity and pressure solve  problem  \eqref{eq:Oseen_system}. These solutions are given by 
$$\bu(x,y)=\left(y,-x\right)^{t}\qquad\text{and }\quad p(x,y)={x^2+y^2}-\dfrac{2}{3}.$$

Let us mention that in particular, the boundary conditions that we will consider for this test is $\bu=\boldsymbol{g}$ on $\partial\O$. For this test, we have proved computationally  the method for different polygonal meshes as is  depicted in Figure \ref{FIG:Meshes}. Since the results do not present significant differences, we report in Figure \ref{fig:curve_test4_1} error curves for $\CT_h^3$.
 \begin{figure}[H]
	\centering\includegraphics[height=4.2cm, width=4.2cm]{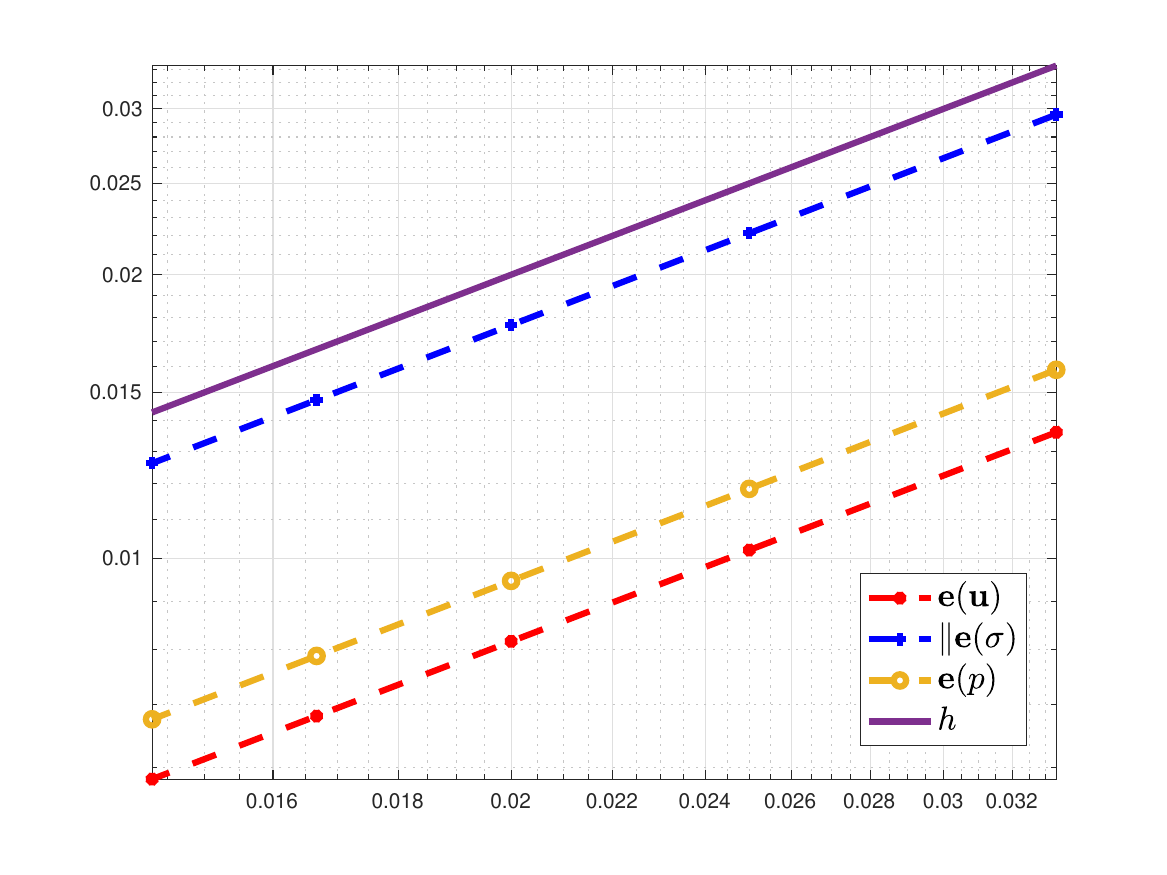}
	\centering\includegraphics[height=4.2cm, width=4.2cm]{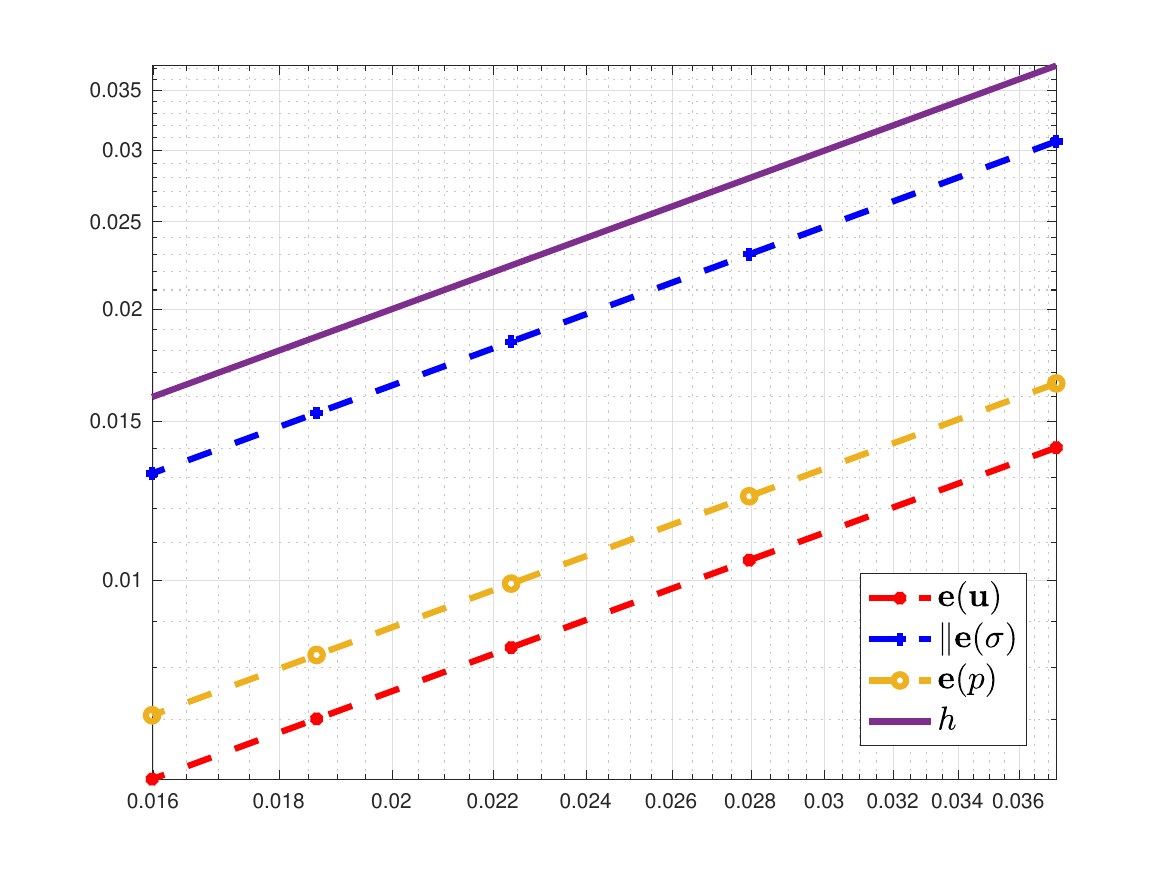}
       	\centering\includegraphics[height=4.2cm, width=4.2cm]{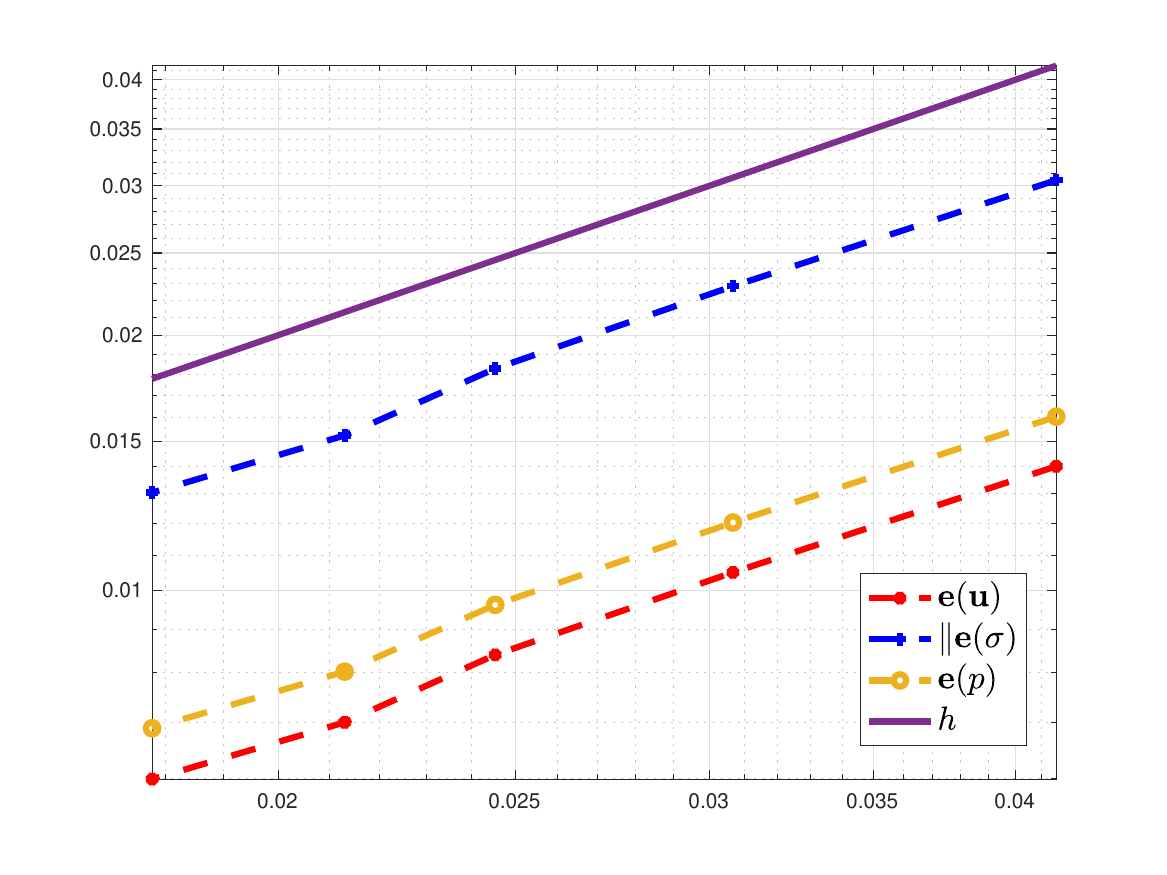}

     				\caption{Test 4: Error curves computed with $\CT_h^2$ (left), $\CT_h^3$ (middle) and $\CT_h^4$ (right). In red, the error of the velocity, in blue the error of the pseudostress, and in yellow the error of the pressure.}
		\label{fig:curve_test4_1}
\end{figure}

We observe from Figure \ref{fig:curve_test4_1} that the order of convergence for each of the unknowns of our problem is $\mathcal{O}(h)$ as is precisely predicted. As we mention before, these results were analogous for other families of polygonal meshes. We also report plots of the computed fields with $\CT_h^3$.
\begin{figure}[H]
	\begin{center}
		\begin{minipage}{13cm}
			\centering\includegraphics[height=5.2cm, width=5.2cm]{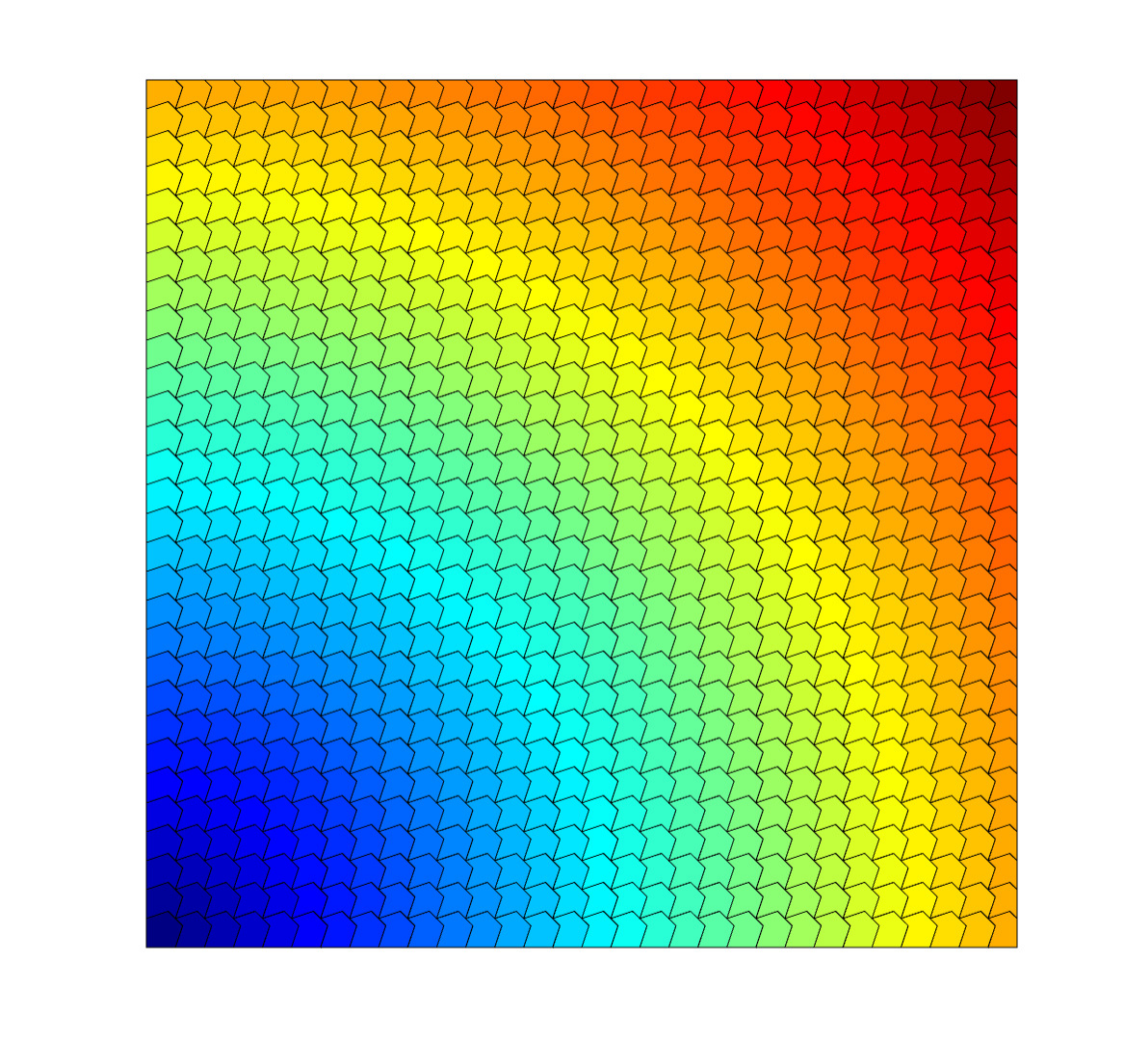}\hspace{1.4cm}
       	\centering\includegraphics[height=5.2cm, width=5.2cm]{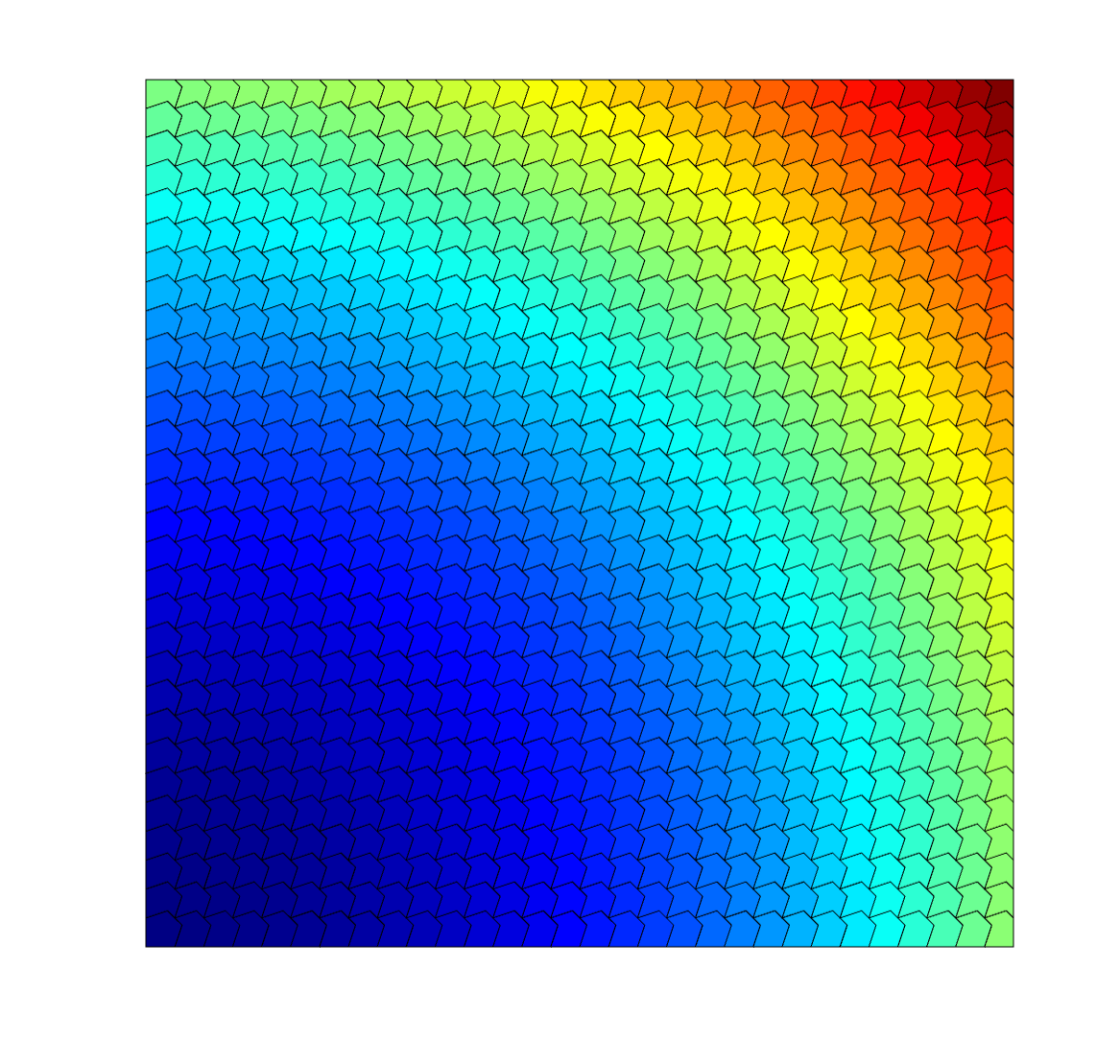}
	
         \end{minipage}
     				\caption{Test 4: Left: plot of the computed velocity magnitude $\bu_h$ with mesh $\CT_h^3$. Right: plot of the  computed pressure fluctuation  $p_h$ with mesh $\CT_h^3$.}
		\label{fig:curve_plots_ncT4up}
	\end{center}
\end{figure}
\begin{figure}[H]
	\begin{center}
		\begin{minipage}{13cm}
			\centering\includegraphics[height=5.2cm, width=5.2cm]{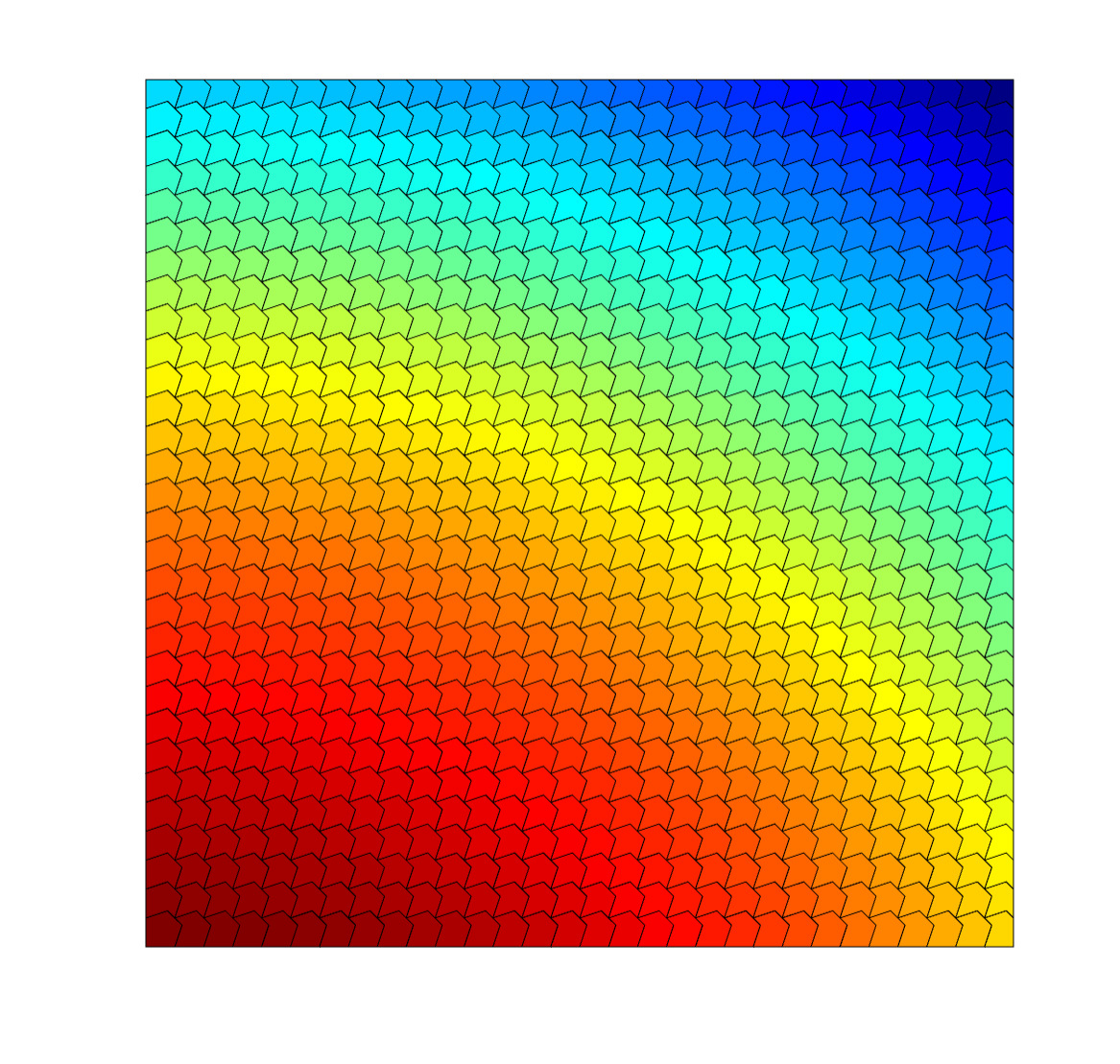}\hspace{1.4cm}
       	\centering\includegraphics[height=5.2cm, width=5.2cm]{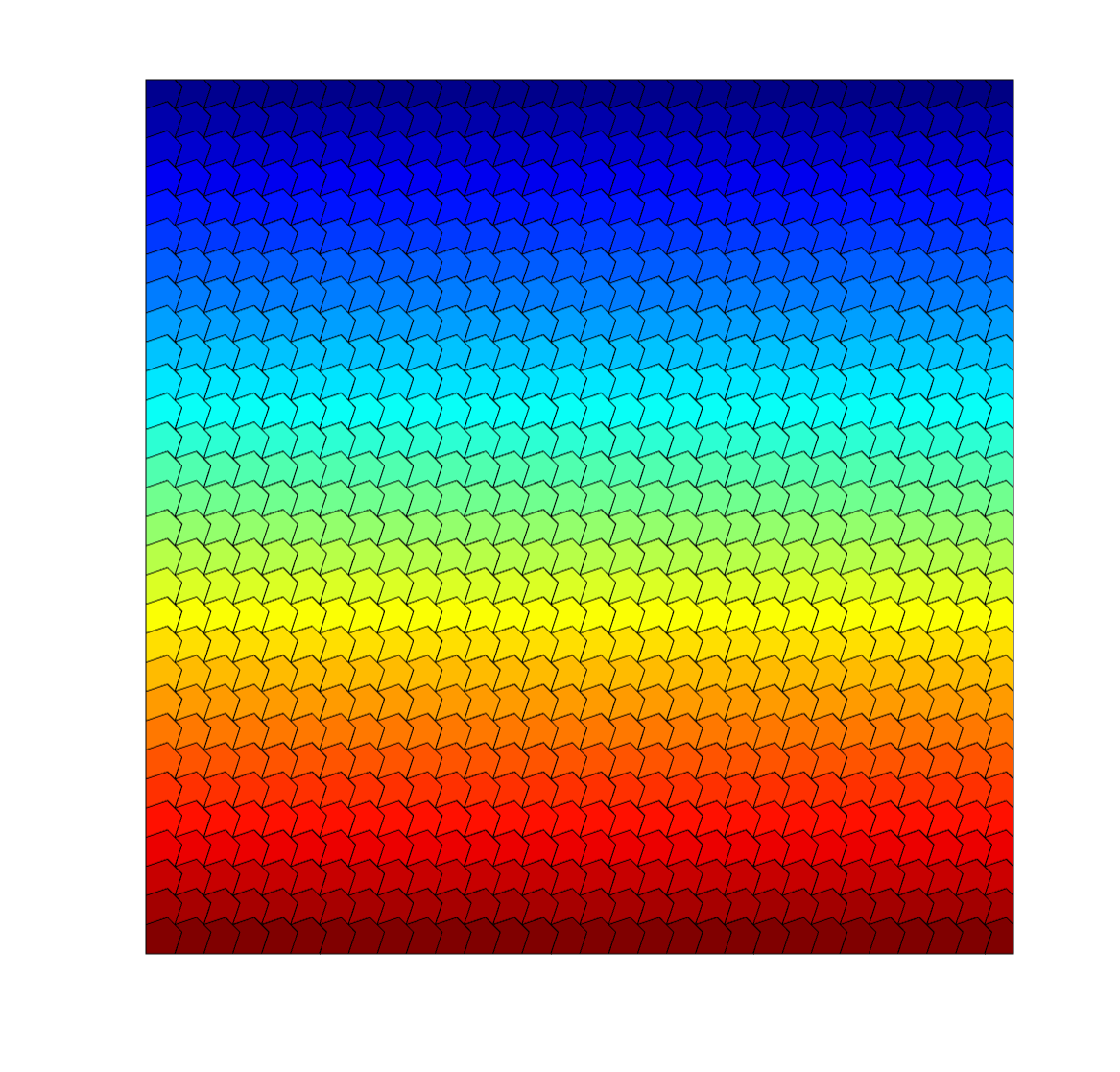}
	\centering\includegraphics[height=5.2cm, width=5.2cm]{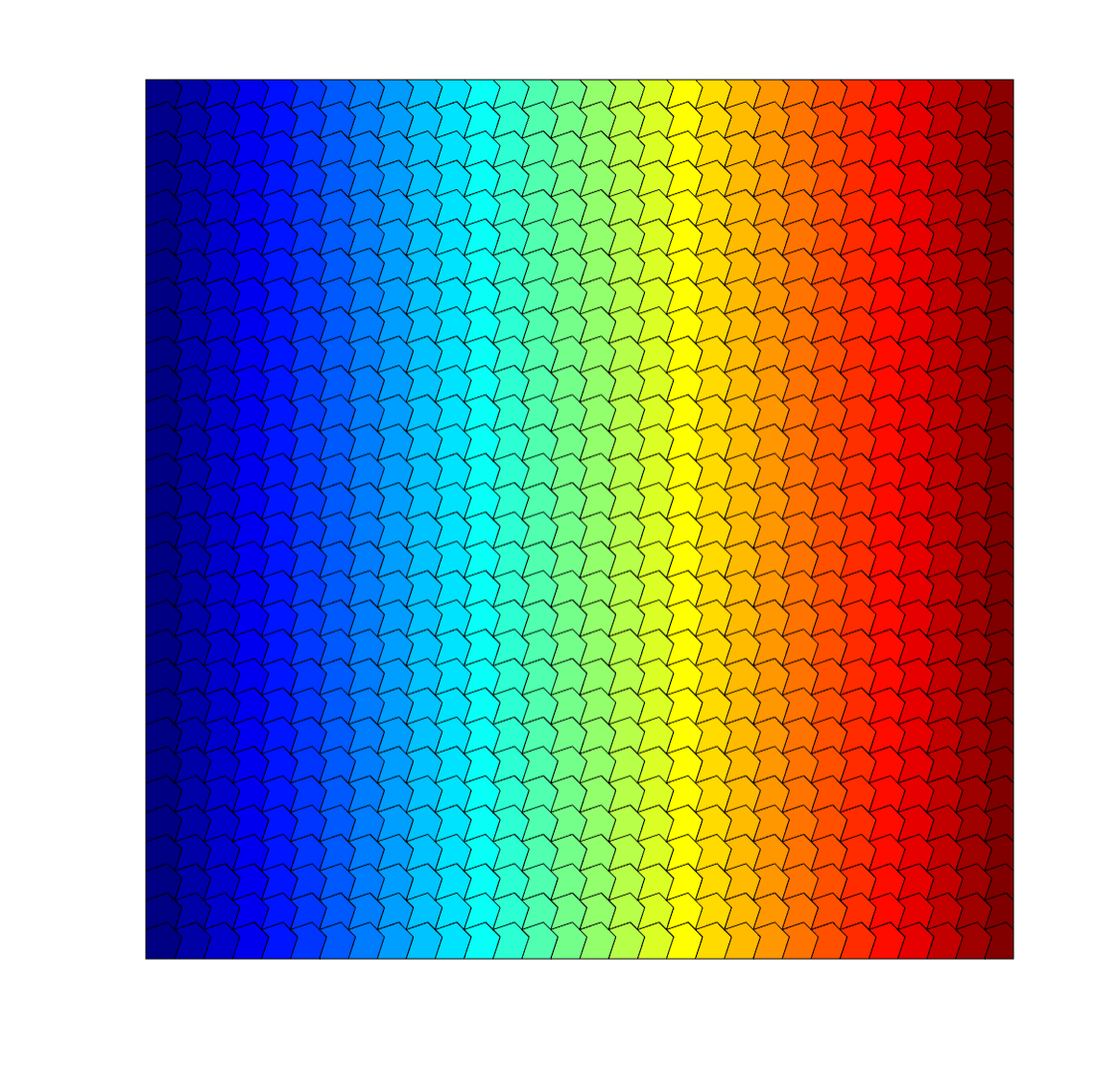}\hspace{1.4cm}
       	\centering\includegraphics[height=5.2cm, width=5.2cm]{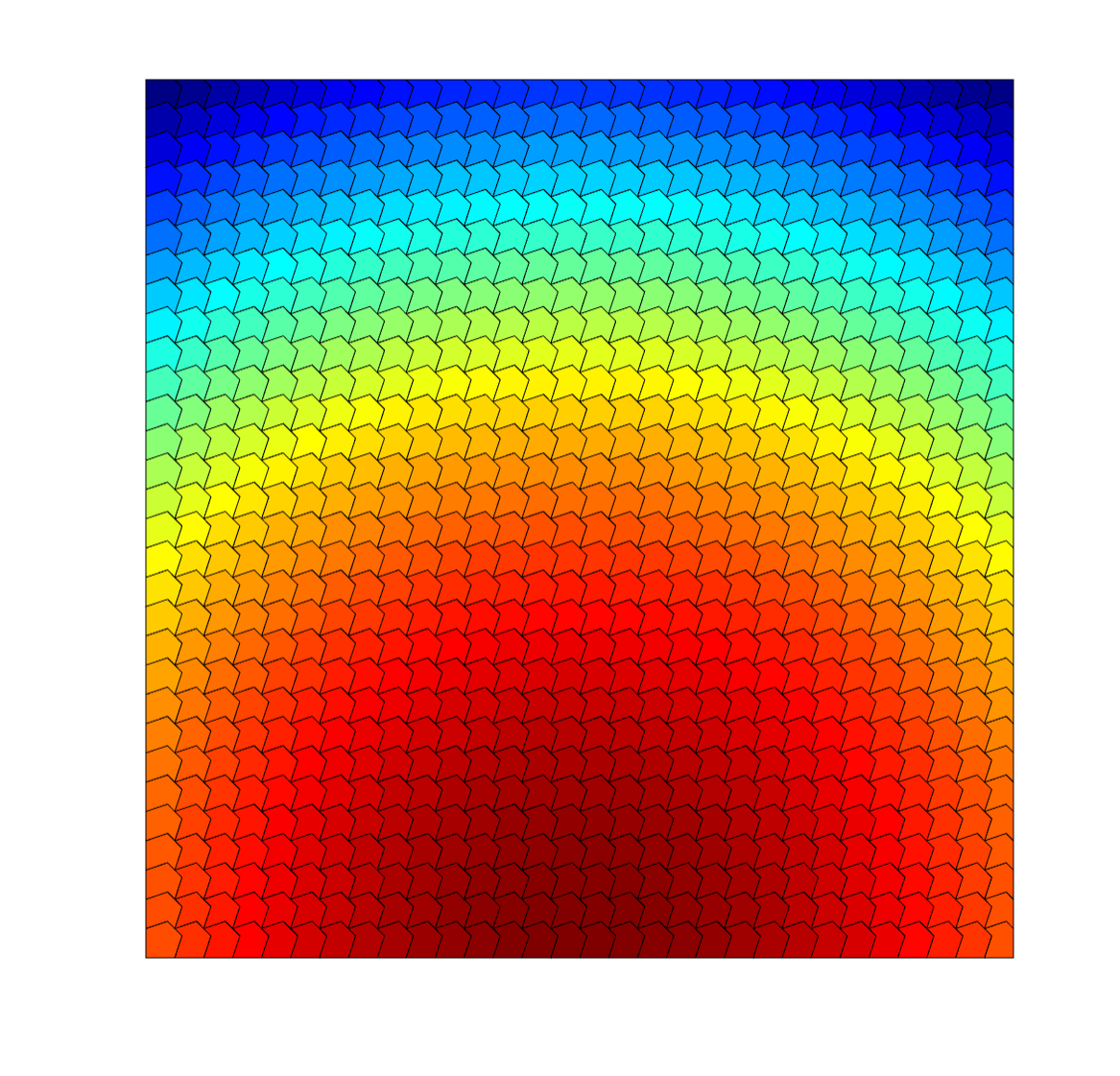}

         \end{minipage}
     				\caption{Test 4: Computed pseudostress tensor components with $\CT_h^3$. Top row,  from left to right: components $\boldsymbol{\sigma}_{h,11}$ and  $\boldsymbol{\sigma}_{h,12}$. Bottom row,  from left to right:  $\boldsymbol{\sigma}_{h,21}$ and $\boldsymbol{\sigma}_{h,22}$.}
		\label{fig:curve_plots_ncT4}
	\end{center}
\end{figure}

Finally from Figures \ref{fig:curve_plots_ncT4up} and \ref{fig:curve_plots_ncT4} we observe that the boundary condition $\boldsymbol{u}(x,y)=\boldsymbol{g}(x,y)=(y,-x)^t$ is captured by our method.
\subsection{Test 5: boundary layer} In this test we consider the domain $\Omega=(0,1)^2$, with $\boldsymbol{\beta}=(1,1)^2$, $\nu=10^{-3}$ and $\kappa=10^{-2}$. The function $\phi(x,y)$ is defined by
$$\phi(x,y)=x^2(1-e^{\lambda(x-1)})^2y^2(1-e^{\lambda(y-1)})^2, $$
where $\lambda=\dfrac{1}{2\nu}.$ The vector field $\boldsymbol{f}$ is chosen so that the exact solution to problem \eqref{eq:Oseen_system} is given by: 
$$\bu(x,y):=\left(\dfrac{\partial \phi}{\partial y},-\dfrac{\partial \phi}{\partial x}\right)\qquad\text{and }\quad p(x,y):=e^{x+y}-(e-1)^2.$$

For this test we consider  the boundary condition $\bu=\boldsymbol{0}$ on $\partial\O$. In Table \ref{TABLA:333} we report the convergence history
of our variables.

\begin{footnotesize}
\begin{table}[H]
\begin{center}
\caption{Test 5: Convergence history of the velocity, pressure, and pseudostress for different polygonal meshes on $\O=(0,1)^2$.}
\begin{tabular}{|c|c||c|c||c|c||c|c||c|}
\hline
\multicolumn{2}{ |c|| }{$\CT_{h}^1$} & \multicolumn{2}{ |c|| }{$\CT_{h}^2$} & \multicolumn{2}{ |c|| }{$\CT_{h}^3$} & \multicolumn{2}{ |c|| }{$\CT_{h}^4$}\\
\hline
 $h$ &  $r(\bu)$ & $h$ &  $r(\bu)$ & $h$ &  $r(\bu)$ & $h$ &  $r(\bu)$  \\
\hline
    
    0.0118 &   --        &  0.0167&  --         &  0.0224  & ---        &0.0208   & --\\
    0.0101 & 1.9664 & 0.0143 &  1.6399 &  0.0186 &1.6272 & 0.0179  &1.7758\\
    0.0088 & 1.9369 & 0.0125 &  1.6100 &  0.0160 &1.6016 & 0.0157  &1.5712\\
    0.0079 & 1.8894 & 0.0111  & 1.5732 &  0.0140 &1.5634  & 0.0139 &1.5722\\
    0.0071 & 1.8334 & 0.0100 &  1.5339 &  0.0124 &1.5198 & 0.0127  &1.8198\\
    0.0064 & 1.7749 & 0.0091 &  1.4945 &  0.0112 &1.4749 & 0.0115  &1.4325\\
    0.0059 & 1.7172 & 0.0083 &  1.4566 &  0.0102 &1.4313 & 0.0106  &1.6301\\
    0.0054 & 1.6623 & 0.0077 &  1.4208 &  0.0093 &1.3903 & 0.0097  &1.2847\\
    0.0051 & 1.6112 & 0.0071 &  1.3876 &  0.0086 &1.3525 & 0.0089  &1.2154\\
    0.0047 & 1.5641 & 0.0067 &  1.3571 &  0.0080 &1.3181 & 0.0084  &1.6683\\
    0.0044 & 1.5212 & 0.0063 &  1.3291 &  0.0075 &1.2871 & 0.0080  &1.4629\\
    0.0042 & 1.4821 & 0.0059 &  1.3037 &  0.0070 &1.2593 & 0.0075  &1.4100\\
    0.0039 & 1.4468 & 0.0056 &  1.2806 &  0.0066 &1.2345 & 0.0073  &1.9284\\
    0.0037 & 1.4148 & 0.0053 &  1.2597 &  0.0062 &1.2124 & 0.0068  &1.0432\\
    0.0035 &1.3859  & 0.0050 & 1.2407  &  0.0059&1.1927  & 0.0064 &0.8960\\
   
 \hline
 $h$ &  $r(\bsig)$ & $h$ &  $r(\bsig)$ & $h$ &  $r(\bsig)$ & $h$ &  $r(\bsig)$  \\
\hline
 0.0118  &              &  0.1189  &               & 0.0224 &               &  0.0208 & \\
 0.0101  &  1.9564 &  0.0896 &  1.5533  & 0.0186 &   1.5063 &  0.0179 &  1.6461\\
 0.0088  &  1.8987 &  0.0700 &  1.5223  & 0.0160 &   1.4747 &  0.0157 &  1.5211\\
 0.0079  &  1.8226 &  0.0562 &  1.4860  & 0.0140 &   1.4342 &  0.0139 &  1.4589\\
 0.0071  &  1.7409 &  0.0462 &  1.4482  & 0.0124 &   1.3915 &  0.0127 &  1.6913\\
 0.0064  &  1.6614 &  0.0387 &  1.4113   & 0.0112 &  1.3500  &  0.0115 & 1.3693\\
 0.0059  &  1.5881 &  0.0330 &  1.3763  & 0.0102 &   1.3113  &  0.0106 & 1.4902\\
 0.0054  &  1.5228 &  0.0285 &  1.3440  & 0.0093 &   1.2763 &  0.0097 &  1.2319\\
 0.0051  &  1.4654 &  0.0248 &  1.3145  & 0.0086 &   1.2451 &  0.0089 &  1.1706\\
 0.0047  &  1.4158 &  0.0219 &  1.2877  & 0.0080 &   1.2176 &  0.0084 &  1.5713\\
 0.0044  &  1.3729 &  0.0195 &  1.2636  & 0.0075 &   1.1934 &  0.0080 &  1.3918\\
 0.0042  &  1.3359 &  0.0175 &  1.2419  & 0.0070 &   1.1723 &  0.0075 &  1.3171\\
 0.0039  &  1.3040 &  0.0158 &  1.2224  & 0.0066 &   1.1537 &  0.0073 &  1.8852\\
 0.0037  &  1.2764 &  0.0144 &  1.2048  & 0.0062 &   1.1375 &  0.0068 &  0.9766\\
 0.0035  &  1.2525&   0.0131 &  1.1891 & 0.0059  &  1.1232  & 0.0064  & 0.8538\\
 \hline
 $h$ &  $r(p)$ & $h$ &  $r(p)$ & $h$ &  $r(p)$ & $h$ &  $r(p)$  \\
\hline
0.0118 &                  &   0.1189  &               &   0.0224  &               &   0.0208&  \\
0.0101 &    2.1864   &   0.0896 &   1.8361 &   0.0186  &   1.8325 &   0.0179&    1.8526\\
0.0088 &     2.2064  &   0.0700 &   1.8501 &   0.0160  &   1.8535 &   0.0157&    1.7500\\
0.0079 &     2.2039  &   0.0562 &   1.8571 &   0.0140  &   1.8621 &   0.0139&    1.8601\\
0.0071 &     2.1884  &   0.0462 &   1.8588 &   0.0124  &   1.8632 &   0.0127&    1.9242\\
0.0064 &     2.1656  &   0.0387 &   1.8564 &   0.0112   &  1.8588  &   0.0115&   1.8034\\
0.0059 &     2.1390  &   0.0330 &   1.8508 &   0.0102  &   1.8505 &   0.0106&    2.0853\\
0.0054 &     2.1106  &   0.0285  &  1.8426  &   0.0093  &  1.8391  &   0.0097&   1.7350\\
0.0051 &     2.0816  &   0.0248 &   1.8324 &   0.0086  &   1.8255 &   0.0089&    1.5379\\
0.0047 &     2.0527  &   0.0219 &   1.8206 &   0.0080  &   1.8101 &   0.0084&    2.2742\\
0.0044 &     2.0243  &   0.0195 &   1.8075 &   0.0075  &   1.7933 &   0.0080&    2.1110\\
0.0042 &     1.9965  &   0.0175 &   1.7934 &   0.0070  &   1.7754 &   0.0075&    1.8566\\
0.0039 &     1.9696  &   0.0158 &   1.7784 &   0.0066  &   1.7569 &   0.0073&    2.7260\\
0.0037 &     1.9434  &   0.0144 &   1.7628 &   0.0062  &   1.7377 &   0.0068&    1.4209\\
0.0035 &     1.9180  &   0.0131 &   1.7467 &   0.0059  &   1.7183 &   0.0064&    1.1649\\
 \hline
\end{tabular}
\label{TABLA:333}
\end{center}
\end{table}
\end{footnotesize}
We observe from Table \ref{TABLA:333} that  the order of convergence for the velocity and pseudostress, as $h\rightarrow 0$, is $\mathcal{O}(h)$ as the theory predicts. On the other hand, for the pressure we observe an improved order of convergence $\mathcal{O}(h^2)$ compared with the order theoretically proved in Corollary \ref{cor:error_prressure}. This improvement was attained for all the meshes of this test. The analysis of this behavior is further investigation. 

We end this test with plots of the velocity magnitude, pressure fluctuation, and the components of the pseudostress. We observe the velocity field $\bu$ develops a boundary layer along the lines $x=1$ and $y=1$ as is expectable for this experiment.
\begin{figure}[H]
	\begin{center}
		\begin{minipage}{13cm}
			\centering\includegraphics[height=4.9cm, width=4.9cm]{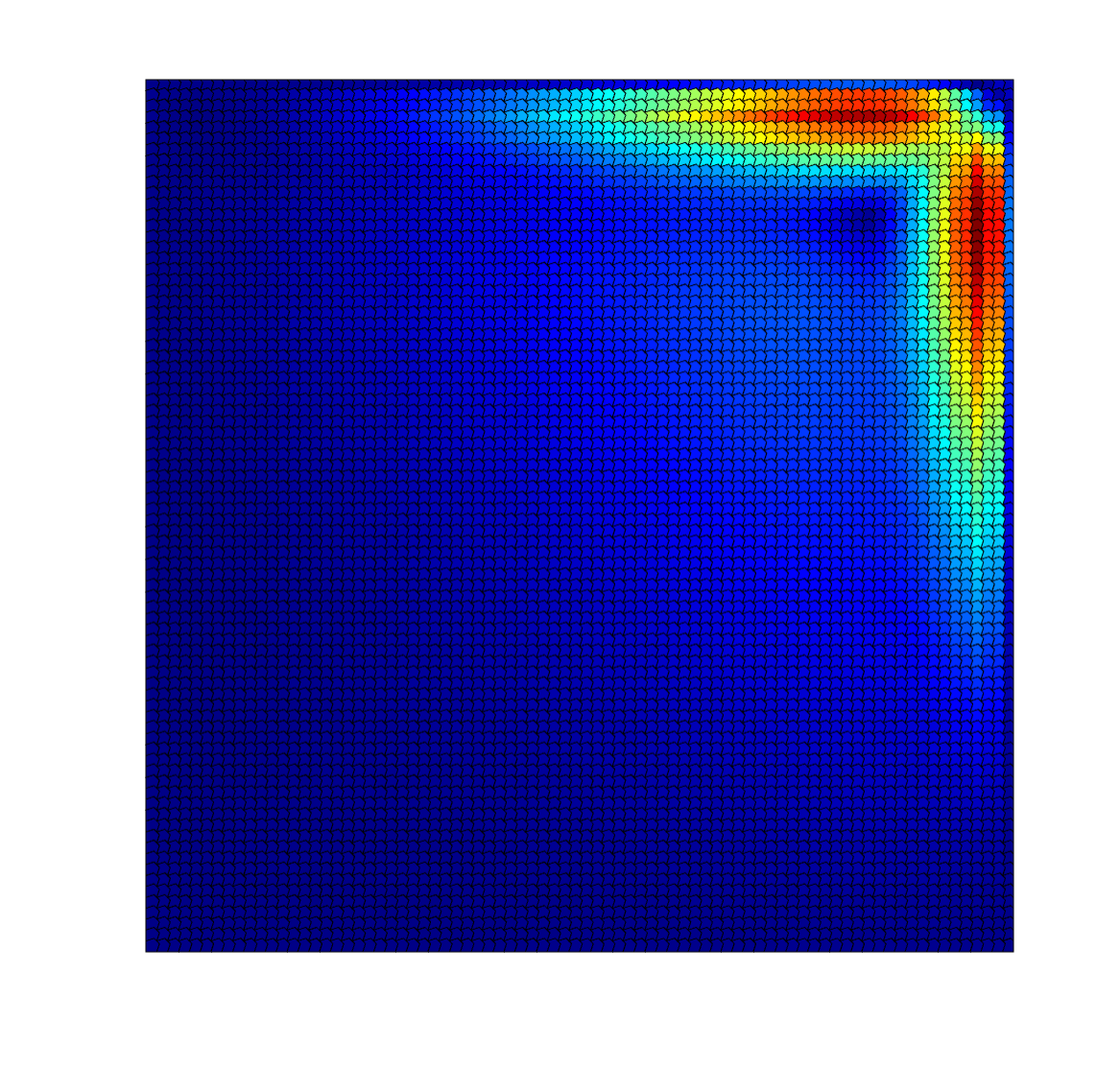}\hspace{1.4cm}
			\centering\includegraphics[height=4.9cm, width=4.9cm]{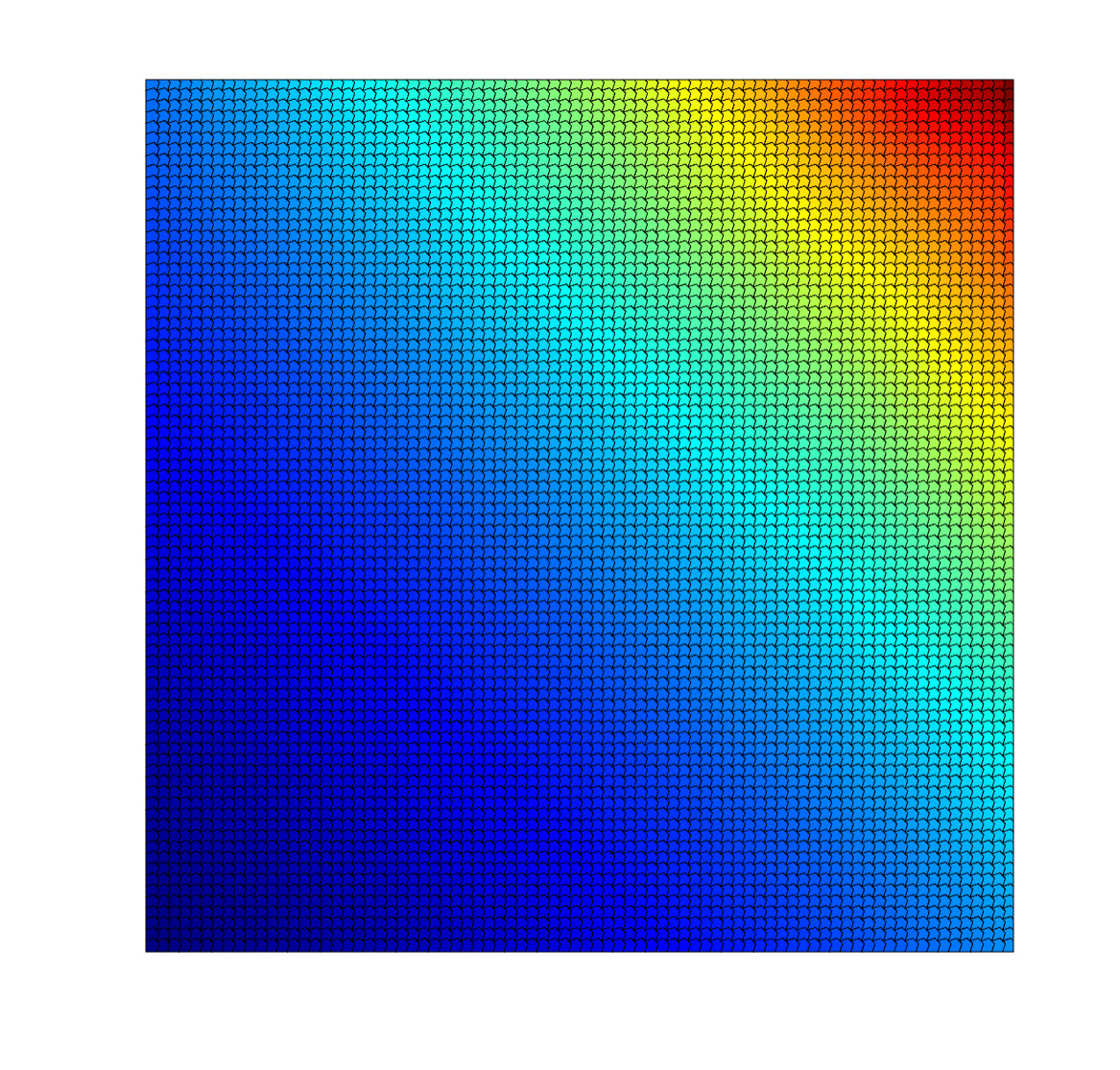}
			         \end{minipage}
     				\caption{Test 5: Left: plot of the computed velocity magnitude $\bu_h$ with mesh $\CT_h^3$. Right: plot of the  computed pressure fluctuation  $p_h$ with mesh $\CT_h^3$.}
		\label{fig:test_5curve_plots_square}
	\end{center}
\end{figure}
\begin{figure}[H]
	\begin{center}
	\vspace{-0.48cm}
		\begin{minipage}{12cm}
			\centering\includegraphics[height=4.9cm, width=4.9cm]{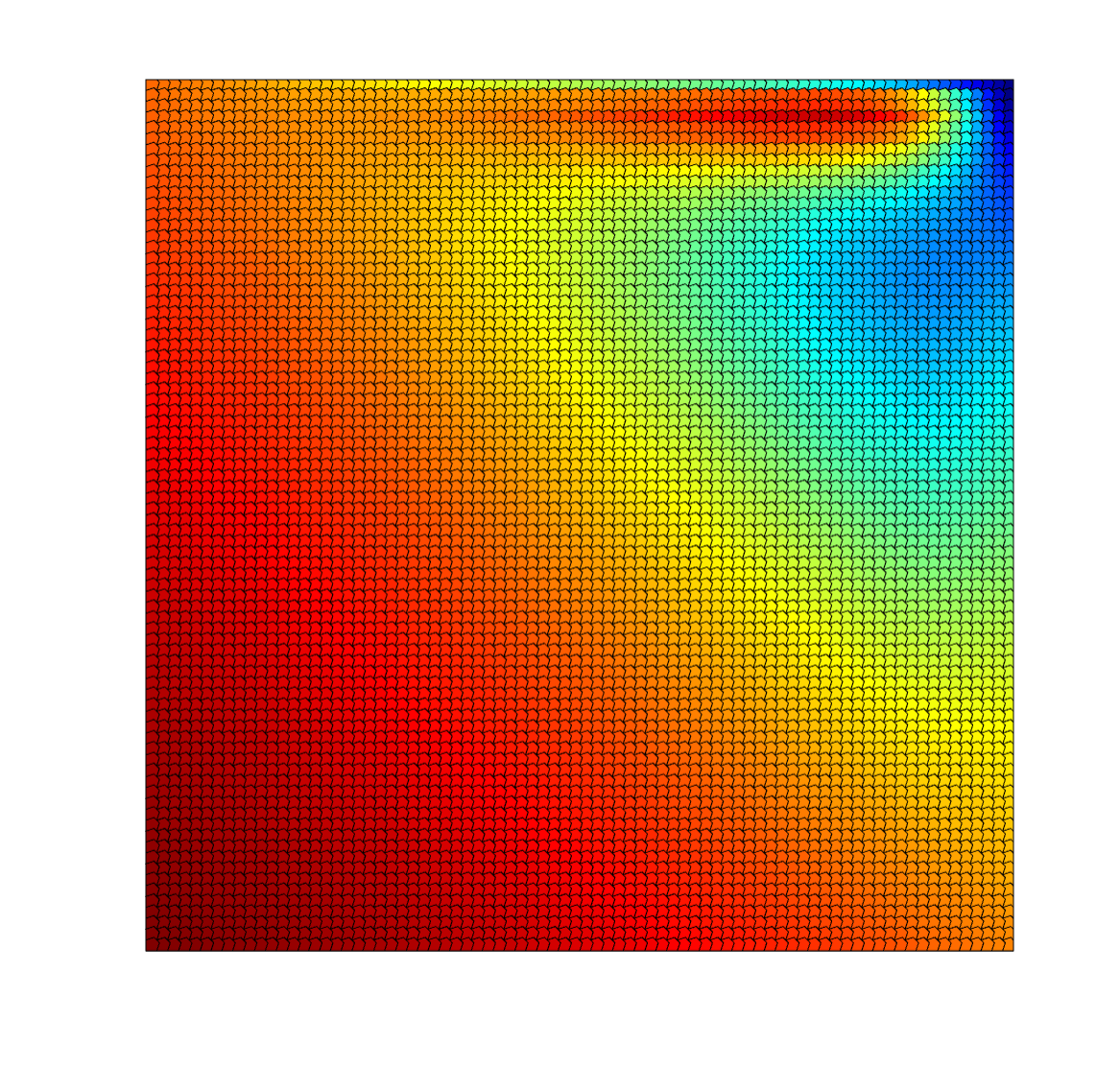}\hspace{1.4cm}
       	\centering\includegraphics[height=4.9cm, width=4.9cm]{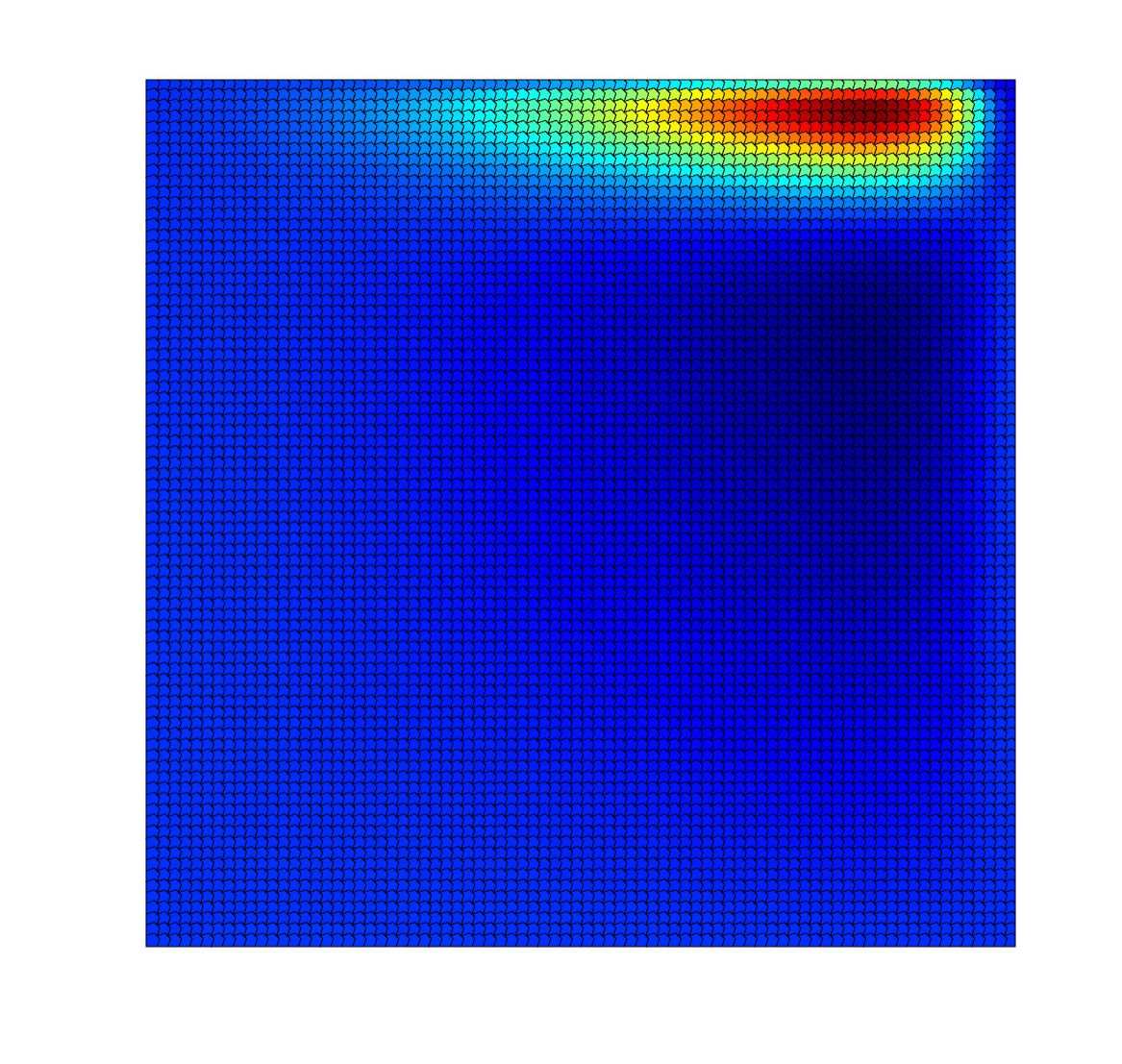}
	\centering\includegraphics[height=4.9cm, width=4.9cm]{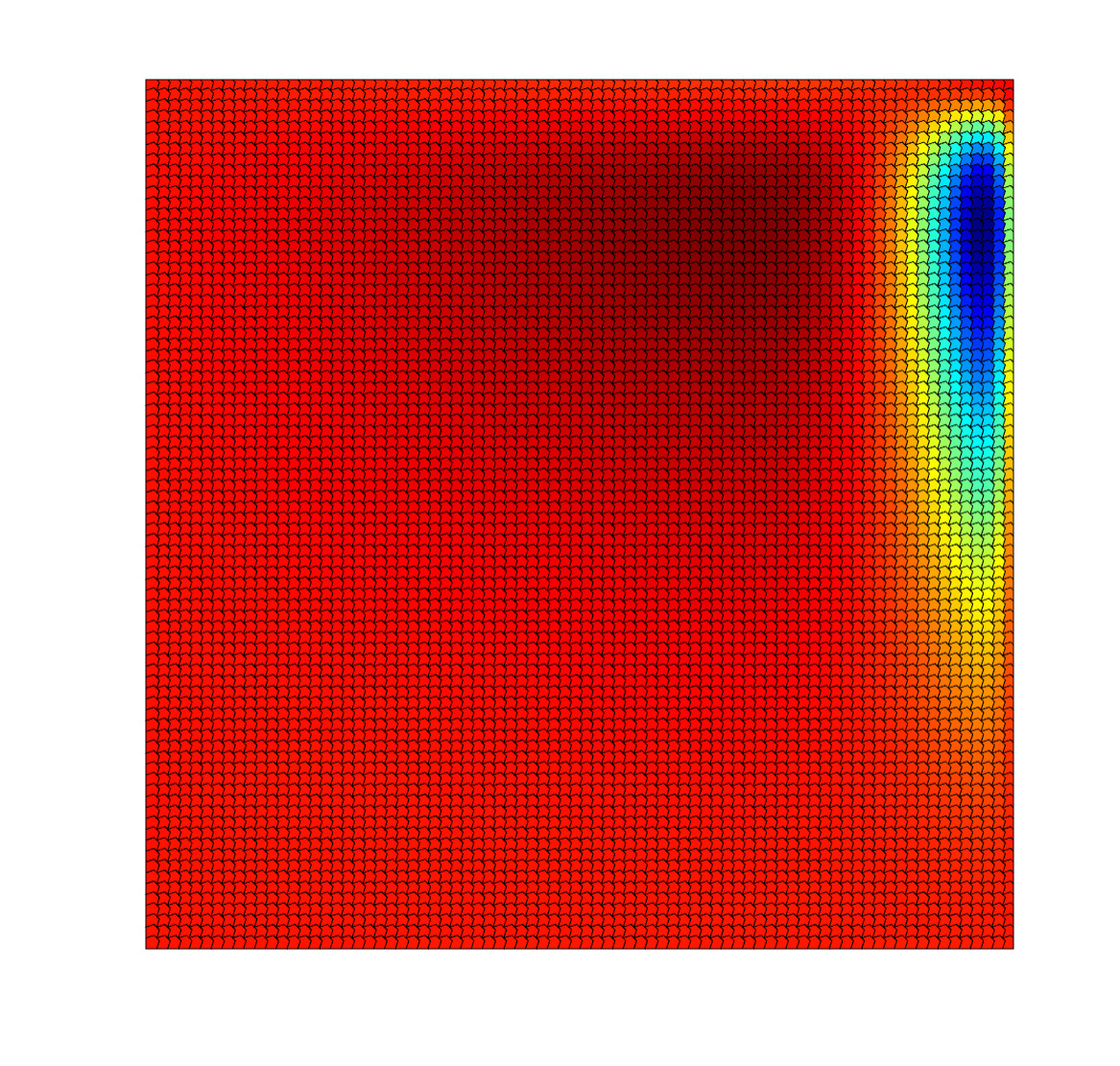}\hspace{1.4cm}
       	\centering\includegraphics[height=4.9cm, width=4.9cm]{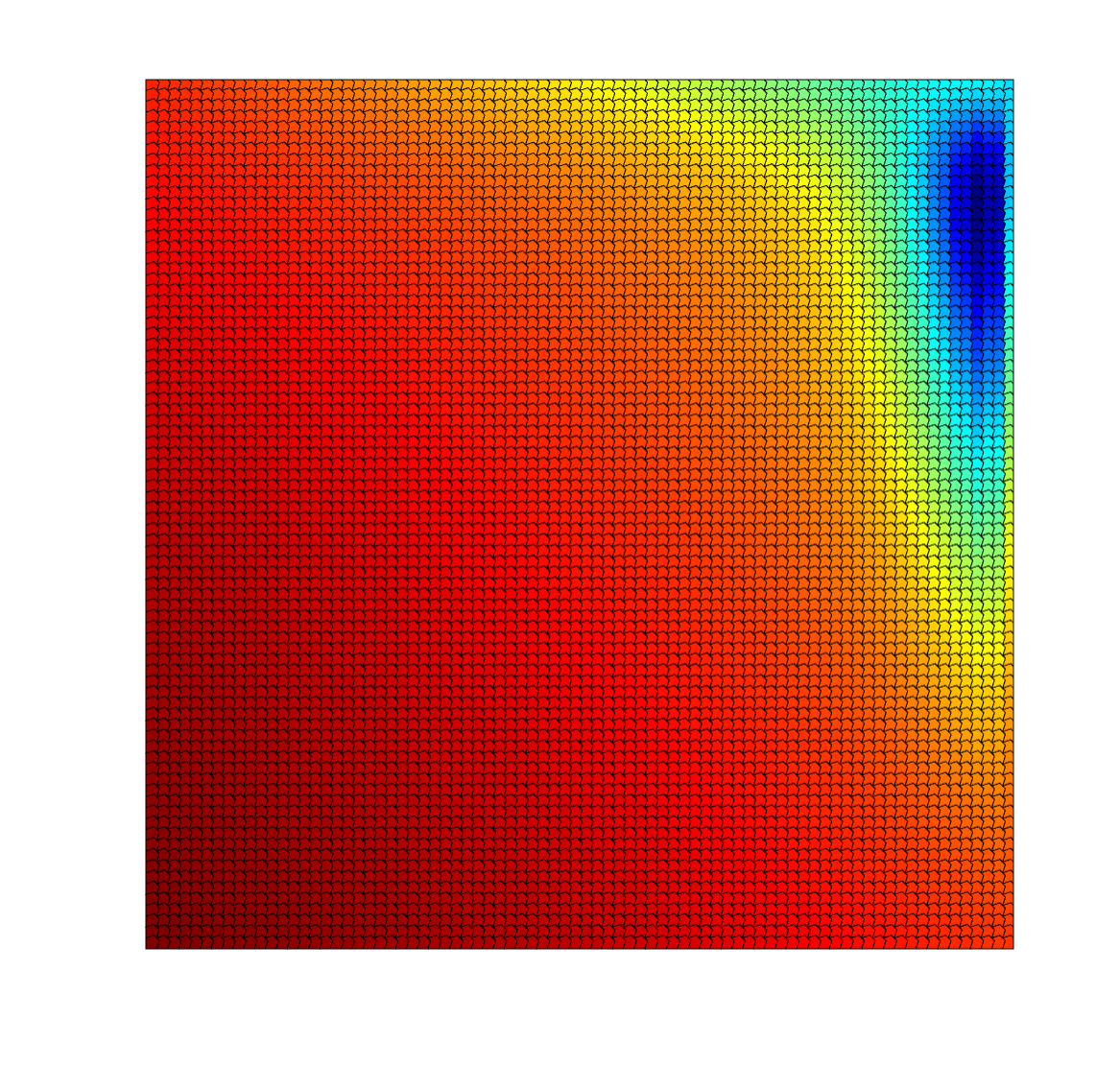}

         \end{minipage}
     				\caption{Test 5: Computed pseudostress tensor components with $\CT_h^3$. Top row,  from left to right: components $\boldsymbol{\sigma}_{h,11}$ and  $\boldsymbol{\sigma}_{h,12}$. Bottom row,  from left to right:  $\boldsymbol{\sigma}_{h,21}$ and $\boldsymbol{\sigma}_{h,22}$.}
		\label{fig:test_5curve_plots_square}
	\end{center}
\end{figure}
\section{Conclusions}
We have solved the generalized Oseen system using a mixed virtual element scheme, where we have introduced the pseudostress as an additional unknown, allowing us to neglect the pressure in the formulation. We have proved existence and uniqueness of the solution at the continuous level using a fixed-point strategy, which demands smallness assumptions on the data. The discrete version, based on the virtual element method, is stable, and existence and uniqueness of the solution have been proved, together with a priori error estimates for the velocity, pseudostress, and pressure. We have tested the scheme in different scenarios, where the optimal order of convergence is attained. However, for a boundary layer example, we have observed an additional order of convergence for the pressure, which leads us to further investigate the Oseen equations and the virtual element method. Also, an a posteriori error analysis of this problem is under investigation.



\end{document}